\documentclass[11pt,reqno]{amsart}
\usepackage[dvipsnames,svgnames,x11names]{xcolor}
\usepackage[markup=underlined]{changes}
\usepackage{fullpage}
\usepackage{placeins}
\usepackage{tikz}
\usepackage{chngcntr}
\usepackage{subcaption}
\newcommand\numberthis{\addtocounter{equation}{1}\tag{\theequation}}
\usepackage{multicol}
\usepackage{amsfonts}
\usepackage{amsthm,amsmath}
\usepackage{amsbsy}
\usepackage{bm}
\usepackage{amssymb} % Without this \qed does not get filled square
\usepackage{verbatim}
\usepackage{wrapfig}

\usepackage{times}
\usepackage{pgf,tikz}
\usepackage{graphicx}
\usetikzlibrary{trees,shadows}
\usetikzlibrary{arrows}
\usepackage{units}
\usepackage{enumerate,enumitem}
\usepackage{bbm}
\usepackage{exscale}
\usepackage[hyperindex,breaklinks]{hyperref}
\hypersetup{ colorlinks=true, linkcolor=blue, citecolor=blue,
  filecolor=blue, urlcolor=blue}
  \numberwithin{equation}{section} 
\DeclareMathOperator{\one}{\mathbbm{1}} %%%% indicator function

\usepackage[comma, sort&compress]{natbib}

\numberwithin{equation}{section}
\theoremstyle{plain} %documentation says there are only three styles
    \newtheorem{theorem}{Theorem}[subsection]
    
    \newtheorem{lemma}[theorem]{Lemma}
    \newtheorem{proposition}[theorem]{Proposition}
    \newtheorem{corollary}[theorem]{Corollary}
    \newtheorem{claim}[theorem]{Claim}

    \newtheorem{definition}[theorem]{Definition}
    \newtheorem{fact}[theorem]{Fact}
    
    \newtheorem{remark}[theorem]{Remark}
    \newtheorem{example}[theorem]{Example}

\DeclareMathOperator{\Ba}{\mathcal{B}}

\DeclareMathOperator{\ESD}{ESD}

\usepackage{mathtools}
\newcommand{\Khorunzkyintegral}{\sqrt{u}\int_0^{\infty}\frac{\J(2\sqrt{uv})}{\sqrt{v}}}

\newcommand {\ota}{\dot{\iota}}

\DeclareMathOperator{\J}{J_1}
\DeclareMathOperator{\Tr}{Tr}

\DeclareMathOperator{\tr}{tr}

\DeclareMathOperator{\dd}{d}
\DeclareMathOperator{\Var}{Var}
\DeclareMathOperator{\St}{S}
\DeclareMathOperator{\R}{R}

\newcommand{\E}{\mbox{${\mathbb E}$}}

\newcommand{\bigO}{\mbox{${\mathrm O}$}}
\newcommand{\lito}{\mbox{${\mathrm o}$}}

\newcommand{\vep}{\varepsilon}

\newcommand {\prob}{\mathbb{P}}
\usepackage{xcolor}

\newcommand{\e}{\mathrm{e}}
\newcommand{\M}{\mathbf{M}}
\newcommand{\C}{\mathbb{C}}

\newcommand{\bA}{\mathbf{A}}

\definecolor{xdxdff}{rgb}{0.49019607843137253,0.49019607843137253,1}
\definecolor{ffffff}{rgb}{1,1,1}
\definecolor{ududff}{rgb}{0.30196078431372547,0.30196078431372547,1}
\definecolor{zzttqq}{rgb}{0.6,0.2,0}

\begin{document}
\title[Inhomogeneous random graphs]{Limiting Spectra of inhomogeneous random graphs}

%    author one information
\author[L.~Avena]{Luca Avena} 
\address{{ Leiden University, Mathematical Institute, Niels Bohrweg 1 2333 CA, Leiden. The Netherlands.} \linebreak
{DIMAI, Dipartimento di Matematica e Informatica ``Ulisse Dini", Università degli studi di Firenze, Viale G.B. Morgagni 67A, 50134, Florence, Italy.}}
\email{luca.avena@unifi.it}
%%    author two  information
%%    author three  information
\author[R.~S.~Hazra]{Rajat Subhra Hazra}
\address{University of Leiden, Niels Bohrweg 1, 2333 CA, Leiden, The Netherlands}
\email{r.s.hazra@math.leidenuniv.nl}
%%    author four  information
\author[N. Malhotra]{Nandan Malhotra}
\address{University of Leiden, Niels Bohrweg 1, 2333 CA, Leiden, The Netherlands}
\email{n.malhotra@math.leidenuniv.nl}

\date{\today}

\begin{abstract}
%\section{Abstract}
We consider sparse inhomogeneous Erd\H{o}s-R\'enyi random graph ensembles where edges are connected independently with probability $p_{ij}$. We assume that $p_{ij}= \varepsilon_N f(w_i, w_j)$ where $(w_i)_{i\ge 1}$ is a sequence of deterministic weights, $f$ is a bounded function and $N\varepsilon_N\to \lambda\in (0,\infty)$. We characterise the limiting moments in terms of graph homomorphisms and also classify the contributing partitions. We present an analytic way to determine the Stieltjes transform of the limiting measure. The convergence of the empirical distribution function follows from the theory of local weak convergence in many examples but we do not rely on this theory and exploit combinatorial and analytic techniques to derive some interesting properties of the limit. We extend the methods of \cite{KSV2004} and show that a fixed point equation determines the limiting measure. The limiting measure crucially depends on $\lambda$ and it is known that in the homogeneous case, if $\lambda\to\infty$, the measure converges weakly to the semicircular law (\cite{Jung-Lee}). We extend this result of interpolating between the sparse and dense regimes to the inhomogeneous setting and show that as $\lambda\to \infty$, the measure converges weakly to a measure which is known as the operator-valued semicircular law. 

\end{abstract}
\keywords{generalized random graphs, inhomogeneous random graph, empirical spectral distribution, spectrum, free probability}
\subjclass[2000]{05C80, 60B20, 60B10, 46L54}
\maketitle

\tableofcontents 

\section{\bf Introduction}\label{intro}
Homogeneous Erd\H{o}s-R\'enyi Random Graphs (ERRG) serve as the basis for many mathematical theories in random graphs. Real-world networks are highly inhomogeneous and have a far more complex structure. Various attempts have been made to generalize this to other kinds of random graph models. One of the successful extensions is the inhomogeneous Erd\H{o}s-R\'enyi random graph model introduced by \cite{bollobas2007phase}. This graph has $N$ vertices labeled by $[N]={1,..., N}$, and edges are present independently with probability $p_{ij}$ given by $p_{ij} = \frac{f(x_i, x_j)}{N}\wedge 1$, where $f$ is a nice symmetric kernel on a state space $S\times S$, and $x_i$ are certain attributes associated with vertex $i$ belonging to $S$.  If $f$ is bounded, the graph is a sparse random graph. To introduce the non-sparse regime, in this article, we consider a small variant of the above inhomogeneous random graph. The vertex set remains the same, but the connection probabilities are given by
\begin{equation}\label{eq:formp}
p_{ij} = \varepsilon_N f(w_i, w_j) \wedge 1,
\end{equation}
where $\varepsilon_N$ is a tuning parameter, $(w_i)$ is a sequence of deterministic weights, and $f$ is a symmetric, bounded function on $[0,\infty)^2$. The weights can signify a property of vertex $i$. They can also be generally random, but we do not consider this case. Note that when $N\varepsilon_N \to \infty$, the average degree is unbounded, and when $N\varepsilon_N = \mathrm{O}(1)$, the average degree is bounded. We call the former case \emph{dense} and the latter case \emph{sparse}. In the sparse case, the properties of the connected components were studied in \cite{bollobas2007phase}. They studied the properties of the connected components and their relationship with the branching process. It was shown that the largest component of the graph has a size of order $N$ if the operator norm of the kernel operator corresponding to $f$ is strictly greater than $1$ (see also \cite[Theorem 3.9]{remcovol2}). In the subcritical case, the sizes of the largest connected components can exhibit different behaviour compared to the ERRG. The study of the largest connected components in various inhomogeneous random graphs has attracted a lot of attention (see, for example, \cite{bhamidi2010scaling,broutin2021limits, bet2023first, van2013critical, devroye2014connectivity}). In this article, we are interested in the empirical distribution of the eigenvalues of the adjacency matrix of the graph and how the transition occurs from the sparse to the dense case in terms of the limiting spectral distribution. There hasn't been much literature in this area, even though various specific graphs have been studied. For example, the largest eigenvalue of the sparse Chung-Lu random graph was studied in \cite{chung2003eigenvalues}, and this was extended to an inhomogeneous setting by \cite{benaych2020spectral, benaych2019largest}. The bulk of the spectrum of sparse graphs is mainly studied through local weak convergence. Here, we present a unifying approach to understanding both the sparse and the dense cases, allowing us to interpolate between the two regimes.

In the case of homogeneous ERRG, it is known that in the dense case, the empirical distribution converges to the semicircle law after an appropriate scaling (\cite{tran2013sparse}). In the sparse case, it converges to a measure that depends on the parameter $ N\varepsilon_N \to \lambda$. The behaviour is much more complicated in the sparse case. Various interesting properties were predicted by \cite{BG2001}. The existence of the limiting distribution was proved by \cite{KSV2004}, who also showed some interesting properties of the moments and the limiting Stieltjes transform. The local geometric behaviour of sparse random graphs can be well studied using the theory of local weak convergence (LWC), which builds on the works \cite{Aldous:Lyons} and \cite{Benjamini:Schramm}. It roughly describes how a graph looks like in the limit around a uniformly chosen vertex. For a detailed review of LWC and various other applications, see \cite{remcovol2}. In a remarkable work by \cite{bordenave:lelarge:2010}, it was proved that if a graph with $N$ vertices converges locally weakly to a Galton-Watson tree, then the Stieltjes transform of the empirical spectral distribution converges in $L^1$ to the Stieltjes transform of the spectral measure of the tree, and it satisfies a recursive distributional equation. The example of homogeneous ERRG was treated in \cite[Example 2]{bordenave:lelarge:2010}. The limiting measure of sparse ERRG depends on $\lambda$ and is still very non-explicit. It was proved by \cite{bordenave:virag:sen,arras2021existence} that the measure has an absolutely continuous component if and only if $\lambda>1$. The size of the atom at the origin was shown by \cite{bordenave:lelarge:salez}, and the nature of the atomic part of the measure was studied in the same article. The study of so-called extended states at origin was initiated in \cite{coste:salez}, and it was shown that for $\lambda < e$, there were no extended states, and for $\lambda > e$, it has extended states. All these results were conjectured in \cite{BG2001}. Most of these results on local limits show that properties are generally true for unimodular Galton Watson trees.  

In the simulations of \cite{BG2001}, it is clear that when $\lambda$ is slightly larger than 1, the limiting measure already starts taking the shape of the semicircular law. It was shown in \cite{Jung-Lee} that indeed, if $\lambda\to\infty$, then the limiting measure converges to the semicircular law. The main motivation of this work comes from the work of \cite[Theorem 1]{Jung-Lee}, and we extend the results from ERRG to inhomogeneous models. We explicitly derive the moments of the limiting measure for the inhomogeneous setting, extending the works of \cite{KSV2004}, albeit with a much simpler proof. The moments of the limiting measure depend on certain kinds of graph homomorphism counts, which also appeared in the works of \cite{zhu2020graphon}. Although the theory of local weak convergence is very useful, we do not know if it can be used to derive the moments of the limiting measure. We also study the Stieltjes transform of the limiting measure, following the idea of \cite{KSV2004}, and attempt an expansion of it for $\lambda$ large enough. This has also gained attention in the physics literature, see references in \cite{physics:2023}. 
We show that when $\lambda \gg 1$, the limiting moments closely resemble those of the inhomogeneous ERRG, as derived in \cite{chakrabarty:hazra:denhollander:sfragara} and also implied by the work of \cite{zhu2020graphon}. In \cite{chakrabarty:hazra:denhollander:sfragara}, they considered inhomogeneous ERRG to have weights $w_i= i/N$, and $N\vep_N\to \infty$. This result can be extended to general deterministic weights without significant effort, and we state this general result in Section \ref{pap1-section 2}. The limiting measure is well-known in the free probability literature and appears as a universal object in many inhomogeneous systems, referred to as the \emph{operator-valued semicircle law} \cite[Theorem 22.7.2]{speicher}. The Stieltjes transform satisfies a recursive analytic equation. We derive the Stieltjes transform in the sparse setting using a fixed-point equation. The fixed point is simpler in the case of homogeneous ERRG, but in the inhomogeneous case, it becomes more complex. We explicitly characterise  this fixed-point equation. We believe that in the future, this will aid in determining the rate of convergence of the empirical spectral distribution, which can be precisely quantified in terms of $\lambda$ and $N$. The rates of convergence in the free central limit theorem were recently explored in \cite{Banna:Mai}, but these results are not directly applicable to our setting. We leave this as an open problem. Obtaining an explicit rate of convergence will provide an exact explanation of why the limiting measure in the sparse setting is very close to the non-sparse setting for relatively small $\lambda>1$.

{\bf Brief summary of the results.} The two main results of this work aim to characterise the limiting spectral measure of inhomogeneous Erd\H{o}s-R\'enyi random graphs. Our first result, Theorem \ref{pap1-momenttheorem}, gives a characterisation of the moments of this measure, where the $k^{\text{th}}$ moment for any $k\geq 0$ is described in terms of homomorphism densities of the inhomogeneity function $f$ and special classes of partitions of the tuple $[k]$. We can recover the moments of the dense regime asymptotically (as $\lambda\to \infty$) using this result. The second result Theorem \ref{pap1-Resolventtheorem} provides an analytic characterisation of the measure. In particular, we provide an analytic characterisation of a functional of the resolvent of the adjacency matrix in terms of a fixed point equation. As a consequence, in Corollaries \ref{pap1-Stieltjescorollary} and \ref{pap1-Stieltjeslargelambda}, we obtain the Stieltjes transform of the sparse and dense limiting measures. The form of the limiting Stieltjes transform can be seen as an alternative description of the form obtained through local weak convergence (whenever it applies).  

{\bf Outline}. We begin Section \ref{pap1-section 2} by describing the model and stating the results of the dense regime. We state the assumptions on the sparse setting more explicitly and proceed by stating our main results for this setting. We then describe a relationship with local weak convergence and also give some examples of popular random graph models. We show that the sparse Chung-Lu type model falls into our setting, and while the Norros-Riettu model and the Generalised Random Graph models do not directly fall into our setting, we show that asymptotically the three models have the same spectral distribution, which has a free-multiplicative part that can be seen from our main results. 

In Section \ref{pap1-Momentsection}, we prove our first main result which takes a combinatorial approach, and we set up all the necessary tools used in proving the result. We identify the moments of the limiting spectral measure in terms of partitions of a tuple and graph-homomorphism densities. We provide a characterisation of the partitions and explicit expressions for the moments that are given by homomorphism densities defined based on these partitions. We further identify a leading order of the moments and a polynomial in $\lambda^{-1}$, which was also seen for the homogeneous setting in \cite{Jung-Lee}. 

In Section \ref{pap1-Resolventsection}, we prove our second main result which in contrast has an analytic flavour. We set up the relevant analytic structures, and instead of working directly with the Stieltjes Transform, we work with a functional of the resolvent of the adjacency matrix, which was introduced in \cite{KSV2004}. We borrow both fundamental and advanced tools from analysis to provide an exact analytic characterisation of the limiting spectral measure. We conclude with the Appendix as Section \ref{pap1-appendix} where we state the key analytic tools we use in Section \ref{pap1-Resolventsection}.
\section{\bf Setting and Main results}\label{pap1-section 2}
\subsection{\bf Model}
We consider the inhomogeneous Erd\H{o}s-R\'enyi random graph (IER) $\mathbb G_N$ on the vertex set $[N]=\{1, \ldots, N\}$ where edges are independently added with probability $p_{ij}$. As mentioned before we will assume that $p_{ij}$ has a special form as
\begin{equation*}
p_{ij}= \varepsilon_N f( w_i, w_j) \wedge 1,
\end{equation*}
where $\varepsilon_N$ is a tuning parameter such that $\vep_N\to0$, $(w_i)_{i\ge 1}$ is a sequence of deterministic non-negative weights and $f:[0, \infty)^2\to [0,\infty)$  is bounded and continuous.  We will use $\prob_N$ to denote the law of this random graph and we will drop the subscript $N$ for notational convenience, and $\E$ will be the expectation with respect to $\prob$. We will always assume that $N$ is large enough and hence $\varepsilon_N$ is small enough to make $p_{ij}\le 1$ since $f$ is bounded.

 Let $\M_N$ denote the adjacency matrix of the graph $\mathbb G_N$, that is, the $(i,j)$-th entry is $1$ if $i$ shares an edge with $j$, and $0$ otherwise. So $\M_N$ is a symmetric matrix, where any entry $\M_N(i,j)$ is distributed as Bernoulli random variable with parameter $p_{ij}$ as in \eqref{eq:formp} and $\{ \M_N(i,j), i\geq j\}$ is an independent collection. Instead of studying the adjacency matrix $\M_N$ we will study the scaled adjacency matrix. In particular, we do a CLT type scaling by the variance of the entries, that is, we study the matrix
\begin{align}\label{pap1-eq:A_NM_N} \frac{1}{\sqrt{N\varepsilon_N(1-\varepsilon_N)}}\M_N.
\end{align}
The empirical measure which puts mass $1/N$ on each eigenvalue of an $N\times N$ random matrix $\bA_N$ is called the \emph{Empirical Spectral Distribution} of $\bA_N$, and is denoted by 
\begin{equation}\label{eq:esd:def}
\ESD(\bA_N) := \frac{1}{N}\sum_{i=1}^N\delta_{\lambda_i}.
\end{equation}
We are interested in studying the following object:
\begin{align*}
    \ESD\left(\frac{\M_N}{\sqrt{N\vep_N(1-\vep_N)}}\right) = \frac{1}{N}\sum_{i=1}^N\delta_{\lambda_i},
\end{align*}
where $\lambda_1,\ldots,\lambda_N$ are the eigenvalues of $(N\vep_N(1-\vep_N))^{-1/2}\M_N$. 

We are interested in the weak convergence (in probability) of the above measure and the limiting measure is called the \emph{Limiting Spectral Distribution} (LSD). The limiting measure depends on the following two geometric regimes in random graphs and its properties differ in the two cases: 
\begin{itemize}
    \item {\bf Dense Regime}:   $\varepsilon_N \to 0$ and $N\varepsilon_N \to \infty$. The connectivity regime with $N\varepsilon_N \gg C\log N$ falls in this regime.
    \item {\bf Sparse Regime} : $\varepsilon_N \to 0$ and $N\varepsilon_N \to \lambda \in (0,\infty)$.
\end{itemize}

\paragraph{\bf Dense Regime:}
In literature, the dense regime is characterised by $\varepsilon_N \equiv \text{constant}$ but we will not use the features of dense graphs in this article and hence by abuse of terminology we say that a graph is dense when it is not sparse. Let us now recall briefly what happens in the dense regime. The following result was proved in \cite{chakrabarty:hazra:denhollander:sfragara} and also can be obtained from \cite{zhu2020graphon}. 
\begin{theorem}[ {\bf ESD in the dense case}]\label{theorem:dense} Consider the inhomogeneous ER graph with $p_{ij}$ as in \eqref{eq:formp} with $\vep_N\to 0$ and $N\vep_N\to \infty$ . Suppose the deterministic weights satisfy the following assumption:

 Let $o_N$ be an uniform random variable on $[N]$ and let $W_N= w_{o_N}$. We assume that there exists a $W$ with law $\mu_w$ such that $$W_N \xrightarrow{d} W.$$
Then there exists a measure $\mu_f$ which is compactly supported such that  
$$\lim_{N\to\infty}\ESD\left(\frac{\M_N}{\sqrt{N\vep_N(1-\vep_N)}}\right)= \mu_f \text{ weakly in probability}.$$  
\end{theorem}

Many interesting properties of this limiting measure are known. To define the moments we need a quantity which is similar to the homomorphism density of graphons.
Define
\begin{equation}\label{pap1-graphhomo}
  t(H_k,f,\mu_w) := \int_{[0,\infty)^k} \prod_{\{a,b\}\in E(H_k)} f(w_a,w_b) \mu^{\bigotimes k}_w(\dd \mathbf{w}),
\end{equation}
where $H_k$ is a simple graph on $k$ vertices with the edge set $E(H_k)$, $\mu^{\bigotimes k}_w(\cdot)$ is the $k$-fold product measure of $\mu_w(\cdot)$, and $\mathbf{w} = (w_1,...,w_k)$. If we restrict the range of $f$ to $[0,1]$ and take $\mu_w(\cdot)$ as the Lebesgue measure on $[0,1]$, then this quantity is the standard graph homomorphism density (see \cite{lovasz2006limits}). 

The \textit{rooted planar tree} is a planar graph with no cycles, with one distinguished vertex as a root, and with a choice of  ordering at each vertex. The ordering defines a way to explore the tree starting at the root. One of the algorithms used for traversing the rooted planar trees is \textit{depth-first search}. An enumeration of the vertices of a tree is said to have depth-first search order if it is the output of the depth-first search.

We now recall the definition of a Stieltjes transform of a measure $\mu$ on $\mathbb R$. For $z\in \mathbb C^{+}$, where $\C^+$ is the upper half complex plane, the Stieltjes Transform of a measure $\mu$ is given by  
$$\St_\mu(z)= \int_{\mathbb R} \frac{1}{x-z} \mu(\dd x).$$
The following proposition gives the properties of the measure $\mu_f$ which appear in Theorem \ref{theorem:dense}.
\begin{proposition}
\begin{enumerate}
    \item[(a)]\emph{[{\bf Moments}]} The measure $\mu_f$ is the unique probability measure identified by the following moments:
\begin{equation}\label{moments:zhu}
    \int x^{2k}\mu_f(\dd x)=\sum_{j=1}^{C_{k}}t(T_{j}^{k+1},f, \mu_w),\ \int x^{2k+1}\mu_f(\dd x)=0,  \,\, k\ge 0,
\end{equation} 
where $T_{j}^{k+1}$ is the $j^{th}$ rooted planar tree with $k+1$ vertices and $C_{k}$ is the $k^{th}$ Catalan number.
\item[(b)] \emph{[{\bf Stieltjes transform}]} There exists an unique analytic function $\mathcal H$ defined on $\mathbb C^+\times [0,\infty)$ such that
$$
\St_{\mu_f}(z)=\int_0^\infty \mathcal{H}(z, x) \mu_w(\dd x),
$$
and $\mathcal{H}(z, x)$ satisfies the integral equation
\begin{equation}\label{eq:hzx}
z \mathcal{H}(z, x)=1+\mathcal{H}(z, x) \int_0^\infty \mathcal{H}(z, y) f(x, y)\mu_w(\dd y), \quad x\ge 0.
\end{equation}
\end{enumerate}

\end{proposition}

\begin{example}[{\bf Rank 1}]
    One special case which arises in many examples of random graphs, and will be discussed later is when $f$ has a multiplicative structure, that is, $f(x,y)= r(x) r(y)$, where $r:[0,\infty)\to [0,\infty)$ is a bounded continuous function. In this case, the measure $$\mu_f= \mu_s \boxtimes \mu_{r(W)}$$
    where $\mu_s$ is the standard semicircle law and $\mu_{r(W)}$ is the law of $r(W)$ and $\boxtimes$ is the free multiplicative convolution of the two measures. When $r$ is identically equal to $1$ then $\mu_f=\mu_s$, the standard semicircle law. We refer to \citet[Theorem 1.3]{chakrabarty:hazra:denhollander:sfragara} for details. 
\end{example}

\paragraph{\bf Sparse regime.}

The seminal work of \cite{bordenave:lelarge:2010} characterises the limiting spectral distribution for locally tree-like graphs. In particular, if one takes $\bA_N$ to be the scaled adjacency matrix as given in \eqref{pap1-eq:A_NM_N}, they show that if the sequence of random graphs $\{\mathbb{G}_N\}_{N\geq 1}$ have a weak limit $\mathbb{G}$, and  for any uniformly chosen root $o_N\in \mathbb{G}_{N}$, its degree sequence $(\mathrm{deg}(\mathbb{G}_N,o_N))_N$ is uniformly integrable, then there exists a unique probability measure $\mu_{\lambda}$ on $\mathbb{R}$ such that $\lim_{N\to\infty}\ESD(\bA_N) = \mu_{\lambda}$ weakly in probability. Furthermore, it is shown that when $f\equiv 1$, the measure $\mu_{\lambda}$ represents the expected spectral measure associated with the root of a Galton-Watson tree with an offspring distribution of $\mathrm{Poi}(\lambda)$ and weights $1/\sqrt{\lambda}$. This result comes from the theory of local weak convergence, also known as Benjamini-Schramm convergence (see \cite{remcovol2, Benjamini:Schramm}), which is a powerful tool to study spectral measures associated with many sparse random graph models.

In particular, consider the space $\mathbb{H}$ of holomorphic functions $f:\C^+\to\C^+$, equipped with the topology induced by uniform convergence on compact sets. Then, this is a complete separable metrizable compact space. The \emph{resolvent} of the adjacency operator is given as 
\begin{align*}
    \R_{\bA_N}(z) = (\bA_N - zI)^{-1}
\end{align*} for each $z\in\C^+$. The map $z\mapsto \R_{\bA_N}(z)(i,i)$ is in $\mathbb{H}$, and the Stieltjes transform of $\ESD(\bA_N)$ is given by $\tr\R_{\bA_N}(z)$, where $\tr = N^{-1}\Tr$ denotes the normalised trace operator. Let $\mathcal{G}^*$ denote the set of rooted isomorphism classes of rooted connected locally finite graphs. Assume that the random graph sequence $(\mathbb{G}_N)_N$ has the random local limit $\mathbb{G}\in\mathcal{G^*}$, and further that $\mathbb{G}$ is a \emph{Galton Watson Tree} with degree distribution $F_*$, that is, a rooted random tree obtained from a Galton-Watson process with root having offspring distribution $F_*$ and all children having a distribution $F$ (which may or may not be the same as $F_*$). 

Let $\St_{\bA_N}(z)$ denote the Stieltjes transform of the empirical measure $\ESD(\bA_N)$. It was shown in \citet[Theorem 2]{bordenave:lelarge:2010} that there exists a unique probability measure $Q$ on $\mathbb{H}$, such that for each $z\in\C^+$
$$Y(z)\overset{d}= \left(z+ \sum_{i=1}^P Y_i(z)\right)^{-1} $$
where $P$ has distribution $F$ and $Y, \{Y_i\}_{i\geq1}$ are iid with law $Q$ and independent of $P$. Moreover
\begin{align*}
\lim_{N\to\infty}\St_{\bA_N}(z) = \E X(z) \text{ in $L^1$},
\end{align*}  
where $X(z)$ is such that: 
\begin{equation}\label{eq:bordenavelwc}
X(z) \overset{d}= -\left( z + \sum_{i=1}^{P_*} Y_i(z) \right)^{-1},
\end{equation}
where $\{Y_i\}_{i\geq1}$ are i.i.d. copies with law $Q$, and $P_*$ is a random variable independent of $\{Y_i\}_{i\geq 1}$ having distribution $F_*$. 

In \citet[Example 2]{bordenave:lelarge:2010}, we see that the sparse Erd\H{o}s-R\'{e}nyi random graph with $p=\frac{\lambda}{N}$ falls in their setup, and in particular, $P$ is distributed as $\mathrm{Poi}(\lambda)$. For a general $f$,  \citet[Theorem 1]{bordenave:lelarge:2010} still guarantees the existence of $\mu_{\lambda}$, since the graphs we will consider will have a local weak limit known as the \emph{multitype branching process} (see \cite[Chapter 3]{remcovol2} for more details).  As $f$ is bounded, we get that the degree sequence will still remain uniformly integrable. 
% A locally finite graph $G=(V,E)$ can be identified with the \emph{adjacency operator} $\bA$ on the Hilbert space 
% \[
% \ell^2(V) = \left\{  \psi:V\to \mathbb{C} : \sum_{i\in V}|\psi(i)|^2 < \infty   \right\} ,
% \]
% endowed with $\langle\psi,\phi\rangle = \sum_{i\in V} \overline{\psi(i)}\phi(i)$, where $\bA$ acts on the canonical orthonormal basis $(\delta_i)_{i\in V}$ as 
% \[
% \bA\delta_i = \sum_{j\in \partial{i}}\delta_j.
% \]
% Now, $\mu_{\lambda}$ is the spectral measure associated with $\bA$. We refer the reader to the thesis \cite{Salez} {\color{blue} (Cite Salez thesis)} of Salez for a more precise formalisation of this. The moments of this measure are then identified as 
% \[
% m_k(\mu_{\lambda}) = \int x^k \mu_{\lambda}(\dd x) = \tr(\bA^k),
% \]
% where $\tr$ is the normalised trace operator on $\ell^2(V)$. Moreover, if one defines the \emph{resolvent} of $\bA$ as 
% \begin{align*}
%     \R_{\bA}(z) = (\bA - zI)^{-1}
% \end{align*}
% for any $z\in\C^+$ with a positive imaginary part, then, the \emph{Stieltjes transform} of $\mu_{\lambda}$ can be identified as  
% \begin{align*}
% \St _{\nu_{\lambda}}(z) = \int_{\mathbb{R}}\frac{1}{x-z}\mu_{\lambda}(\dd x) = \tr(\R_{\bA}(z)).
% \end{align*}
As mentioned before we will not follow this well-known route of local weak convergence. Instead, we show the above convergence through albeit classical methods. We now introduce the conditions under which we will work. We will have the following sparsity assumption on $\varepsilon_N$ and a regularity assumption on the function $f$ and the weights:
\begin{enumerate}[label=\textbf{A.\arabic*}]
\item \label{item:A1} {\bf Connectivity function:} Let $f:[0,\infty)^2\to [0,\infty)$ be a bounded, continuous function, with $|f|\leq C_f \in (0,\infty)$,
\item \label{item:A2} {\bf Sparsity assumption :} $N\varepsilon_N\to \lambda\in (0,\infty)$,
\item \label{item:A3} {\bf Assumption on weights:} Let $o_N$ be an uniform random variable on $[N]$ and let $W_N= w_{o_N}$. We assume that there exists a $W$ with law $\mu_w$ such that $$W_N \xrightarrow{d} W.$$
\end{enumerate}

 We make some preliminary remarks about the assumptions. Since $f$ is bounded, we can easily see that $f$ is $\mu_w-$ integrable. In the sparse setting, in most important examples, the graph is locally tree-like and this can be seen from the theory of local weak convergence.

 Note that the limit $\lambda \to\infty$ recovers the dense regime. By this choice, we can see that $1 - \varepsilon_N \approx 1$ as N becomes very large, and $N\varepsilon_N(1-\vep_N)\to \lambda$. Thus, our matrix of interest is a scaled adjacency matrix now defined as follows: 
\begin{equation}\label{eq:adjacency}
\bA_N = \frac{1}{\sqrt{\lambda}}\M_N.
\end{equation}

\subsection{\bf Main Results}
In this subsection, we state the main results of this article. As mentioned before in the introduction, we would like to understand first the limiting empirical distribution of the sparse inhomogeneous Erd\H{o}s R\'enyi (IER) Random Graph and also study the behaviour of the measure when the sparsity parameter increases.  
Recall that the adjacency matrix is defined in \eqref{eq:adjacency} and the empirical spectral distribution is denoted by $\ESD(\bA_N)$ (see \eqref{eq:esd:def}). 
In what follows, we will see that \begin{align}\label{pap1-eq:mulambda}
    \lim_{N\to\infty}\ESD(\bA_N)= \mu_\lambda \text{ weakly in probability}\end{align}
and $\mu_{\lambda}\Rightarrow \mu_f$ where $\mu_f$ is as in Theorem \ref{theorem:dense}.  
For the homogeneous case, where $f\equiv 1$, we get the final limit as the classical Wigner's semicircular law, that is, $\mu_f = \mu_{sc}$. These iterated limits were studied in \cite{Jung-Lee}. An interesting open question is how close is $\mu_{\lambda}$ to $\mu_f$. Although we do not manage to give an explicit estimate, through the moment method we show that it is very close and the structure of the moments of $\mu_f$ is hidden inside the structure of moments of $\mu_{\lambda}$. This will be our first result. To describe the moments we need to introduce some notation.

    % \
    
    % \noindent One approach to analyse the empirical and limitings measures is through the so-called ``Method of moments'', that is, we study the quantity 
    % \[
    % M_k := \int x^k \nu_{N,\lambda}(\dd x).
    % \]
    % By spectral decomposition, we naturally get 
    % \[
    % M_k = \frac{1}{N}\sum_{i=1}^N \lambda_i^k = \frac{1}{N}\Tr(A_N^k) = \tr(A_N^k).
    % \]
\subsubsection{\bf Method of moments: Combinatorial Approach} We first define the \emph{Special Symmetric Partitions} which was introduced in \cite{Bose-Saha-Sen-Sen}. Let $\mathcal P(k)$ denote the set of partitions of $k$ and $\mathcal P_2(k)$ be the set of pair partitions where each block has size $2$. Let $NC(k)$ be the set of non-crossing partitions of $[k]$ and $NC_2(k)$ be the set of non-crossing pair partitions of $[k]$. Note that $|NC_2(2k)|= \frac{1}{1+k}{2k \choose k}$ and these are known as the Catalan numbers and represent the even moments of the semicircle distribution.

\paragraph{\bf Partition terminology.}
  Let $\pi$ be a partition of a tuple $[k]$.  Let $\pi$ consist of disjoint blocks $V_1,V_2,\ldots,V_m$, for some $1\leq m\leq k$. We arrange the blocks in the ascending order of their smallest element. For any block $V_i$, a \emph{sub-block} is defined to be a subset of \emph{consecutive integers} in the block. Two elements $j$ and $k$ in a block $V_i$ are said to be \emph{successive} if for all $a$ between $j$ and $k$, $a \notin V_i$. 

\begin{definition}[Special Symmetric Partition]\label{pap1-SSP}
A partition $\pi$ of a tuple $[k]=\{1,2,...,k\}$ is said to be a Special Symmetric partition if it satisfies the following: 
    \begin{itemize}
        \item All blocks of $\pi$ are of even size.
%\item For any two blocks of $\pi$, the sub-blocks of a block that lies inside the other block must be of even size. By sub-block here, we mean a set of consecutive elements inside a block.  Thus, for $\pi = \{\{1,2,3,4,7,8\}, \{5, 6\}\}$, the sub-blocks are $\{1,2,3,4\}$ and $\{7,8\}$.
\item  Let $V\in\pi$ be any arbitrary block, and let $a,b\in V$ be two successive elements in $V$ with $b>a$. Then, either of the following is true:
\begin{itemize}
\item[1.] $b=a+1$,  or
\item[2.] between $a$ and $b$ there are sub-blocks of even size. \\
In other words,
there exist elements $\{a_{i_1},a_{i_1+1},\ldots,a_{i_1+k_1}\}\in V_1$, $\{a_{i_2},\ldots,a_{i_2+k_2}\}\in V_2$, $\ldots, \{a_{i_l},\ldots,a_{i_l+k_l}\}\in V_l$, with $a = a_{i_1} -1$ and $b ={a_{i_l+k_l}}+1$, such that $k_1,k_2,\ldots,k_l$ are even.
\end{itemize}
\end{itemize} 
\end{definition}
We denote the class of Special Symmetric partitions as $SS(k)$. Note that for $k$ odd, $SS(k)=\emptyset$. For example, take $\pi=\{\{1, 4,5, 8\}, \{2, 3, 6, 7\}, \{9, 10\}\} \in SS(10)$. Note here that between $4$ and $5$ in the first block, there are no elements from the other blocks, and between $5$ and $8$, there is the sub-block $\{6,7\}$ that is of even size.

In \cite{Bose-Saha-Sen-Sen} a more elaborate definition was given and this is useful in computations. Later, it was shown by \cite[Section 3]{pernici2021noncrossing} that the definition in \cite{Bose-Saha-Sen-Sen} is equivalent to the above one. In \cite{pernici2021noncrossing}, the set $SS(2k)$ is denoted by $P_2^{(2)}(k)$, a special subclass of $k$-divisible partitions. 
\begin{remark} We note down some important properties of $SS(k)$:
\begin{itemize}
   \item[1.] If $k$ is even, then 
    \[
    \{\pi \in SS(k) : |\pi| = k/2\} = \{\pi \in NC_2(k)\}.
    \]
    \item[2.] $SS(2k)= NC(2k)$ for $1\le k\le 3$. When $k\ge 4$, there are partitions $\pi\in SS(2k)$ that are either crossing or non-paired. For example, for $k=8$, $\{ \{1, 2, 5, 6\}, \{3, 4, 7, 8\}\}$ is a Special Symmetric partition. In particular, crossings start appearing when there are at least two or more blocks in a partition having 4 or more elements. 

    \item[3.] The set of Special Symmetric partitions are in one-to-one correspondence with coloured rooted trees (see \cite[Lemma 5.1]{Bose-Saha-Sen-Sen}) and these trees appeared first in the analysis in the works of \cite{BG2001}.
    \end{itemize}
\end{remark}

Any partition $\pi\in \mathcal P(k)$ can be realized as a permutation of $[k]$, that is, a mapping from $[k]\to [k]$. Let $S_k$ denote the set of permutations on $k$ elements. Let $\gamma=(1, 2, \ldots, k)\in S_k$ be the shift by $1$ modulo $k$. We will be interested in the compositions of the two permutations $\gamma$ and $\pi$, denoted by $\gamma\pi$, and this will be seen below as a partition.
  \begin{definition}[Graph associated to a partition]\label{pap1-partitiongraph}  For a fixed $k\geq 1$, let $\gamma$ denote the cyclic permutation $(1,2,\ldots,k)$. For a partition $\pi$, we define $G_{\gamma\pi}=(V_{\gamma\pi}, E_{\gamma\pi})$ as a rooted, labelled graph associated with any partition $\pi$ of $[k]$, constructed as follows.
    \begin{itemize}
        \item Initially consider the vertex set $V_{\gamma\pi}=[k]$ and perform a closed walk on $[k]$ as $1\to 2\to 3\to \cdots \to k\to 1$ and with each step of the walk, add an edge. 
        \item Evaluate $\gamma\pi$, which will be of the form $\gamma\pi = \{V_1,V_2,\ldots,V_m\}$ for some $m\geq 1$ where $\{V_i\}_i$ are disjoint blocks. Then, collapse vertices in $V_{\gamma\pi}$ to a single vertex if they belong to the same block in $\gamma\pi$, and collapse the corresponding edges. Thus, $V_{\gamma\pi} = \{V_1,\ldots,V_m\}$. 
        \item Finally root and label the graph as follows. 
        \begin{itemize}
            \item \emph{Root:} We always assume that the first element of the closed walk (in this case `1') is in $V_1$, and we fix the block $V_1$ as the root. 
            \item \emph{Label:} Each vertex $V_i$ gets labelled with the elements belonging to the corresponding block in $\gamma\pi$.
        \end{itemize}
        \end{itemize} 
\end{definition}
\begin{remark}
    While $\pi$ is a \emph{partition} and $\gamma$ is a \emph{permutation}, we do a composition in the permutation sense. We read the partition $\pi$ as a permutation, compose it with the permutation $\gamma$, and finally read $\gamma\pi$ as a partition. As an example, consider $\pi = \{\{1,2\},\{3,4\}\}$ and $\gamma=(1,2,3,4)$. To compute $\gamma\pi$, we read $\pi$ as $(1,2)(3,4)$, and compute $\gamma\pi = (1,3)(2)(4)$. We finally read $\gamma\pi$ as $\{\{1,3\},\{2\},\{4\}\}$.
\end{remark}
\begin{example}
    Consider for example partitions of $k=6$ and reading the partitions as permutations and evaluating their composition with $\gamma$ gives us:
    \begin{multicols}{2}
        \begin{enumerate}
        \item  $\pi_1 = \{ \{1,2,5,6\},\{3,4\}\},$
        \item  $\pi_2 = \{ \{1,2,3,4\},\{5,6\}\},$
        \item $\pi_3 = \{ \{1,6\},\{2,3,4,5\}\}.$
    \end{enumerate}
    \columnbreak 
    \begin{enumerate}
        \item $\gamma\pi_1 = \{\{1,3,5\},\{2,6\},\{4\}\}$,
        \item $\gamma\pi_2 =\{\{1,3,5\},\{2,4\},\{6\}\} $, 
        \item $\gamma\pi_3 = \{\{1\},\{2,4,6\},\{3,5\}\} $.
\end{enumerate}
    \end{multicols}
    
     The corresponding graphs $G_{\gamma\pi_1}, G_{\gamma\pi_2}$ and $G_{\gamma\pi_3}$ are as follows:
    \begin{center}
    \begin{tikzpicture}[scale=2] 
    \node[shape=circle,draw=black,scale=0.5] (1) at (2.5,4) {2,6};
    \node[shape=circle,draw=black,scale=0.5] (2) at (1,4) {1,3,5};
    \node[shape=circle,draw=black,scale=0.5] (3) at (2.5,3) {4};
    \path [-] (1) edge node[left] {} (2);
    \path [-] (2) edge node[left] {} (3);

    \node[shape=circle,draw=black,scale=0.5] (4) at (3.5,4) {1,3,5};
    \node[shape=circle,draw=black,scale=0.5] (5) at (5,4) {2,4};
    \node[shape=circle,draw=black,scale=0.5] (6) at (5,3) {6};
    \path [-] (4) edge node[left] {} (5);
    \path [-] (4) edge node[left] {} (6);

    \path [-] (1) edge node[left] {} (2);
    \path [-] (2) edge node[left] {} (3);

    \node[shape=circle,draw=black,scale=0.5] (7) at (6,4) {2,4,6};
    \node[shape=circle,draw=black,scale=0.5] (8) at (7.5,4) {1};
    \node[shape=circle,draw=black,scale=0.5] (9) at (7.5,3) {3,5};
    \path [-] (7) edge node[left] {} (8);
    \path [-] (7) edge node[left] {} (9);
\end{tikzpicture}
 \end{center}
 One can see that structurally the three graphs are the same. However, if we root them on $V_1$, then the first two graphs are different from the third. Further, if we label the vertices as shown, all three graphs become distinct.
\end{example}

\begin{example}
    Here, we illustrate the type of graph structures that can occur for $\pi\in SS(k)$. Consider $k=8$, and the following three partitions. 
    \begin{multicols}{2}
         \begin{enumerate}
        \item $\pi_1 = \{\{1,2,3,4\},\{5,6,7,8\}\}$. 
        \item $\pi_2 = \{\{1,4,5,8\},\{2,3,6,7\}\}$.
        \item  $\pi_3 =\{\{1,2,4,5\},\{3,6,7,8\}\}$.
    \end{enumerate} 
    \columnbreak
    \begin{enumerate}
        \item $\gamma\pi_1 = \{\{1,3,5,7\},\{2,4\},\{6,8\}\}$,
        \item $\gamma\pi_2 = \{\{(1,5\},\{2,4,6,8\},\{3,7\}\}$, 
        \item $\gamma\pi_3 = \{\{1,3,7\},\{2,5\},\{4,6,8\}\}$.
    \end{enumerate}
    \end{multicols}
   
    Then, $\pi_1,\pi_2 \in SS(8)$ but $\pi_3 \notin SS(8)$. Moreover, $\pi_1$ is non-crossing whereas $\pi_2$ has 2 crossings. The corresponding graphs are as below. 
    \begin{center}
    \begin{tikzpicture}[scale=2] 
    \node[shape=circle,draw=black,scale=0.5] (1) at (2.5,4) {2,4};
    \node[shape=circle,draw=black,scale=0.5] (2) at (1,4) {1,3,5,7};
    \node[shape=circle,draw=black,scale=0.5] (3) at (2.5,3) {6,8};
    \path [-] (1) edge node[left] {} (2);
    \path [-] (2) edge node[left] {} (3);

    \node[shape=circle,draw=black,scale=0.5] (4) at (3.5,4) {2,4,6,8};
    \node[shape=circle,draw=black,scale=0.5] (5) at (5,4) {1,5};
    \node[shape=circle,draw=black,scale=0.5] (6) at (5,3) {3,7};
    \path [-] (4) edge node[left] {} (5);
    \path [-] (4) edge node[left] {} (6);

    \path [-] (1) edge node[left] {} (2);
    \path [-] (2) edge node[left] {} (3);

    \node[shape=circle,draw=black,scale=0.5] (7) at (6,4) {4,6,8};
    \node[shape=circle,draw=black,scale=0.5] (8) at (7.5,4) {1,3,7};
    \node[shape=circle,draw=black,scale=0.5] (9) at (7.5,3) {2,5};
    \path [-] (7) edge node[left] {} (8);
    \path [-] (8) edge node[left] {} (9);
    \path [-] (7) edge node[left] {} (9);
\end{tikzpicture}
 \end{center}
\end{example}

The following result is the first main result of the article. This is an extension of the results obtained recently in \cite{Bose-Saha-Sen-Sen} and the homogeneous case obtained in \cite{Jung-Lee}.

\begin{theorem}[Identification of moments]\label{pap1-momenttheorem}
(a) Let $\bA_N$ be the adjacency matrix of the sparse IER random graph as defined in \eqref{eq:adjacency} satisfying assumptions \ref{item:A1}--\ref{item:A3}. Then there exists a deterministic measure $\mu_{\lambda}$ such that 
$$\lim_{N\to\infty}\ESD(\bA_N)=\mu_{\lambda} \text{ weakly in probability}.$$
Moreover, $\mu_{\lambda}$ is uniquely determined by its moments, which are given as follows:
\begin{equation}\label{pap1-limitingmomentexpression}
m_k (\mu_{\lambda}) = \int x^k \mu_{\lambda}(\dd x) = \begin{cases}
    0, & \text{ when $k$ is odd,}\\
    \sum\limits_{l=2}^{k/2+1}\sum\limits_{\substack{\pi\in SS(k):\\ |\gamma\pi| = l}} \lambda^{l - 1 - \frac{k}{2}}\hspace{0.1cm}t(G_{\gamma\pi},f,\mu_w), & \text{ when $k$ is even},
\end{cases}
\end{equation}
where $SS(k)$ is the set of all Special Symmetric partitions of $[k]$ as defined in Definition \ref{pap1-SSP}, $G_{\gamma\pi}$ is the graph associated to a partition $\pi$ as defined in Definition \ref{pap1-partitiongraph}, and $t$ is the homomorphism density as in \eqref{pap1-graphhomo}.

(b) As $\lambda\to\infty$,
$$\mu_{\lambda}\Rightarrow\mu_f,$$  where $\mu_f$ is the measure described in Theorem \ref{theorem:dense}.

\end{theorem}

\begin{remark}
  Note that limiting second moment is given by $m_2= t( G_{\gamma\pi}, f, \mu_{w})$ where $\pi=\{1, 2\}$ and $\gamma \pi= \{\{1\}, \{2\}\}$. Hence $G_{\gamma\pi}$ is the graph with $2$ vertices and $1$ edge. Therefore
\[ m_2(\mu_{\lambda})= \int_{(0,\infty]^2} f(x, y) \mu_w(\dd x) \mu_w(\dd y).\]  
\end{remark}

\vskip20pt
\subsubsection{\bf Stieltjes transform: Analytic approach.} It is well-known that $\mu_{\lambda}$ can be characterised by its Stieltjes transform, which, in turn, can be characterised by a random recursive equation. Local weak convergence is a powerful tool for studying the Stieltjes transform of spectral measures associated with sparse random graphs. However, it becomes challenging to provide accurate estimates on the Stieltjes transform to study local laws and extreme values. Therefore, we present an alternative approach to studying the Stieltjes transform of the spectral measure of IER graphs. The ideas used here originate from the works of \cite{KSV2004}.

We denote the upper half complex plane by $$\mathbb C^{+}=\{z\in \C:\, z = \zeta + \iota\eta, \, \eta>0\}.$$ 
For an analytic approach to the problem, we analyse the \textit{resolvent} of this matrix, defined as 
\[
\R_{\bA_N}(z) := (\bA_N - zI)^{-1} , \, z\in \C^{+}.
\]
The Stieltjes transform of the empirical spectral distribution of $\bA_N$ is given by
\[
\St_{\bA_N}(z) = \int_{\mathbb{R}}\frac{1}{x-z}\ESD(\bA_N)(\dd x) = \tr(\R_{\bA_N}(z)),
\]
where $\tr$ denotes the normalized trace. To get more refined estimates we need an additional assumption on the connectivity function: 
\begin{enumerate}[label=\textbf{A.4}]
\item \label{item:A4} $f:[0, \infty)^2\to [0,\infty)$ is symmetric and bounded by a constant $C_f$. Moreover, $f$ is Lipschitz in one coordinate, that is, for all $x_1, x_2, y\in [0, \infty)$, 
$$|f(x_1, y)-f(x_2, y)|\le C_L|x_1-x_2|$$ where $C_L$ is the Lipschitz constant for $f$.
\end{enumerate}

To state the result we will need a Banach space of analytic functions. Consider the space $\Ba$ defined by 
\begin{equation}\label{pap1-Banachdefn}
\Ba = \left\{ \phi:[0,\infty)\times [0,\infty)\to \C \hspace{0.2cm} \text{analytic } \middle| \hspace{0.2cm} \sup_{x,y\geq 0}\frac{|\phi(x,y)|}{\sqrt{1+y}} <\infty  \right\}
\end{equation} 
and take the norm 
\[
\|\phi\|_{\Ba} = \sup_{x,y\geq 0} \frac{|\phi(x,y)|}{\sqrt{1+y}}.
\]
Then, $(\Ba,\|\cdot\|_{\Ba})$ is a Banach space. We defer the proof of this in Proposition \ref{pap1-Banachspace} in the appendix. 

Consider the function $G_N:[0,\infty)\times \C^+$ given by 
\begin{align}\label{pap1-eq:G_Ndefn}
G_N(u,z) := \frac{1}{N}\sum_{i=1}^N \e^{\ota u r_{ii}^N(z)}
\end{align}
where $r_{ii}^N(z)= \R_{\bA_N}(z)(i,i)$,  the $i^{\text{th}}$ diagonal element of the resolvent of $\bA_N$. It turns out that \begin{align*}
    \left.\frac{\partial G_N(u,z)}{\partial u}\right|_{u=0}= \St_{\bA_N}(z)
\end{align*} and hence one can derive a form of the limiting Stieltjes transform.

\begin{theorem}[Analytic functional of the resolvent]\label{pap1-Resolventtheorem}
Let $\bA_N$ be the adjacency of the IER random graph as defined in \eqref{eq:adjacency} and satisfying assumptions \eqref{item:A2}--\eqref{item:A4}.  Further, consider $G_N$ as defined in \eqref{pap1-eq:G_Ndefn}. Define the function $d_f(x)$ as 
\begin{equation}\label{pap1-eq:degree}
    d_f(y) = \int_0^\infty f(x,y) \mu_w(\dd x).
    \end{equation}
Then, for $z\in \C^+$ there exists a function $\phi^*(x,u) := \phi^*_z(x,u)\in \Ba$ such that for each $z\in\C^+$ and uniformly in $u\in (0, 1]$ we have
\begin{equation}\label{eq:limitGN}
\lim_{N\to\infty}\E[G_N(u, z)] =  1 -\sqrt{u}\int_{0}^{\infty}\e^{-\lambda d_f(y)}\int_0^{\infty} \frac{\J(2\sqrt{uv})}{\sqrt{v}}\e^{\ota vz}\e^{\lambda \phi^*(y,v/\lambda)}\dd v\hspace{0.2cm}\mu_w(\dd y)
\end{equation}
and 
$$\Var[ G_N(u, z)]\to 0.$$
Here, $\phi^*:= \phi^*_z$ is a unique analytic solution (in the space $\Ba$) for the fixed point equation:
    \begin{equation}\label{eq:fixedpt}
    \begin{split}
        \phi^*(x,u) &= F_z(\phi^*)(x,u)\\
        &= d_f(x) - \int_0^{\infty}f(x,y)\e^{-\lambda d_f(y)}\left(\Khorunzkyintegral\e^{\ota vz}\e^{\lambda \phi^*\left(y,\frac{v}{\lambda}\right)} \dd{v}\right)\mu_w(\dd{y}),
    \end{split}
    \end{equation}
    where $\J$ is the Bessel function of the first order of the first kind defined as 
    \begin{equation}\label{eq:bessel}
     \J(x)= \frac{x}{2}\sum_{k=0}^\infty \frac{(-1)^k (x^2/4)^k}{k! (k+1)!}.
    \end{equation}
\end{theorem} 
Observe that there is a slight difference in the right-hand sides of \eqref{eq:limitGN} and \eqref{eq:fixedpt} but in the case $f=1$ both are the same. The next corollary describes the convergence of the Stieltjes transform. 
\begin{corollary}[Identification of the Stieltjes Transform]\label{pap1-Stieltjescorollary} Under the assumptions of the above theorem, we have that any $z\in \C^{+}$, 
$$\St_{\bA_N}(z)\to\St_{\mu_{\lambda}}(z) \text{ in probability},$$
where $\mu_{\lambda}$ is as in Theorem \ref{pap1-momenttheorem}.
    The $\St_{\mu_{\lambda}}(\cdot)$ satisfies the following equation:
    
    \begin{equation}\label{eq:stmain}
    \St_{\mu_{\lambda}}(z) = \ota\int_0^{\infty}\e^{-\lambda d_f(y)} \int_0^{\infty}\e^{\ota vz}\e^{\lambda \phi_z^*(y,\frac{v}{\lambda})}\dd v\hspace{0.2cm}\mu_w(\dd y),
    \,\, z\in \C^{+}.
    \end{equation}
\end{corollary}
To recover the dense regime, we study the asymptotic $\lambda\to\infty$ as in the next corollary.

\begin{corollary}[Stieltjes Transform as $\lambda\to\infty$]\label{pap1-Stieltjeslargelambda}
    For $\lambda\to\infty$, we have that 
    \begin{align}
        \lim_{\lambda\to\infty}\St_{\mu_{\lambda}}(z) = \St_{\mu_f}(z)
    \end{align}
    for each $z\in\C^+$, where $\St_{\mu_f}(z)$ satisfies an integral equation given by
    \begin{align}
        \St_{\mu_f}(z) := \int_0^{\infty} \mathcal{H}(z,x)\mu_w(\dd{x}),
    \end{align}
where $\mathcal H(z,x)$ satisfies the $f$ dependent fixed point equation \eqref{eq:hzx}.
\end{corollary}

\begin{remark}[$\phi^*_z$ and the Stieltjes Transform in the homogeneous setting]\label{remark:st:homogeneous} In the case when $f\equiv 1$, we recover the homogeneous setting. We know $\phi^*_z$ satisfies the fixed point equation \eqref{eq:fixedpt}. If we substitute $f=1$ in \eqref{eq:fixedpt} we get
    \[
    \phi^*(x,u) = 1- \sqrt{u}\int_0^{\infty}\e^{-\lambda }\left(\Khorunzkyintegral\e^{\ota vz}\e^{\lambda \phi^*\left(y,\frac{v}{\lambda}\right)} \dd{v}\right)\mu_w(\dd y).
    \]
    We see that the right-hand side has no dependency on the parameter $x$, and so, we have a unique analytical functional $\widetilde{\phi^*}(u) = \phi^*(x,u)$ that satisfies the fixed point equation
    \begin{equation}\label{pap1-eq:homogeneousFP}
    \widetilde{\phi^*}(u) = 1 - \e^{-\lambda}\Khorunzkyintegral\e^{\ota vz}\e^{\lambda\widetilde{\phi^*}(v/\lambda)}\dd{v}.
    \end{equation}
    This matches the result of \cite{KSV2004}. 

    From Example 2 of \cite{bordenave:lelarge:2010}, we have that  $\widetilde{\phi^*_z}$ has the form $\widetilde{\phi^*_z}(u) = \E[\e^{\ota u X(z)}]$ for each $z\in\C^+$, where $X(z)$ has the law $Q$ as described in \eqref{eq:bordenavelwc}. So, 
    \begin{align*}
    \St_{\mu_{\lambda}}(z) = \ota\int_0^{\infty}\e^{\ota vz}e^{-\lambda+\lambda\E\left[\e^{\ota \frac{v}{\lambda} X(z)}\right] )}\dd{v} = \ota\int_0^{\infty}\e^{\ota vz} \varphi_P \left(  \E\left[\e^{\ota \frac{v}{\lambda} X(z)}\right]  \right) \dd v\hspace{0.2cm},
    \,\, z \in \C^{+},
    \end{align*}
    where 
    \begin{align*}
    \varphi_P(z) = \E[z^P]= \e^{\lambda(z-1)}, \,\, , P\sim \mathrm{Poi}(\lambda).
    \end{align*}

    \end{remark}

\subsection{\bf Examples}
 We now list out a few examples of the model that can be approached by our methods. 
\paragraph{\bf Example 1: Homogeneous Erd\H{o}s-R\'enyi Random Graph.}
     When we have $f\equiv 1$, the model reduces to the standard homogeneous Erd\H{o}s-R\'enyi graph with edge probability $p = \lambda/N$. As discussed, in this case the moments of $\mu_{\lambda}$ can be computed. In particular, we have $t(G_{\gamma\pi}, f, \mu_w)=1$ for all $\pi$. Hence we have
$$m_{2k}(\mu_{\lambda})= \sum_{l=1}^{k}\lambda^{l-k} |\{ \pi\in SS(2k):  |\pi|=l\}|= |NC_2(2k)|+ \sum_{l=1}^{k-1}\lambda^{l-k}|\{ \pi\in SS(2k): |\pi|=l\}|$$
Since the (even) moments of the semicircle law are given by the Catalan numbers, it is immediate that 
$$\lim_{\lambda\to\infty} m_{2k}(\mu_{\lambda})= m_{2k}(\mu_s).$$
Hence Theorem \ref{pap1-momenttheorem}(b) is true in this special case. It is known that $\mu_{\lambda}$ has an absolutely continuous spectrum when $\lambda>1$ (see \cite{bordenave:virag:sen, arras2021existence}). In this case, the Stieltjes transform is given by 
\[
\St_{\mu_{\lambda}}(z)= -\ota \int_0^\infty\e^{\ota v z}\e^{-\lambda+\lambda \widetilde{\phi^*}(v/\lambda)} \dd{v},
\]
and $\widetilde{\phi^*}(v/\lambda)$ satisfies the equation \eqref{pap1-eq:homogeneousFP}. What is interesting and cannot be immediately derived from our results is the rate of convergence of the measure $\mu_{\lambda}$ to $\mu_s$ as $\lambda$ becomes large. In the simulation below we consider the $\lambda=10$ and the simulation already suggests the appearance of semicircle law. We believe the representation above of the Stieltjes transform as in Corollary \ref{pap1-Stieltjescorollary} can be used to prove the rate of convergence as done in the classical Wigner case in \cite{bai2008methodologies}.

\begin{figure}
     \centering
     \begin{subfigure}[b]{0.48\textwidth}
         \centering
         \includegraphics[width=\textwidth]{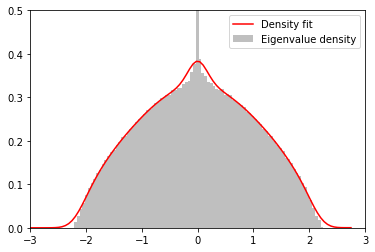}
         \caption{$\lambda=5$.}
         \label{pap1-fig:ERlam5}
     \end{subfigure}
     \hfill
     \begin{subfigure}[b]{0.48\textwidth}
         \centering
         \includegraphics[width=\textwidth]{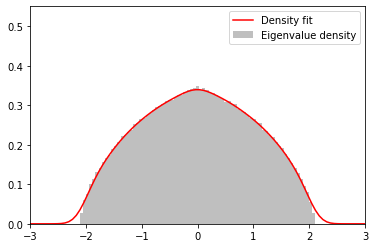}
         \caption{$\lambda=10$.}
         \label{pap1-fig:ERlam10}
     \end{subfigure}
     \hfill
        \caption{The homogeneous Erd\H{o}s-R\'{e}nyi Random Graph on 10,000 vertices.}
        \label{pap1-fig:ERRG}
\end{figure}

%\FloatBarrier

\paragraph{\bf Example 2: Chung-Lu Random Graph:} Let $(d_i)_{i\in [n]}$ be a graphical sequence and denote by $m_1 = \sum_i d_i$ and $m_{\infty} = \max_i d_i$, the total and the maximum degree, respectively. Let $f$ be defined on $[0,1]^2$ as 
    \[
    f(x,y) = xy\wedge 1
    \]
    and 
    \[
    w_i = \frac{d_i}{m_{\infty}}, \hspace{0.3cm} \varepsilon_N = \frac{m_{\infty}^2}{m_1}.
    \]
We can choose an appropriate degree sequence $(d_i)_{i\ge 1}$ such that $m_{\infty} = \lito(\sqrt{m_1})$ and  $N\varepsilon_N \to \lambda$. The connection probabilities will be given by 
    \[
    p_{ij}^{\mathrm{cl}} = \varepsilon_N\left(  \frac{d_id_j}{m_{\infty}^2} \wedge 1\right) = \frac{d_id_j}{m_1}.\]

    Let $o_N$ be a uniformly chosen vertex and $d_{o_N}$ be the degree of this vertex. We assume that $$\frac{d_{o_N}}{m_{\infty}}\overset{d}\to  W$$
    \emph{where $W$ has law $\mu_w$ which is compactly supported.} Then the conditions of Theorem \ref{pap1-momenttheorem} are satisfied. Hence there exists a limiting spectral distribution which we call $\mu_{CL, \lambda}$ and the even moments can identified in the following way.

Let $SS_\ell(2k)$ be the set of Special Symmetric partitions with $l$ blocks. Then, 
\begin{align*}
\int_{\mathbb R} x^{2k} \mu_{CL,\lambda}(\dd{x})&= \sum_{\ell=1}^{k} \sum_{\pi\in SS_{\ell}(2k)} \lambda^{\ell-k} t(G_{\gamma\pi}, f, \mu_w)\\
&=\sum_{\ell=1}^{2k}\sum_{\pi\in SS_{\ell}(2k)} \lambda^{\ell-k} \prod_{j=1}^{\#(\gamma \pi)}\int_{\mathbb R} x^{b_j(\gamma\pi)} \mu_w(\dd{x}),
\end{align*}
 where $b_1(\sigma),\cdots, b_{\#\sigma}(\sigma)$ denotes the size of the blocks of a partition $\sigma$. 
For $\sigma\in NC_2(k)$, its \emph{Kreweras complement} $K(\sigma)$ 
is the maximal non-crossing partition $\bar\sigma$ of $\{\bar 1,\ldots,\bar k\}$, such that 
$\sigma\cup\bar\sigma$ is a non-crossing partition of $\{\bar 1,1, \ldots,\bar k, k\}$. For example,
\begin{eqnarray*}
K\left(\{\{1,2\},\{3,4\},\{5,6\}\}\right)&=&\{\{1\},\{2,4,6\},\{3\},\{5\}\},\\
K\left(\{ (\{1, 2\}, \{3, 6\}, \{4, 5\},\{7,8\}\}\right)&=& \{ \{1, 3, 7\}, \{4, 6\}, \{2\}, \{5\}, \{8\}\}.
\end{eqnarray*}
Note that this slightly differs from the standard notation of Kreweras complement in \cite{nica2006lectures} but for pairings, the $\pi$ and $\pi^{-1}$ coincide. It follows easily that when $\pi\in NC_2(2k)$, $\gamma \pi$ can be replaced by $K(\pi)$. The benefit of this representation is the following. It follows from \cite[Page 228]{nica2006lectures} that
$$\int_{\mathbb R} x^{2k} (\mu_{w}\boxtimes \mu_s)(\dd{x})= \sum_{\pi\in NC_2(2k)} \prod_{j=1}^{k+1} \int_{\mathbb R} x^{b_j(K(\pi))} \mu_w(\dd{x}),$$
where $\mu_w\boxtimes \mu_s$ is the free multiplicative convolution of the measures $\mu_w$ and semicircle law $\mu_s$.
Hence the moments of $\mu_{CL, \lambda}$ can be written as
$$\int_{\mathbb R} x^{2k} \mu_{CL,\lambda}(\dd{x})=\int_{\mathbb R} x^{2k} (\mu_{w}\boxtimes \mu_s)(\dd{x})+\sum_{\ell=1}^{k-1}\sum_{\pi\in SS_{\ell}(2k)} \lambda^{\ell-k} \prod_{j=1}^{\#(\gamma \pi)}\int_{\mathbb R} x^{b_j(\gamma\pi)} \mu_w(\dd{x}).$$

This also shows that 
$$\lim_{\lambda\to\infty}\int_{\mathbb R} x^{2k} \mu_{CL,\lambda}(\dd{x})=\int_{\mathbb R} x^{2k} (\mu_{w}\boxtimes \mu_s)(\dd{x}),$$ 
and consequently, $\mu_f$ is of the form $\mu_{w}\boxtimes \mu_s$.

\begin{remark}
    We want to add a remark about heavy-tailed degrees. Our conditions are not satisfied when the degree sequence follows a power-law distribution. In that case, the $w_i$ need to be scaled differently, and the limiting $W$ will not have a compact support. For further discussion on inhomogeneous random graphs with heavy tails, we refer to \cite[Chapter 6]{remcovol1}.
\end{remark}
% \begin{figure}[h]
% 	  \centering
%    \includegraphics[width=0.5\textwidth]{ChungLu_1to5_10k.png}
%     \caption{Chung-Lu Random graph on 10,000 vertices with $\{d_i\}_i$ uniformly generated integers in $[1,5]$}   \label{pap1-fig:ChungLufig}
% \end{figure}
% \FloatBarrier

\paragraph{\bf Example 3: Generalized random graph:}  
Again let $(d_i)$ be as above. Let $f(x,y)= \frac{xy}{1+xy}$ and $w_i = d_1/\sqrt{m_1}$. Then, 
    \[
    p_{ij}^{\mathrm{grg}} =\frac{d_id_j}{m_1 + d_id_j}.
    \]
Although the above example does not directly fall in our set-up (due to lack of $\vep_N$), one can still derive the limiting spectral distribution using the Chung-Lu model. We will use the following two facts. The first is \cite[Corollary A.41]{bai2008methodologies}, and is also a corollary of the Hoffman-Wielandt inequality.

\begin{fact}\label{fact:HW}
If $d_L$ denotes the L\'evy distance between two probability measures, then for $N\times N$ 
symmetric matrices $A$ and $B$,
\[
d_L^3\left(\ESD(A),\ESD(B)\right)\le\frac1N\Tr\left((A-B)^2\right)\,.
\]
\end{fact}

The following is a fact about the coupling of two Bernoulli random variables with parameters $p$
and $q$ (see \citet[Theorem 2.9]{remcovol2})
\begin{fact} \label{fact:coupling}
There exits a coupling between $X\sim \mathrm{Ber}(p)$ and $Y\sim \mathrm{Ber}(q)$ such that
\[
    \prob(X\neq Y) = |p-q|. 
    \]
\end{fact}
Using the above coupling, we can construct a sequence of independent Bernoulli random variables $(b_{ij})$ and $(c_{ij})$ with parameters $p_{ij}^{\mathrm{cl}}$ and $q_{ij}^{\mathrm{grg}}$, respectively. Let $\M_N^{\mathrm{cl}}$ and $\M_N^{\mathrm{grg}}$ be the adjacency matrix of Chung-Lu and Generalized random graph models respectively, with the above coupled Bernoulli random variables. Suppose the sequence $(d_i)_{i\in[n]}$ satisfies the assumptions described in Example 2 and let $N\vep_N\to \lambda$ and $\bA_N^{\mathrm{cl}}=\lambda^{-1/2} \M_N^{\mathrm{cl}}$ and $\bA_N^{\mathrm{grg}}= \lambda^{-1/2}\M_N^{\mathrm{grg}}$. Then, 
\[
\begin{split}
       \E \left[d_L^3\left(\ESD(\bA_N^{\mathrm{cl}}),\ESD(\bA_N^{\mathrm{grg}})\right)\right] &\leq \frac{1}{N}\E\left[\Tr(\bA_N^{\mathrm{cl}}-\bA_N^{\mathrm{grg}})^2\right] = \frac{1}{N \lambda}\E\left[\sum_{i,j=1}^N (b_{ij} - c_{ij})^2\right] \\
        &= \frac{1}{\lambda N}\E\left[\sum_{i,j=1}^N (b_{ij} - c_{ij})^2 \mathbbm{1}_{\{b_{ij}\neq c_{ij}\}}\right] \\
        &\leq \frac{1}{\lambda N}\sum_{i,j=1}^N \prob(b_{ij}\neq c_{ij}) 
        \leq \frac{1}{\lambda N}\sum_{i,j=1}^N \left|p_{ij}^{\mathrm{cl}} - p_{ij}^{\mathrm{grg}}\right|.
    \end{split}
\]
since $(b_{ij}-c_{ij})^2$ can be trivially bounded by 1. Using $x- \frac{x}{1+x}\le \frac{x^2}{1+x}\le x^2$ for any $x>0$, we have
    \[
    p_{ij}^{\mathrm{cl}} - p_{ij}^{\mathrm{grg}} = \frac{d_id_j}{m_1} - \frac{d_id_j}{m_1 + d_id_j} \le \frac{d_i^2d_j^2}{m_1^2} \leq \frac{m_{\infty}^4}{m_1^2}.
    \]

  %  Also, it follows from a trivial computation that there exists a constant $C_1$ such that
  %  $$\E[(b_{ij}-c_{ij})^4]\le C_1\frac{d_id_j}{m_1}$$
   % and hence we get that
   Therefore
    \begin{align*}
         \E \left[d_L^3\left(\ESD(\bA_N^{\mathrm{cl}}),\ESD(\bA_N^{\mathrm{grg}})\right)\right]&\le \frac{C}{\lambda N} \sum_{i,j=1}^N  \frac{m_\infty^4}{m_1^2}\\
         &= \frac{C}{\lambda N} N^2 \frac{m_\infty^4}{m_1^2}\le \bigO\left(\frac{Nm_{\infty}^4}{m_1^2}\right).
    \end{align*}
    
If we consider $m_{\infty} = o(m_1^{1/4})$, then the empirical distribution functions are close. Now using Markov inequality and the fact that $\ESD(\bA_N^{\mathrm{cl}})$ converges weakly in probability to $\mu_{CL,\lambda}$ it follows that
    $$\lim_{N\to\infty}\ESD(\bA_N^{\mathrm{grg}})= \mu_{CL,\lambda}\, \text{ weakly in probability}. $$

% \begin{figure}[h]
% 	  \centering
%    \includegraphics[width=0.5\textwidth]{GRG_1to5_10k.png}
%     \caption{Generalised Random graph on 10,000 vertices with $\{d_i\}_i$ uniformly generated integers in $[1,5]$}   \label{pap1-fig:GRGfig}
% \end{figure}
% \FloatBarrier

\paragraph{\bf Example 4: Norros-Riettu.} 
Let $(d_i)_i$ be a given sequence and $
    w_i = \frac{d_i}{\sqrt{m_1}}$. Take $f(x,y) = 1 - \exp(-xy)$. Then, 
    \[
    p_{ij}^{\mathrm{nr}} = 1 - \exp\left(-\frac{d_id_j}{m_1}\right).
    \]
Again the form of the above connection probability does not fall directly in our set-up but we can show that Norros-Riettu model is close to the generalized random graph models. Let $\bA_N^{\mathrm{nr}}=\lambda^{-1/2} \M_N^{\mathrm{nr}}$ where $\M_N^{\mathrm{nr}}$ is the adjacency of the Norros-Riettu model. Without loss of generality, we assume that we can couple Bernoulli random variable $c_{ij}$ and $d_{ij}$ with parameters $p_{ij}^{\mathrm{grg}}$ and $p_{ij}^{\mathrm{nr}}$ using Fact \ref{fact:coupling}. Just as in the previous example, it follows using Fact \ref{fact:HW} that
\begin{align*}
  \E \left[d_L^3\left(\ESD(\bA_N^{\mathrm{grg}}),\ESD(\bA_N^{\mathrm{nr}})\right)\right] \leq \frac{1}{\lambda N}\sum_{i,j=1}^N \E\left[(c_{ij}-d_{ij})^2\one_{\{c_{ij}\neq d_{ij}\}} \right]  .
\end{align*}
We bound trivially $(c_{ij}-d_{ij})^2$ by a constant $C_1>0$ and hence we get that
\begin{align*}
  \E \left[d_L^3\left(\ESD(\bA_N^{\mathrm{grg}}),\ESD(\bA_N^{\mathrm{nr}})\right)\right] \leq \frac{C_1}{\lambda N}\sum_{i,j=1}^N \prob\left(c_{ij}\neq d_{ij}\right) =  \frac{C_1}{\lambda N}\sum_{i,j=1}^N ( p_{ij}^{\mathrm{nr}}-p_{ij}^{\mathrm{grg}})
\end{align*}
Now, 
\[\begin{split}
    p_{ij}^{\mathrm{nr}}-p_{ij}^{\mathrm{grg}} &= \left( 1 - \exp\left( -\frac{d_id_j}{m_1}\right) - \frac{d_id_j}{m_1 + d_id_j}\right)\\
    &= \left( \frac{d_i^2d_j^2}{m_1^2 + m_1d_id_j} \right) + \frac{\lambda}{N}\bigO\left( \frac{d_i^2d_j^2}{m_1^2} \right)\\
    &\leq C'\frac{d_i^2d_j^2}{m_1^2},
    \end{split}
    \]
    for some constant $C'>0$. Therefore, for some new constant $C_1'>0$, 
\begin{align}
    \E \left[d_L^3\left(\ESD(\bA_N^{\mathrm{grg}}),\ESD(\bA_N^{\mathrm{nr}})\right)\right]  \leq \frac{C'_1}{\lambda N} \frac{m_2^2}{m_1^2} \label{eq:distm2m1}
\end{align}
where $m_2= \sum_{i=1}^N d_i^2$. Since $W$ has compact support, we have that $\frac{m_2}{Nm_{\infty}}\to \E[W^2]$ and $\frac{m_1}{Nm_{\infty
}} \to \E[W]$. So $\frac{m_2^2}{m_1^2}$ is bounded for large $N$ and hence the right hand side of \eqref{eq:distm2m1} goes to $0$. This shows that

$$\lim_{N\to\infty}\ESD(\bA_N^{\mathrm{nr}})= \mu_{CL,\lambda}\, \text{ weakly in probability}. $$

\begin{figure}[h]
	  \centering
   \includegraphics[width=0.5\textwidth]{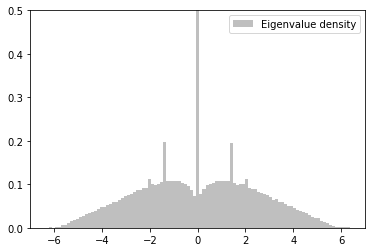}
    \caption{Spectral distributions for the Chung-Lu random graph, the generalised random graph, and the Norros-Riettu random graph on 10,000 vertices with $\{d_i\}_i$ uniformly generated integers in $[1,5]$}   \label{pap1-fig:NRfig}
\end{figure}
\FloatBarrier

\paragraph{\bf Example 5: Inhomogeneous Random Graphs:} Let $w_i = \frac{i}{N}$ and $f:[0,1]^2\to[0,1]$ be any continuous function. Then, 
    \[
    p_{ij} = \varepsilon_Nf\left(\frac{i}{N},\frac{j}{N}\right).
    \]
This is a case which falls directly into our set-up if we assume $N\vep_N\to \lambda$ and the measure $\mu_w$ is the Lebesgue measure. The other examples considered in this section are mostly of the rank-1 type but through this example, one can achieve limiting measures which are of wide variety. 

%\paragraph{\bf Simulations:} We will end this subsection with some simulations to illustrate the examples discussed above.

\begin{figure}
     \centering
     \begin{subfigure}[b]{0.48\textwidth}
         \centering
         \includegraphics[width=\textwidth]{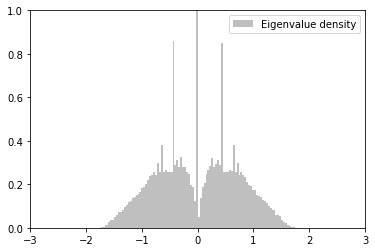}
         \caption{$\lambda=5$.}
         \label{pap1-fig:IRGlam5}
     \end{subfigure}
     \hfill
     \begin{subfigure}[b]{0.48\textwidth}
         \centering
         \includegraphics[width=\textwidth]{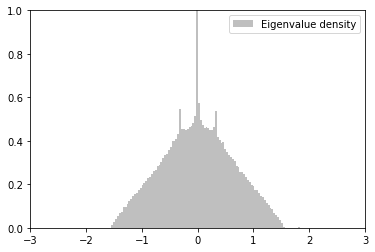}
         \caption{$\lambda=10$.}
         \label{pap1-fig:IRGlam10}
     \end{subfigure}
     \hfill
        \caption{The Inhomogeneous Random Graph on 10,000 vertices, with the inhomogeneity function $f(x,y) = r_1(x)r_1(y) + r_2(x)r_2(y)$, where $r_1(x)=\frac{x}{1+x}$ and $r_2(x)=x$.}
        \label{pap1-fig:IRG}
\end{figure}

\FloatBarrier

We note that in \cite{remcovol2}, inhomogeneous random graphs are introduced in a much more abstract setting, following the works of \cite{bollobas2007phase}. The connectivity function $f$ is generally continuous and also satisfies reducibility properties. The above examples also fall under the setup described there. 

% The second question can be addressed by either of the two approaches, which form the majority of this paper. The question about existence has already been answered by the results of Bordenave et.al., which we rephrase for our setting. 

% \begin{lemma}[Existence of a limiting measure]
% Let $\nu_{N,\lambda}$ be as defined above. Then, $\exists$ a limiting measure $\nu_{\lambda}$ such that $\lim_{N\to\infty}\nu_{N,\lambda} = \nu_{\lambda}$
% \end{lemma}

% This lemma holds since the local weak limit of the graph is a multitype Galton-Watson Tree (more details in the next section), so its degree sequence is uniformly integrable and thus the graph $G_N$ converges to a graph $G_\infty$. By the results of Bordenav-Lelarge, $\exists \mu \in \mathcal{M}_{\leq 1}(\mathbb{R})$ such that $\lim_{N\to\infty}\nu_{N,\lambda} = \mu$. 

%\section{\bf Main Results}\label{Main results}

% \begin{theorem}[Explicit form of the resolvent]\label{pap1-Resolventtheorem} Let $r_{lk} := \R_A(z)_{(l,k)}$ and $r_{lk}^{E_{lk}} := \left(A + \frac{1}{\sqrt{\lambda}}E_{lk} - zI \right)_{(l,k)}^{-1}$ where $E_{lk} = e_le_k^T + e_ke_l^T$. Let $K \in \mathbb{N}$ be a fixed integer and $N>0$. Then, we have 
% \[ 
% \begin{split} 
% \St_{\nu_{N,\lambda}}(z) = \tr(\R_{\bA_N}(z)) &= -\frac{1}{z} - \frac{1}{z}\E\left[\sum_{t=1}^K\frac{1}{N^2\lambda^{t-1}} \left(  \sum_{l=1}^N r_{ll}^t\sum_{k=1}^N r_{kk}^tf(w_k,w_l)\right) \right] \\
% &\hspace{0.5cm} -\frac{1}{z}\E\left[\frac{1}{N^2\lambda^K}\sum_{k,l=1}^N f(w_k,w_l)r_{ll}^{K+1}r_{kk}^Kr_{kk}^{E_{lk}}\right] + \mathrm{O}(N^{-\frac{1}{2}})
% \end{split} 
% \]

% \end{theorem}
\section{\bf Method of Moments: Proof of Theorem \ref{pap1-momenttheorem}}\label{pap1-Momentsection} 
In this section, we will prove the main result Theorem \ref{pap1-momenttheorem} using the method of moments.

 We begin with a small observation. Recall from Assumption \ref{item:A3} that if $o_N$ is an uniformly chosen vertex and $W_N= w_{o_N}$ and we assume $W_N \xrightarrow{\dd} W$. This means that $W_N$ has a distribution function $F_N(x)$ given by
 $$F_N(x)= \frac{1}{N}\sum_{i=1}^N \one_{\{w_i\le x\}}$$ and if we denote by $F$ the distribution of $W$ then for any continuity point $x$ of $F$ we have
 $$F_N(x)\to F(x).$$
 Also for any bounded continuous function $g$, we have $\E[g(W_N)] \to \E[g(W)]$. Let $o_1, \ldots, o_k$ be i.i.d. Uniform random variables on $[N]$. Let $W_{N,i}=w_{o_i}$ for $i=1, \ldots, k$. Then
 %More generally, if $(W_{U_N^{(1)}},...,W_{U_N^{(k)}} )$ are $k$ i.i.d. copies, that is, we have a sequence $\{w_i\}_{i\in I}$ where $I$ is a set of $k$ independent realisations of $U_N$, then, 
\[
\left(W_{N, 1},...,W_{N, k} \right) \xrightarrow{\dd} (W_1, W_2,..., W_k)
\]
 where $W_1,\ldots, W_k$ are $k$ independent copies of the limiting variable $W$. Hence for any bounded continuous $g$ in $k$-variables we have
 \begin{equation}\label{pap1-weightweakconvergence}
     \E\left[g(W_{N,1}, \ldots, W_{N,k})\right]\to \E\left[g(W_1,\dots, W_k)\right].
 \end{equation}

%  This can be seen by the following computation. 
% \[
% \begin{split}
% \prob\left(W_{U_N^{(1)}} \leq k_1, W_{U_N^{(2)}} \leq k_2\right) &= \sum_{i=1}^{k_1}\sum_{j=1}^{k_2} \prob\left( w_i \leq k_1, w_j \leq k_2 | U_N^{(1)}=i, U_N^{(2)} = j \right)\prob\left(U_N^{(1)}=i, U_N^{(2)} = j   \right) \\
% &= \frac{1}{N^2}\sum_{i=1}^{k_1}\sum_{j=1}^{k_2} \E\left( \mathbbm{1}_{\{w_i \leq k_1\}}, \mathbbm{1}_{\{w_j \leq k_2\}} | U_N^{(1)}=i, U_N^{(2)} = j \right)\\
% &= \frac{1}{N^2}\sum_{i=1}^{k_1}\mathbbm{1}_{\{w_i\leq k_1\}} \sum_{j=1}^{k_2}\mathbbm{1}_{\{w_j\leq k_2\}} \rightarrow \prob(w^{(1)}\leq k_1,w^{(2)}\leq k_2)
% \end{split}.
% \]

In our model, we can allow self-loops as we are not imposing that $f(x,x)=0$ but the presence of self-loops does not affect the $\ESD$. The following lemma shows that we can remove the self-loops.

\begin{lemma}[Diagonal contribution] \label{pap1-zerodiagonal}
Let $\widetilde{\bA}_N$ be the matrix $\bA_N$ with zero on the diagonal, and  let $d_L$ denote the L\'{e}vy distance. Then, 
\begin{equation*}
d_L\left( \ESD(\widetilde \bA_N), \ESD(\bA_N)\right)\xrightarrow{\prob} 0.
\end{equation*}
In particular, if $\ESD(\bA_N)$ converges weakly in probability to $\mu_{\lambda}$, then so will  $\ESD(\widetilde \bA_N)$ and visa-versa. 
\end{lemma}
\begin{proof}
Let $D_N$ denote the diagonal of $\bA_N$. Then, $D_N = \bA_N - \tilde{\bA}_N$. Using Fact \ref{fact:HW} we have 
\[
d_L^3\left( \ESD(\widetilde \bA_N), \ESD(\bA_N)\right) \leq \frac{1}{N}\Tr(D_N^2) = \frac{1}{N\lambda}\sum_{1\leq i \leq N}a_{ii}^2.
\]
Hence we have
\begin{align*}
  \E\left[d_L^3\left( \ESD(\widetilde \bA_N), \ESD(\bA_N)\right)\right] &\leq \frac{\sqrt{\lambda}}{N^2} \sum_{1\leq i\leq N} f\left(w_i,w_i\right)\\ 
&\leq C_f\frac{\sqrt{\lambda}}{N},  
\end{align*}
for some constant $C_f$, which comes from the fact that $f$ is bounded. The result follows using Markov's inequality.
\end{proof}

\noindent We are now ready to begin with the proofs of the main results. 
\subsection{\bf Expected Moments}
We split up the proof into three parts. To ease the notation we abbreviate the empirical spectral distribution and its expectation as
\begin{equation}\label{eq:mudef}
\mu_{N,\lambda}(\cdot)=\ESD(\bA_N)(\cdot) \qquad \text{ and } \qquad  \bar{\mu}_{N,\lambda}(\cdot)= \E[\ESD[\bA_N]](\cdot)=\frac{1}{N}\sum_{i=1}^N\prob(\lambda_i\in \cdot).
\end{equation}
Note that $\bar{\mu}_{N, \lambda}$ is now a deterministic measure, for which we compute the moments as 
\[
\begin{split}
\int x^k \bar{\mu}_{N,\lambda}(\dd x) &= \frac{1}{N}\sum_{i=1}^N\int_{\mathbb{R}}x^k\prob(\lambda_i \in \dd{x}) = \frac{1}{N}\sum_{i=1}^N\E[\lambda_i^k] =\E\tr \bA_N^k,
\end{split}
\]
where $\tr$ denotes the normalized trace. Using the trace formula it follows that 
\begin{equation}\label{pap1-traceRMTeqn}
\E[\tr \bA_N^k] = \frac{1}{N}\mathbb{E}[\Tr \bA_N^k] = \frac{1}{N\lambda^{k/2}}\sum_{1\leq i_1, i_2,...,i_k \leq N} \mathbb{E}[a_{i_1i_2}a_{i_2i_3}...a_{i_ki_1}],
\end{equation}
where $a_{ij}$ are entries of the adjacency matrix $\M$. We compute the expected moments and demonstrate that they are finite. Subsequently, we establish a concentration result to show that the moments of the empirical measure converge to $m_k$ in probability. Next, we prove that the sequence $m_k$ satisfies Carleman's condition, thereby uniquely determining the limiting measure. 

Let $SS(k)$ be the set of Special Symmetric partitions, and $\gamma = (1, 2, \ldots, k)$  be the cyclic permutation. For the following computations, one has to read the partition $\pi$ as a permutation, with elements of a block in the partition set in an ascending manner in the permutation. That is, if $\pi = \{\{1,2,5,6\},\{3,4\}\}$, then the corresponding permutation is $(1,2,5,6)(3,4)$. 

\begin{lemma}[{\bf Expected moments}]\label{pap1-expectedmoments}
    Let $\mu_{N,\lambda}$ be the $\ESD$ of $\bA_N$ and $\bar{\mu}_{N,\lambda} = \E\mu_{N,\lambda}$. Let $\gamma\pi$ be decomposed into blocks of the form 
\[
\gamma\pi = \{V_1,V_2,\ldots,V_m\}.
\]
where $m=|\gamma\pi|$ be the number of blocks. 
Define $\mathcal{F}_{\gamma\pi}$ as 
    \begin{equation}\label{eq:gammapidef}
    \mathcal{F}_{\gamma\pi} := \{ \mathbf{i} \in \mathbb{N}^k \mid i_j = i_{j'} \text{ if and only if there exists } l\in [m] \text{ s.t. } j,j' \in V_l    \}.
    \end{equation}
    Then, 
  \begin{align}\label{eq:lemma:moments}
     \int x^k \bar{\mu}_{N,\lambda}(\dd x) = \begin{cases}
    \bigO(\lambda^{k/2} N^{-1}), &\text{$k$ odd}\\
    \sum\limits_{\pi\in SS(k)} \lambda^{|\gamma\pi| - 1 - k/2} \sum\limits_{\mathbf{i}\in \mathcal{F}_{\gamma\pi}} \frac{1}{N^{|\gamma\pi|}} \prod\limits_{(a,b)\in E_{\gamma\pi}} f\left(w_{i_a},w_{i_b} \right) +\bigO(\lambda^{k/2} N^{-1}), &\text{$k$ even} 
    \end{cases}.
    \end{align}
\end{lemma} 

\begin{example}
For $k=4$, take $\pi=\{\{1,2\}, \{3,4\}\}$. Then, $\gamma\pi=\{\{1, 3\}, \{2\}, \{4\}\}$. We see that tuples of the form $(1,2,1,3)$ and $(2,3,2,4)$ belong in $\mathcal{F}_{\gamma\pi}$.
\end{example}
\begin{proof}[Proof of Lemma \ref{pap1-expectedmoments}]
Recall from \eqref{pap1-traceRMTeqn} that 
\[
\frac{1}{N\lambda^{k/2}}\mathbb{E}[\Tr \bA_N^k] = \frac{1}{N\lambda^{k/2}}\sum_{\mathbf{i}\in N^k} \mathbb{E}[a_{i_1i_2}a_{i_2i_3}...a_{i_ki_1}],\]
where $\mathbf{i} = (i_1,\ldots,i_k)$. The term $a_{i_1i_2}a_{i_2i_3}...a_{i_ki_1}$ is associated with the closed walk $i_1 i_2 \ldots i_k i_1$. Let the set of distinct vertices and edges along a closed walk correspond to a $k$-tuple $\mathbf{i}$ be denoted by $V(\mathbf{i})$ and $E(\mathbf{i})$, respectively. An edge that connects vertices $i_j$ and $i_{j+1}$, will be denoted by $e=(i_j, i_{j+1})$. Without loss of generality, we assume that in  $V(\mathbf{i})$ we assign the positions where the first of distinct indices appear in $\mathbf{i}$.

For example, for the 4-tuple $ \mathbf{i}=(1,2,1,3)$, we have $V(\mathbf{i}) = \{1,2,4\}$. So, $E(\mathbf{i}) = \{(1,2), (1,4)\}$. Since $$a_{i_1i_2}a_{i_2i_3}...a_{i_ki_1}=1 \text{ if and only $a_la_{l+1}=1$ for all $(l,l+1)\in E(\mathbf{i})$}$$ we can rewrite \eqref{pap1-traceRMTeqn} as 
\begin{equation}\label{pap1-traceinitialexpression}\begin{split}
\frac{1}{N\lambda^{k/2}}\mathbb{E}[\Tr \bA_N^k] &= \frac{1}{N\lambda^{k/2}}\sum_{1\leq i_j \leq N : j \in V(\mathbf{i})} \left( \frac{\lambda}{N}\right)^{|E(\mathbf{i})|} \prod_{(a,b)\in E(\mathbf{i})} f(w_{i_a},w_{i_b}).
\end{split}
\end{equation}
Let $\pi$ be a partition of $[k] := \{1,2,\ldots,k\}$ and $\gamma\pi = \{V_1, V_2, \ldots, V_m\}$, where $m=|\gamma\pi|$. Recall the definition of $\mathcal{F}_{\gamma\pi}$ as in \eqref{eq:gammapidef} and also the graph $G_{\gamma\pi}$ corresponding to $\gamma \pi$ as in Definition \ref{pap1-partitiongraph}. Note that for a fixed $\mathbf{i}\in \mathcal{F}_{\gamma\pi}$, $V(\mathbf{i})= V_{\gamma\pi}$ and $E(\mathbf{i})=E_{\gamma\pi}$. Moreover,
if $\mathbf{i},\mathbf{i}' \in \mathcal{F}_{\gamma\pi}$, then $V(\mathbf{i}) = V(\mathbf{i}')$ and $E(\mathbf{i}) = E(\mathbf{i}')$. Using this formulation, we can rewrite our summation in \eqref{pap1-traceinitialexpression} once again as

\[
\frac{1}{N\lambda^{k/2}}\mathbb{E}[\Tr \bA_N^k] = \frac{1}{N\lambda^{k/2}}\sum_{\pi\in \mathcal P(k)}\sum_{\mathbf{i}\in \mathcal{F}_{\gamma\pi}} \left( \frac{\lambda}{N}\right)^{|E_{\gamma\pi}|} \prod_{(a,b)\in E_{\gamma\pi}} f\left(w_{i_a}, w_{i_b} \right) .
\]
Since $|\gamma\pi| = |V(\mathbf{i})|$, we can multiply and divide by $N^{|\gamma\pi|}$ to get 

\[
\frac{1}{N\lambda^{k/2}}\mathbb{E}[\Tr \bA_N^k] = \sum_{\pi\in \mathcal P(k)}\frac{1}{N^{|\gamma\pi|}}\sum_{\mathbf{i}\in \mathcal{F}_{\gamma\pi}} \lambda^{|E_{\gamma\pi}| - k/2}N^{|\gamma\pi| - |E_{\gamma\pi}| - 1}  \prod_{(a,b)\in E_{\gamma\pi}} f\left(w_{i_a},w_{i_b} \right) .
\]
Note that since $f$ is bounded, then the product is bounded. For a fixed $k$ and a partition $\pi$ of $[k]$, $|E_{\gamma\pi}| \leq k$. One can also see that $|\mathcal{F}_{\gamma\pi}| \sim N^{|\gamma\pi|}$. We thus focus only on $\lambda^{|E_{\gamma\pi}| - k/2}N^{|\gamma\pi| - |E_{\gamma\pi}| - 1}$. For this to contribute, a tuple $\mathbf{i}$ must yield a tree structure in $G_{\gamma\pi}$, this will give us $|V(\mathbf{i})| = |E(\mathbf{i})| + 1$, which would imply $|\gamma\pi| = |E_{\gamma\pi}| + 1$. In particular, all tuples $\bf{i}\in\mathcal{F}_{\gamma\pi}$ such that $G_{\gamma\pi}$ is a coloured rooted tree as defined in Definition \ref{pap1-partitiongraph} contribute to the summation.

For other graphs with $|V(\mathbf{i})| < |E(\mathbf{i})| +1$, the leading error would be of the order $\bigO(N^{-1})$. The leading order error is given when $G_{\gamma\pi}$ is a $k$-cycle and hence the error is of the order of $\lambda^{k/2} N^{-1}$. Thus our sum reduces to 
\[
\frac{1}{N\lambda^{k/2}}\mathbb{E}[\Tr \bA_N^k] = \sum_{\substack{\pi \in \mathcal P(k) :\\ G_{\gamma\pi} \text{ is a rooted labelled tree}}}\sum_{\mathbf{i}\in \mathcal{F}_{\gamma\pi}} \lambda^{|E_{\gamma\pi}| - k/2}\frac{1}{N^{|\gamma\pi|}} \prod_{(a,b)\in E_{\gamma\pi}} f\left(w_{i_a},w_{i_b} \right) \hspace{0.2cm} + \bigO(\lambda^{k/2} N^{-1}).
\]
Thus rewriting the expression with $|E_{\gamma\pi}| = |\gamma\pi| + 1$ we get,
\begin{equation}\label{pap1-tracetreeexpression}
\frac{1}{N\lambda^{k/2}}\mathbb{E}[\Tr \bA_N^k] = \sum_{\substack{\pi\in \mathcal P(k) : \\ G_{\gamma\pi} \text{ is a rooted labelled tree}}} \lambda^{|\gamma\pi| + 1 - k/2} \sum_{\mathbf{i}\in \mathcal{F}_{\gamma\pi}} \frac{1}{N^{|\gamma\pi|}} \prod_{(a,b)\in E_{\gamma\pi}} f\left(w_{i_a},w_{i_b} \right)  \hspace{0.2cm} + \bigO(\lambda^{k/2} N^{-1}).
\end{equation}

\begin{remark}\label{pap1-oneedgeremark}
    We would like to remark here that if there exists an edge $e$, such that it is traversed only \textit{once} in the closed walk, then the graph cannot be a tree. Consider, without loss of generality, that this edge $e$ is $(1,2)$, with $1\in V_1$ and $2\in V_2$, as in figure \ref{pap1-fig:oneedge}, where $V_1,V_2\in\gamma\pi$. Here $C_1$ and $C_2$ are the remaining components of the graph $G_{\gamma\pi}$.

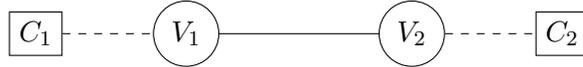
\begin{figure}[h]
    \centering
    \begin{center}
    \begin{tikzpicture}[scale=2] 
    \node[shape=circle,draw=black,scale=1] (1) at (2.5,4) {$V_2$};
    \node[shape=circle,draw=black,scale=1] (2) at (1,4) {$V_1$};
    \node[shape=rectangle,draw=black,scale=1] (A) at (0,4) {$C_1$}; \node[shape=rectangle,draw=black,scale=1] (B) at (3.5,4) {$C_2$};
    
    \path [-] (1) edge node[left] {} (2); 
    \path [dashed] (A) edge node[left] {} (2);
    \path [dashed] (1) edge node[left] {} (B);
\end{tikzpicture}
 \end{center}
    \caption{Graph associated to $\gamma\pi$ having blocks $V_1$ and $V_2$ with the edge between them traversed only once.}
    \label{pap1-fig:oneedge}
\end{figure}
\FloatBarrier
Thus, since the closed walk $1\to2,2\to3,\ldots k\to 1$ has to return back to $V_1$, it has to do so via $C_1$ since the edge $e$ cannot be traversed again. Clearly, this will form a cycle in the graph. Thus, every edge must be traversed at least twice.
\end{remark}
It is well-known (see \cite{nica2006lectures}) that for $\pi\in NC_2(k)$ if and only if $|\gamma \pi|= 1+k/2$, but in the above setting we shall see that other partitions will also contribute as $|\mathcal{F}_{\gamma\pi}| \sim N^{|\gamma\pi|}$. In particular, we need to sum over only those $\pi$ that give rise to a tree structure.  We show in a series of characterizations that the resulting partitions are $SS(k)$.

\paragraph{\bf Characterising partitions} Recall from Definition \ref{pap1-partitiongraph} that to construct a graph $G_{\gamma\pi}$ associated with a partition $\pi$ of $[k]$, we need to evaluate $\gamma\pi$ to construct the vertex set and then perform a closed walk. We prove a property that will play a key role in characterising partitions in the proof of Theorem \ref{pap1-momenttheorem}.

\paragraph{{\bf Property 1.}} [Block characterisation]\label{gammapi}
For $\pi\in \mathcal P(k)$ with $\gamma\pi = \{V_1, \ldots, V_l\}$, if $G_{\gamma\pi}$ has a tree structure, then all elements of a block $V_j$, $\forall\,  1 \leq j \leq l$,  have either all odd elements or all even elements. 

\begin{proof}[Proof of Property 1]
For simplicity, we show that the first block has this property. Assume that $V_1$ has all odd elements except one special element $a\in [k]$. We assume that element `1' belongs to $V_1$.

Recall from the definition of $G_{\gamma\pi}$ that we first perform a closed walk on $[k]$ as $1\to 2 \to 3 \to \ldots \to k \to 1$, and then collapse elements of the same block of $\gamma\pi$ into a single vertex. Thus, if $a-1$ (or $a+1$) belong to $V_1$, then, we get a self-loop since $a-1$ and $a$ collapse to the same vertex and the edge $a-1\to a$ (or $a\to a+1$) forms a loop, which does not give a tree structure. Hence $a-1$ (respectively $a+1$) is not in $V_1$. 

Now, suppose $a-1\in V_j$ for some $j\neq 1$. Then, there exists a path from $V_1$ to $V_j$ of length $t>1$, since if $t=1$, the closed walk $1\to 2\to \ldots$ would imply that $a-2\in V_1$, which contradicts our claim. Now, if $t>1$, the next edge $\{a-1\to a\}$ from the closed walk will be from $V_j$ to $V_1$, leading  to a cycle in the graph. Thus, violating property 1 yields a graph that is not a tree. 

% {\color{red} Here the proof is confusing and not clear, the closed walks were on initial vertices $[k]$ and then graph $G_{\gamma\pi}$ was collapsed to form the reduced graph. Now here is a confusion between walks from blocks and walks with blocks. I dont understand what is going on here. The proof does not fit in with the definitions given yet.}
% WLOG, consider $V_1$ to contain only odd elements, with the exception of the element $e$ which is an even element. If $e-1$ or $e+1$ belong in $V_1$, then we get a self loop by doing a closed walk on the indices $e-1 \to e \to e+1$, which does not give a tree structure. Assume $e-1, e+1 \not\in V_1$. If $e-1 \in V_{k}$, then since $e-1>1$ ($e-1\neq 1 $ as `1'$\in V_1$), we have that there exists a path from $V_1$ to $V_k$ of length $t$ (say). The next edge $\{e-1\to e\}$ from the closed walk will be from $V_k$ to $V_1$, which will lead to a cycle in the graph unless $t=1$. If $t=1$, then we must have that $e-2 \in V_1$, which contradicts the fact that $e$ is the only even element of $V_1$.
\end{proof} 

 \paragraph{{\bf Property 2.}}[Initial characterisation of $\pi$] \label{inipi}
If $\pi\in \mathcal P(k)$ then in any block of $\pi$, no two consecutive elements can either be both odd or both even. 

\begin{proof}[Proof of Property 2]
Suppose $a_1$ and $a_2$ belong in the same block of $\pi$ with no elements between them, and $a_1<a_2$, either both even or odd. Then in $\gamma\pi$, $a_1$ and $ a_2 + 1$ belong in the same block, which contradicts Property 1. 
\end{proof}

\paragraph{\bf Property 3.}[Diagonal terms]\label{diagonalterms}
If $\pi$ is a contributing partition, then for any $\mathbf{i} = (i_1,\ldots,i_k)$ in $\mathcal{F}_{\gamma\pi}$, each element of $\mathbf{i}$ must be pairwise distinct, that is, $i_1 \neq i_2, i_2 \neq i_3,\ldots,i_{k-1} \neq i_k$. 

\begin{proof}[Proof of Property 3]
Suppose not, and assume $i_a=i_{a+1}$ for some $1\le a\le k-1$. Then, in $\gamma\pi$, `a' and `a+1' belong to the same block. This contradicts Property 1.
\end{proof} 
We now use the above properties for further characterisation of the partitions.

%\noindent \textit{Step 1.2: Block sizes and odd moments}
\begin{lemma}\label{even blocks}
Every block in $\pi$ must be of even size. 
\end{lemma}
\begin{proof}[Proof of Lemma \ref{even blocks}]
We prove this by contradiction. Consider an odd-sized block $ V=\{l_1,\ldots, l_r\} \in \pi$ with $l_1< l_2< \cdots < l_r$. Assume that $l_1$ is odd. By Property 2, $l_2$ must be even, and by continuing the argument, we have that at every even position, the element is even, and at odd positions, it is odd. Since $r$ is odd, and $l_r$ is on the $r^{\text{th}}$ position, which is an odd position, $l_r$ must be odd. Then, in $\gamma\pi$, the element $l_r$ will map to the element $l_1+1$ which is even, which contradicts Property 1. A similar argument holds when $l_1$ is taken to be even. This proves the result.
\end{proof}

\begin{corollary}[Vanishing odd moments]\label{oddmoment}
The odd moments vanish as $N\rightarrow \infty$.
\end{corollary}
\begin{proof}[Proof of Corollary \ref{oddmoment}]
 Recall that partitions whose graphs do not yield a tree structure contribute to the error term with leading order $\bigO(N^{-1})$. For $k$ odd, every $\pi\in SS(k)$ must have at least one block of odd size. Therefore, Lemma \ref{even blocks} is violated, and consequently, the odd moments asymptotically. 
\end{proof}

%\noindent \textit{Step 1.3: Classifying contributing partitions}
\begin{proposition}\label{pap1-simplesymmetrictree}
Let $\pi\in \mathcal P(k)$  such that $G_{\gamma\pi}$ is a rooted labelled tree. Then $\pi$ must satisfy the following properties. 
\begin{itemize} 
\item All blocks of the partition must be of even size.
\item Between any two successive elements of a block, there lie sub-blocks of even sizes.
\end{itemize}
\end{proposition}
\begin{proof}[Proof of Proposition \ref{pap1-simplesymmetrictree}]
The first condition is already proved using Lemma \ref{even blocks}. For the second condition, begin by considering a block $B$ that is of the form 
\[
B = \{\ldots,a_1,a_1+1,\ldots, a_1+e,a_2,\ldots\}    
\]
with $a_1-1\notin B$, and there doesn't exist any element $a'$ such that $a_1+e<a'<a_2$ and $a'\in B$. The sub-block here of interest is $\{a_1,a_1+1,\ldots,a_1+e\}$. We claim that this sub-block has an odd number of elements, or equivalently, $e$ is an even number. We can also assume, without loss of generality, that $a_1$ is an odd number. As a consequence of Property 2, $a_2$ must be even. If we now evaluate $\gamma\pi$ using the above information, we have that $\gamma\pi$ contains the following three (and possibly more) blocks.
\[
\begin{split}
    V_1 &= \{\ldots ,a_1, a_1+2,\ldots ,a_1+e, a_2+1,\ldots \},\\
    V_2 &= \{\ldots ,a_1+1, a_1+3,\ldots ,a_1+e-1, a_1+e+1,\ldots\},\\
    V_3 &= \{\ldots ,a_2,\ldots \}.
\end{split}
\]
Thus, the graph associated with $\gamma\pi$ will be as shown in figure \ref{pap1-fig:oddsubblock}, where $C_1$, $C_2$, and $C_3$ are the remaining components of the graph.
\begin{figure}[h]
    \centering
    \begin{center}
    \begin{tikzpicture}[scale=2] 
    \node[shape=circle,draw=black,scale=1] (1) at (2.5,4) {$V_2$};
    \node[shape=circle,draw=black,scale=1] (2) at (1,4) {$V_1$};
    \node[shape=circle,draw=black,scale=1] (3) at (2.5,3) {$V_3$};
    \node[shape=rectangle,draw=black,scale=1] (A) at (0,4) {$C_1$}; \node[shape=rectangle,draw=black,scale=1] (B) at (3.5,4) {$C_2$};
    \node[shape=rectangle,draw=black,scale=1] (C) at (1.5,3) {$C_3$};
    \path [-] (1) edge node[left] {} (2);
    \path [-] (2) edge node[left] {} (3); 
    \path [dashed] (A) edge node[left] {} (2);
    \path [dashed] (1) edge node[left] {} (B);
    \path [dashed] (3) edge node[left] {} (C);
\end{tikzpicture}
 \end{center}
    \caption{Graph associated to $\gamma\pi$ having blocks $V_1$, $V_2$ and $V_3$.}
    \label{pap1-fig:oddsubblock}
\end{figure}
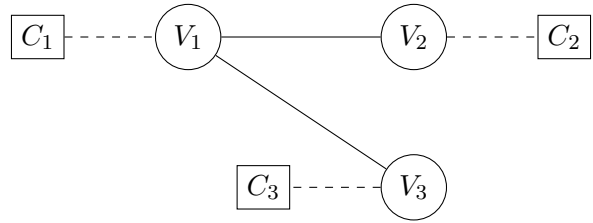
\FloatBarrier
We now focus on the closed walk that occurs on the tuple $[k]$. Since this is a closed walk, it does not matter if instead of beginning at $1$, we begin at an arbitrary element $k_1\in[k]$ and perform $\{k_1\to k_1+1,\ldots, k\to 1, 1\to 2,\ldots k_1-1\to k_1\}$. So, we pick $a_1$ as the starting point and consequently, without loss of generality, we assume the walk begins at $V_1$. 

The walk will immediately proceed to move back and forth between $V_1$ and $V_2$ due to the path $\{a_1\to a_1+1, a_1+1\to a_1+2,\ldots, a_1+e\to a_1+e+1\}$, and will eventually end at $V_2$.  

Now, the walk will jump from $V_2$ into the component $C_2$. On the other hand, when the walk eventually enters $V_3$, it will move at least once to $V_1$, due to the path $\{a_2\to a_2+1\}$. So, to preserve the tree structure, the walk must first come back to $V_2$ and then proceed to $V_3$ via $V_1$. Thus, there is an element $a'$ such that $a'\in V_2$ and $a'+1\in V_1$, where $a'>a_1+e$ and $a'<a_2$. Therefore, in $\gamma\pi$, $a_1+e$ maps to $a'+1$. This implies that $a_1+e$ and $a'$ belong to the same block in $\pi$, and thus, $a'\in B$. This contradicts our construction, and therefore, the walk must form a cycle from $V_2$ or $C_2$ to either $C_1$, $C_3$ or $V_3$. 
\end{proof}
Recall the definition of \textit{Special Symmetric Partitions} as provided in Definition \ref{pap1-SSP}, where the two properties outlined in Proposition \ref{pap1-simplesymmetrictree} are the main characteristics. As a result, we have demonstrated \eqref{eq:lemma:moments}, leading us to the conclusion of the proof of Lemma \ref{pap1-expectedmoments}.
\end{proof}

We would now like to take limits in \eqref{eq:lemma:moments} and finally get the expression for the moments. The following lemma is an easy consequence of Lemma \ref{pap1-weightweakconvergence} and the fact that $|\mathcal{F}_{\gamma\pi}|\sim N^{|\gamma\pi|}$.
\begin{lemma}\label{pap1-integrability}
Let $\pi\in SS(k)$ and $\mathcal{F}_{\gamma\pi}$ be as in Lemma \ref{pap1-expectedmoments}. Also, $G_{\gamma\pi}=(V_{\gamma\pi},E_{\gamma\pi})$ be the graph as in Definition \ref{pap1-partitiongraph}.
    \begin{equation}\label{eq:limitmoments}
\lim_{N\to\infty} \sum_{\mathbf{i}\in\mathcal{F}_{\gamma\pi}} \frac{1}{N^{|\gamma\pi|}} \prod_{(a,b)\in E_{\gamma\pi}} f\left(w_{i_a},w_{i_b}\right) = \int_{[0,\infty)^{|\gamma\pi|}} \prod_{(a,b)\in E_{\gamma\pi}}f(w_a,w_b) \mu_w^{\bigotimes |\gamma\pi|}( d\mathbf{w}). 
\end{equation}
\end{lemma}
% \begin{proof}[Proof of proposition \ref{pap1-integrability}]
%     Fix an even $k>0$ and $\pi$ a partition of $[k]$. Let $g \in C_b(E^{|\gamma\pi|},[0,1])$. Then, 
% \[\begin{split}
% &\E\left[g\left(w_{U^{(1)}_N},w_{U^{(2)}_N},...,w_{U^{(|\gamma\pi|)}_N}\right)\right] \\
% &:= \sum_{\mathbf{i} \in \mathcal{F}_{\gamma\pi}} g(w_{i_1},w_{i_2},...,w_{i_{|\gamma\pi|}})\prob(U^{(1)}_N=i_1,U^{(2)}_N=i_2,...,U^{(|\gamma\pi|)}_N=i_{|\gamma\pi|}) \\
% &=\frac{1}{N^{|\gamma\pi|}}\sum_{\mathbf{i} \in \mathcal{F}_{\gamma\pi}} g(w_{i_1},w_{i_2},...,w_{i_{|\gamma\pi|}}).
% \end{split}
% \]
% But, from \eqref{pap1-weightweakconvergence}
% \[
% \E\left[g\left(w_{U^{(1)}_N},w_{U^{(2)}_N},...,w_{U^{(|\gamma\pi|)}_N}\right)\right] \xrightarrow{N\to\infty} \E[g(W^{(1)},...,W^{(|\gamma\pi|)})]
% \]
% where 
% \[
% \E[g(W^{(1)},...,W^{(|\gamma\pi|)})] = \int_{E^{|\gamma\pi|}} g(W_{a_1},W_{a_2},...,W_{a_{|\gamma\pi|}}) \mu_w^{\bigotimes |\gamma\pi|} \left( \dd \mathbf{W} \right).
% \]
% Setting $g(w_{i_1},w_{i_2},...,w_{i_{|\gamma\pi|}}) = \prod_{i_ai_b\in E_{\gamma\pi}} f(w_a,w_b)$ completes the proof of proposition \ref{pap1-integrability}.
% \end{proof}
Now, going back to equation \eqref{pap1-tracetreeexpression} and taking limits gives us  
\begin{align}\label{pap1-limitingmomentseq1}
\lim_{N\to\infty}\E[\tr \bA_N^k]= \begin{cases}
    0, &\text{$k$ odd}\\ \sum\limits_{\substack{\pi\in SS(k)}} \lambda^{|\gamma\pi| -1 - k/2}t(G_{\gamma\pi},f,\mu_w), &\text{$k$ even}
\end{cases}.
\end{align}
Now, the sum over $SS(k)$ can be further split up as the sum over $NC_2(k)$ and the remaining partitions. Moreover, for $\pi\in SS(k)$, we have $|V_{\gamma\pi}| = |\gamma\pi| \in \{2,3,\ldots, k/2+1\}$. In particular, for $\pi\in NC_2(k)$, $|\gamma\pi| = k/2+1$, and when $\pi$ is the full partition $\{\{1,2,\ldots,k\}\}$, $|\gamma\pi| = 2$. So, we can write
\begin{align}\label{pap1-limitingmoments}
\lim_{N\to\infty}\E[\tr \bA_N^k]= \begin{cases}
    0, &\text{$k$ odd}\\
    \sum\limits_{\pi\in NC_2(k)} t(G_{\gamma\pi},f,\mu_w) + \sum\limits_{l=2}^{k/2}\sum\limits_{\substack{\pi\in SS(k)\setminus NC_2(k):\\ |\gamma\pi| = l}} \lambda^{l -1- k/2}t(G_{\gamma\pi},f,\mu_w), &\text{$k$ even}
\end{cases}.
\end{align}
% Now, if one takes the limit $\lambda\to\infty$ in \eqref{pap1-limitingmoments}, we obtain 
% \[
% \lim_{\lambda\to\infty}\E m_{2k} = \sum\limits_{\pi\in NC_2(k)} t(G_{\gamma\pi},f,\mu_w) ,
% \]
% which characterise the moments $\{m'_k\}_k$ of the measure $\nu_f$. 

%If $f\equiv 1$, then $t(G_{\gamma\pi},f,\mu_w) = 1$. The sum of $NC_2(2k)$ is precisely the Catalan number $C_k$, which are the moments of the semicircular law $\mu_{sc}$.

\

\subsection{\bf Concentration and uniqueness.} We now show a concentration result to obtain convergence in probability. 
\begin{lemma}[Concentration of trace] \label{pap1-traceconcentrate}
For all $k\geq 0$, we have that 
\[
\Var \left[\tr \bA_N^k \right] = \mathrm{O}((\lambda N)^{-1}) .
\]
\end{lemma} 
\begin{proof}
We shall proceed to compute the variance 
\[
\Var  \left[  \tr \bA_N^k 	\right].
\]
Let $\mathbf{i}$ and $\mathbf{i'}$ denote the tuples 
\[
\mathbf{i} = \{i_1,\ldots,i_k\}, \hspace{0.2cm} \mathbf{i'} = \{i_{k+1},\ldots,i_{2k}\}
\]
and denote by $P(\mathbf{i})$ the expectation 
\[
P(\mathbf{i}) = \E[a_{i_1i_2} a_{i_2i_3}\ldots a_{i_ki_1}].
\]
Similarly, we have
\[
P(\mathbf{i'}) = \E[a_{i_{k+1}i_{k+2}} a_{i_{k+2}i_{k+3}}\ldots a_{i_{2k}i_{k+1}}].
\]
For the tuple $\mathbf{i}$, we can define a closed walk as in the proof of Lemma \ref{pap1-expectedmoments} to get a graph $G(\mathbf{i}):= (V(\mathbf{i}),\, E(\mathbf{i}))$. In the same spirit, one can define $G(\mathbf{i},\mathbf{i'}) = (V(\mathbf{i},\mathbf{i'}),E(\mathbf{i},\mathbf{i'}))$, with the closed walk now performed as 
\[
1\to 2\to\ldots k\to 1, k+1\to k+2\to\ldots 2k\to k+1,
\]
where the jump from 1 to $k+1$ is without an edge. Then, we can define
\[
P(\mathbf{i},\mathbf{i'}) = \E[a_{i_1i_2} a_{i_2i_3} \ldots a_{i_ki_1} a_{i_{k+1}i_{k+2}}\ldots a_{i_{2k}i_{k+1}}] .
\]
With this notation set up, one can see that 
\begin{equation}\label{pap1-concentrationvariance}
\begin{split}
\Var \left[	\tr \bA_N^k\right] &= \frac{1}{N^2} \left[ \E [ (\Tr(\bA_N^k)^2] - (\E[\Tr(\bA_N^k)])^2 \right] \\
&= \frac{1}{N^2\lambda^k} \sum_{i_1,i_2,\ldots,i_k,i_{k+1},\ldots,i_{2k}=1}^N P(\mathbf{i},\mathbf{i'}) - P(\mathbf{i})P(\mathbf{i'}).
\end{split}
\end{equation}
We remark here that the construction of the graph $G(\mathbf{i},\mathbf{i'})$ is similar to how we did in Lemma \ref{pap1-expectedmoments}, with the essential difference being the closed walk structure over two separate $k-$tuples. 

Suppose that $E(\mathbf{i})\cap E(\mathbf{i'}) = \phi$. Then by independence, \eqref{pap1-concentrationvariance} becomes 0. Thus, we must have $E(\mathbf{i})\cap E(\mathbf{i'}) \neq \phi$. Moreover, due to remark \ref{pap1-oneedgeremark}, each term must appear at least twice in $P(\mathbf{i},\mathbf{i'})$, that is, each edge in $E(\mathbf{i},\mathbf{i'})$ is traversed at least twice. This implies that the maximum number of edges our graph can have is $k$.

Next, note that the only way the graph $G(\mathbf{i},\mathbf{i'})$ will be disconnected is when the closed walk over the two $k-$ tuples yields two disjoint graphs, and thus we once again obtain $P(\mathbf{i},\mathbf{i'}) = P(\mathbf{i})P(\mathbf{i'})$. 

Thus, our computation boils down to the case where $G(\mathbf{i},\mathbf{i'})$ is a connected graph, with each edge appearing at least twice, and $E(\mathbf{i})\cap E(\mathbf{i'}) \neq \phi$. Note that one can have $G(\mathbf{i},\mathbf{i'})$ to be connected and still have $E(\mathbf{i})\cap E(\mathbf{i'})=\phi$, for example when $i_1$ and $i_{k+1}$ are collapsed into the same vertex. 
This gives us that $|V(\mathbf{i},\mathbf{i'})| \leq |E(\mathbf{i},\mathbf{i'})| +1 \leq k+1$. Using $|f|\leq C_f$ gives us that 
\[
\Var \left[	\tr \bA_N^k\right] \leq C_f\frac{1}{N^2\lambda^k}N^{|V|}\left(  \frac{\lambda}{N}  \right)^{|E|} = C_f\lambda^{|E|-k}N^{|V|-|E|-2} = \bigO(N^{-1}).
\]
% To make this error uniform in $k$, we can improve the bound further. Recall from the moment method computation, as in equation \eqref{pap1-tracetreeexpression}, that contributing graphs have a tree structure. Now, there is some edge of $G(\mathbf{i})$ (say $i_li_{l+1}$ with $l<k$) that has also appeared in $G(\mathbf{i'})$ (since $E(\mathbf{i}) \cap E(\mathbf{j}) \neq \phi$). But, $i_li_{l+1}$ must have been traversed twice in $G(\mathbf{i})$ as well since otherwise there is a cycle that arises when the walk returns to `1'. This implies that there is an edge traversed in $G(\mathbf{i},\mathbf{j})$ thrice, contradicting our initial claim. Thus, we have at most $k-1$ edges with $k$ vertices, which implies that the summand runs over $k$ indices (that is, $k$ vertices), with $P(\mathbf{i},\mathbf{j})$ having $k-1$ distinct terms (that is, $k-1$ distinct edges of the form $i_li_{l+1}$). Using $|f| \leq 1$ yields 
% \[
% P(\mathbf{i},\mathbf{j}) \leq \left(\frac{\lambda}{N}\right)^{k-1} ,
% \]
% and thus \eqref{pap1-concentrationvariance} yields
% \[
% \Var \left[	\frac{1}{N}\Tr((\bA_N^{\circ})^k)\right] \leq \frac{1}{N^2\lambda^k}N^k\left(\frac{\lambda}{N}\right)^{k-1} = \frac{1}{N\lambda}.
% \]
This completes the proof. 
\end{proof} 
% The following corollary is immediate using Chebychev's inequality for any $\delta>0$. 
% \begin{corollary}[Convergence in probability]\label{pap1-traceconcentratecorrolary} 
% \[
% \prob\left( \left|\tr(\bA_N^k) - \E\left[	\tr(\bA_N^k)	\right] \right| > \delta \right) \leq \frac{1}{N\delta^2}.
% \]
% Thus, 
% \[
% \prob(|m_k - \E m_k| \geq \varepsilon) \xrightarrow{N\to\infty} 0.
% \]
% \end{corollary}
An immediate consequence using Chebychev's inequality is that the moments concentrate around their mean as $N\to\infty$. In other words, for all $k\geq 1$,
\[
\lim_{N\to\infty}\tr \bA_N^k = m_k(\mu_{\lambda}) \hspace{0.2cm}\text{in probability},
\]
where $m_k(\mu_{\lambda})$ are as in \eqref{pap1-limitingmomentexpression}. To conclude Theorem \ref{pap1-momenttheorem}, we now further analyse the sequence $\{m_k\}_{k\geq 0}$, and show that it is unique for the measure $\mu_{\lambda}$. A measure $\mu$ is said to be uniquely determined by its moment sequence $\{m_k\}_{k\geq0}$ if the following holds (Carleman's condition):
\begin{align}\label{eq:Carlemancondition}
    \sum_{k\geq 0} m_{2k}^{-1/2k} = \infty.
\end{align}
\begin{lemma}[Uniqueness of moments]\label{pap1-Carleman}
For $\lambda$ bounded away from 0, that is, $\lambda  >0$, the moments uniquely determine the limiting spectral measure.  
\end{lemma}

\begin{proof}
Let $m_k$ denote the $k^{\text{th}}$ moment. Since $f$ is bounded, we have  
\[
\begin{split}
m_{2k} &= \sum_{\pi \in SS(2k)} \lambda^{|\gamma\pi| - 1 - k} \int_{[0,1]^{|\gamma\pi|}} \prod_{(ab) \in E_{\gamma\pi}} f(x_a,x_b) \prod_{i=1}^{|\gamma\pi|} \mu_w(\dd x_i) \\
&\leq \sum_{\pi \in SS(2k)} C_f^{|\gamma\pi|}\lambda^{|\gamma\pi| - 1 -k} \\
&= \sum_{l=2}^{k+1}\sum_{\pi \in SS(2k) : |\gamma\pi|=l} C_f^l\lambda^{l-1-k} ,\end{split}
\]
Let $A_k$ be defined as 
\[
A_k = \begin{cases}
1,& \hspace{0.3cm} \text{if $\lambda \geq 1$},\\
\lambda^{-k}, &\hspace{0.3cm} \text{if $1>\lambda>0$}.
\end{cases}
\]
Then,
\[
\begin{split}
m_{2k} &\leq C_f^{k+1}A_k\sum_{l=2}^{k+1} |\{\pi\in SS(2k) : |\gamma\pi| = l\}| \\
&\leq A_kC_f^{k+1}|\{SS(2k)\}| \\
&\leq A_kC_f^{k+1} (2k)^{2k},
\end{split}
\]
where the last inequality follows since $SS(2k) \subset P(2k)$ and $|P(2k)|$ is bounded by $2k^{2k}$. Thus, 
\[
m_{2k}^{-1/2k} \geq \frac{1}{2k\sqrt{C_f}}.\frac{1}{(A_kC_f)^{\frac{1}{2k}}} 
\]
So, we have the series $\sum_{k\geq1} m_{2k}^{-1/2k}$ to be lower bounded by $\sum_{k\geq1} a_k$, where
\[
a_k = \frac{1}{2k\sqrt{C_f}}.\frac{1}{(A_kC_f)^{\frac{1}{2k}}} = \frac{1}{C_1k\e^{\frac{1}{2k}\log(A_kC_f)}}.
\]
Thus, 
\begin{align*}
    a_k = \begin{cases}
        \frac{\e^{-C_2/2k}}{C_1k}, &\text{for $\lambda\geq 1$},\\
        \frac{\sqrt{\lambda}\e^{-C_2/2k}}{C_1k}, &\text{for $\lambda<1$}.
    \end{cases}
\end{align*}
Since $\e^{-x} > 1- x$, we see that the series $\sum_{k\geq 1}a_k$ diverges, and consequently, 
\[
\sum_{k\geq 0} m_{2k}^{-1/2k} = \infty.
\]
\end{proof}

\section{\bf Analytic approach: Proof of Theorem \ref{pap1-Resolventtheorem}}\label{pap1-Resolventsection}
\subsection{\bf Resolvent and Stieltjes Transform}
We fix a $z\in\C^+$ throughout this argument, with $\Im(z) = \eta>0$. Recall that resolvent is given by 
\[
\R_{\bA_N}(z) := (\bA_N - zI)^{-1} , \, z\in \C^{+}.
\]
The Stieltjes transform of the empirical spectral distribution of $\bA_N$ is given by
\begin{equation}\label{eq:defstieltjes}
  \St_{\bA_N}(z) = \int_{\mathbb{R}}\frac{1}{x-z}\ESD(\bA_N)(\dd x) = \tr(\R_{\bA_N}(z)),
\end{equation}

where $\tr$ denotes the normalized trace.

\begin{lemma}[Resolvent Properties]\label{resolventproperties}
 For any $z \in \mathbb{C}^{+}, 1 \leq i, j \leq N$, the following properties are well-known for the resolvent $\R_{\bA}$ of an $N\times N$ matrix $\bA$. 
 \begin{itemize}
\item[(i)] {\bf Analytic:} $z \mapsto \R_{\bA}(z)(i,j)$ is an analytic function on $\mathbb{C}^{+} \rightarrow \mathbb{C}^{+}$.
\item[(ii)] {\bf Bounded :} $\|\R_{\bA}(z)\|_{\text{op}} \leq \Im(z)^{-1}$, where $\|\cdot\|_{\text{op}}$ denotes the operator norm. 
\item[(iii)] {\bf Normal :} $\R_{\bA}(z) \R_{\bA}(z)^*=\R_{A}(z)^* \R_{A}(z)$.
\item[(iv)]  {\bf Diagonals are bounded:} $|\R_{\bA}(z) (i,j)|\le \Im(z)^{-1}.$
\item[(v)] {\bf Trace bounded:} $ |\tr(\R_{\bA}(z))| \leq \Im(z)^{-1}$. In particular, 
    \[
    \left|\tr(\R^p_{\bA}(z))\right| \leq \Im(z)^{-p}, \text{ for any $p\ge 1$.}
    \]  
\end{itemize}
\end{lemma}

For the first three properties see \cite[Chapter 3]{bordenave2019lecture}. Note that the property (iv) follows from (iii) by the following argument:
$$|\R_{\bA}(z) (i,j)| \le | \langle \delta_i,\,  \R_{\bA}(z)\delta_j\rangle|\le \sup_{v:\|v\|=1} |\langle \delta_i, \R_{\bA}(z)\delta_j\rangle|= \|\R_{\bA}(z)\|_{\text{op}}.$$ 
The last property (v) follows from (iv). We now state the Ward's identity, for which we refer the reader to \citet[Lemma 8.3]{Laszlobook}.

\begin{lemma}[Ward's identity]\label{pap1-Wald}
Let $\bA$ be a Hermitian matrix and $\R_{\bA}$ be the resolvent. Let $z\in \mathbb{C}^+$. Then for any fixed $k$, we have 
\[
\sum_{l\neq k}|R_{\bA}(l,k)|^2 = \frac{1}{\eta}\Im(R_{\bA}(k,k)). 
\]
\end{lemma} 

Since we have already shown in the previous section $\lim_{n\to \infty} \ESD(\bA_N)= \mu_{\lambda}$ weakly in probability and hence it follows that for any $z\in \mathbb C^+$
$$\lim_{N\to \infty}\St_{\bA_N}(z) \to S_{\mu_{\lambda}}(z). $$
Due to the involved structure of the moments, it is not immediately evident what the limiting Stieltjes transform looks like. 

Recall the notation of \emph{expected} empirical spectral distribution of $\bA_N$ from \eqref{eq:mudef}. Let $\bar\St_{\bA_N}(z) $ denote the Stieltjes transform of $\bar\mu_{N,\lambda}$. Notice that
$\bar\St_{{\bA}_N}(z)= \E[\St_{\bA_N}(z)]$. It is known that if a measure $\mu_N$ converges weakly in probability to a measure $\mu$, then the corresponding Stieltjes transforms converge. In particular, we have the following lemma. 
\begin{lemma}\label{stieltjesconvergencefact} \cite[Theorem 2.4.4]{anderson2010introduction}
  A sequence of measures $\mu_N$ converge weakly in probability to a measure $\mu$ if and only if $\St_{\mu_N}(z)$ converges in probability to $\St_{\mu}(z)$ for each $z\in\C^+$.
\end{lemma}
Thus, we compute an expression for the expected Stieltjes transform $\St_{\bar{\bA}_N}$, and using convergence in probability from Theorem \ref{pap1-momenttheorem}, we can claim that the Stieltjes transform $\St_{\bA_N}(z)$ converges in probability to the same expression. For ease of notation we shall denote by $r_{kk}^N(z):= \R_{\bA_N}(z)(k,k)$ for $1\le k\le N$.

The following identity can be found in \cite{abramowitz+stegun}. For any complex number $z\in \C^+$, we have for all $u\geq 0$, 
    \begin{equation}\label{pap1-exponentialBesselidentity}
    \e^{\ota uz} = 1 - \sqrt{u}\int_0^{\infty}\frac{\J(2\sqrt{uv})}{\sqrt{v}}\e^{-\ota vz^{-1}}\dd{v},
    \end{equation}
    where $\J(x)$ is the first-order Bessel function of the first kind given by \eqref{eq:bessel}. Note that for all $x\geq 0$, $|\J(x)|\leq 1$ (see \cite[Chapter 9]{abramowitz+stegun}). We know that resolvent maps the upper half complex plane to the upper half complex plane. Thus, we begin by fixing $r^N_{jj}(z)$, the $j^{\text{th}}$ diagonal entry of the $N\times N$ resolvent matrix, as our complex variable in $\C^+$. So we can get
    \begin{equation}\label{eq:expresolvent}
        \e^{\ota u r_{jj}^N(z)}= 1 - \sqrt{u}\int_0^{\infty}\frac{\J(2\sqrt{uv})}{\sqrt{v}}\e^{-\ota v(r_{jj}^N)^{-1}}\dd{v}.
    \end{equation}
    If we look at $\sum_{j=1}^N \e^{\ota u r_{jj}^N(z)}$ then the relation between the Stieltjes transform and the above equation becomes apparent. It turns out that
    \begin{equation}\label{eq:derivativeST}
        \St_{\bA_N}(z)= \left. \frac{\partial}{\partial u}\frac{1}{N}\sum_{j=1}^N \e^{\ota u r_{jj}^N(z)} \right|_{u=0}.
    \end{equation}
    
    To understand the Stieltjes transform we will first try to understand the behaviour of \eqref{eq:expresolvent}. We will adapt the approach of \cite{KSV2004}. For ease of notation, for what follows, $\|\cdot\|$ will denote the norm $\|\cdot\|_{\Ba}$ as defined in \eqref{pap1-Banachdefn}, unless stated otherwise. 

    \begin{proposition}\label{pap1-analyticstep1}
        Let $r^N_{jj}:=r^N_{jj}(z)$ denote the $j^{\text{th}}$ diagonal entry of the resolvent $\R_{\bA_N}(z)$. Let 
        \begin{equation}\label{eq:defd_j}
        d_j = \frac{1}{N}\sum\limits_{k=1}^Nf(w_j,w_k)
        \end{equation}
        and for any $b>0$ define the function $g_N: (0,\infty)\times (0,\infty)\times \C^+\to \C$ as follows
        \begin{equation}\label{def:g_N}
        g_N(x,b,z) := \frac{1}{N}\sum\limits_{k=1}^N f(x,w_k)\e^{\ota b r^N_{kk}(z)}.
        \end{equation}
         Then, for any $z\in\C^+$,
        \begin{equation}\label{pap1-expBesselinterpolated}
        \E[\e^{r^N_{jj}}] = 1 - \e^{-\lambda d_j}\Khorunzkyintegral\e^{\ota vz}\E\left[\e^{\lambda g_N\left(w_j,\frac{v}{\lambda},z\right)}\right] \dd{v} + q_{N,\lambda}(u,z),
        \end{equation}
        where $q_{N,\lambda}(u,z) = \bigO\left(  \frac{\lambda\sqrt{u}}{\eta^{5/2}\sqrt{N}}\right)$. 
    \end{proposition}
We begin by stating two results we use in this proof. Note that we conveniently drop the dependence on $z$ for $r_{jj}^N(z)$, since we fix $z\in\C^+$ throughout and hence just use the notation $r_{jj}^N$.
\begin{fact}[Exponential Inequalities]
The following holds true for any real numbers $a,b \in \mathbb{R}$ and complex numbers $z_1, z_2 \in \C^+$.
\begin{align}\label{pap1-expineqimag}
    |e^{a\ota z_1} -\e^{a\ota z_2}| \leq |a\|z_1 - z_2| 
\end{align}
\begin{align}\label{pap1-expineqsum}
   |e^a -\e^b| \leq |a-b|e^{|a|+|b|}
\end{align}
\end{fact}
    \begin{proof}[Proof of Proposition \ref{pap1-analyticstep1}]
    For the resolvent of a matrix with zero diagonal, we have the relation
    \[r^N_{jj} = -\left(z + \sum_{k,l\neq j}\tilde{r}^{N-1}_{kl}a_{kj}a_{lj}\right)^{-1},\] 
    for any diagonal element $r_{jj}^N$ of the resolvent $\R_{\bA_N}(z)$, where $\tilde{r}^{N-1}_{kl}:=\tilde{r}^{N-1}_{kl}(z)$ are the entries of the resolvent of $\bA_{N-1}^{(j)}$ in $z\in\C^+$, which is the adjacency matrix with deleted $j^{\text{th}}$ row and column. Plugging into \eqref{pap1-exponentialBesselidentity} yields
    \begin{align}\label{pap1-exponentialformulaonresolvent}
        \e^{\ota u r^N_{jj}} = 1 - \Khorunzkyintegral\e^{\ota vz}\exp\left(\ota v \sum_{k,l\neq j}\tilde{r}^{N-1}_{kl}a_{kj}a_{lj} \right) \dd{v}.
    \end{align} 
Adding and subtracting the appropriate exponential to \eqref{pap1-exponentialformulaonresolvent} yields
\begin{align}\label{pap1-expBesselcentered}
\e^{\ota ur^N_{jj}} = 1 - \Khorunzkyintegral\e^{\ota vz}\exp\left(\ota v \sum_{k\neq j}\tilde{r}^{N-1}_{kk}a_{kj}^2 \right) \dd{v} + E_1, \end{align}
where $E_1$ is an error term given by 
\begin{align*}
    E_1 = \Khorunzkyintegral\e^{\ota vz}\left(\exp\left(\ota v \sum_{k,l\neq j}\tilde{r}^{N-1}_{kl}a_{kj}a_{lj} \right) - \exp\left(\ota v \sum_{k\neq j}\tilde{r}^{N-1}_{kk}a_{kj}^2 \right)\right) \dd{v}.
\end{align*}It is easy to see that for $z\in\C^+$ with $\Re(z) = \zeta \in \mathbb{R}$ and $\Im(z)=\eta>0$, we have $|\e^{\ota vz}| = |\e^{\ota\zeta v}\e^{-\eta v}| \leq \e^{-\eta v}$. Thus, 
\begin{align}\label{pap1-E1}
\begin{split}
    |E_1| &= \left| \Khorunzkyintegral\e^{\ota vz}\left(\exp\left(\ota v \sum_{k,l\neq j}\tilde{r}^{N-1}_{kl}a_{kj}a_{lj} \right) - \exp\left(\ota v \sum_{k\neq j}\tilde{r}^{N-1}_{kk}a_{kj}^2 \right)\right) \dd{v} \right|\\
    &\leq \sqrt{u}\int_0^{\infty} \frac{v\e^{-\eta v}}{\sqrt{v}}  \sum_{k\leq N}\sum_{l\neq k}|\tilde{r}^{N-1}_{kl}|a_{kj}a_{lj}  \dd{v} = \left(  \sqrt{u}\int_0^{\infty} \sqrt{v}\e^{-\eta v}\dd{v}\right)\sum_{k\leq N}\sum_{l\neq k}|\tilde{r}^{N-1}_{kl}|a_{kj}a_{lj} 
    \end{split}
    \end{align}
where in the last step, we use inequality \eqref{pap1-expineqimag} and the bound $|\J(x)|\leq 1$ for $x\geq 0$. Note that in the last sum in \eqref{pap1-E1}, the entries $a_{kj}$ and $a_{lj}$ are independent of one another, and of $\tilde{r}_{kl}^{N-1}$. Thus, since $f$ is bounded by a constant $C_f$, taking expectation on the summation gives us 
\begin{align}\label{pap1-E1sum}
    \E\left[  \sum_{l\neq k}|\tilde{r}^{N-1}_{kl}|a_{kj}a_{lj}  \right] \leq \frac{\lambda C_f^2}{N^2}\sum_{l\neq k}|\tilde{r}^{N-1}_{kl}|
\end{align}
since $a_{ij}$ are distributed as Bernoulli random variables with parameter $p_{ij}$, and are scaled by a factor $\lambda^{-1/2}$. 
Using \eqref{pap1-E1sum} and taking expectation in \eqref{pap1-E1} gives us

\begin{align*}
\begin{split}
   \E \left[|E_1|\right]&\leq C_f^2\sqrt{u}\int_0^{\infty} \frac{\sqrt{v}\e^{-\eta v}\lambda}{N^2}  \sum_{k\leq N}\sum_{l\neq k}|\tilde{r}^{N-1}_{kl}|  \dd{v} \\
    &\leq C_f^2\sqrt{u}\int_0^{\infty} \frac{\sqrt{v}\e^{-\eta v}\lambda}{N\sqrt{N}}  \sum_{k\leq N}\left(\sum_{l\neq k}|\tilde{r}^{N-1}_{kl}|^2\right)^{\frac{1}{2}}  \dd{v} \hspace{0.3cm} \text{(Cauchy-Schwarz)}\\
     &\leq C_f^2\sqrt{u}\int_0^{\infty} \frac{\sqrt{v}\e^{-\eta v}\lambda}{N\sqrt{N\eta}}  \sum_{k\leq N}(\Im(\tilde{r}^{N-1}_{kk}))^{\frac{1}{2}}  \dd{v} \hspace{0.2cm} \text{(using Lemma \ref{pap1-Wald})}\\
    &\leq C_f^2\sqrt{u}\int_0^{\infty} \frac{\sqrt{v}\e^{-\eta v}\lambda}{\sqrt{N}\eta}\dd{v} \hspace{0.3cm} \text{(using property (iv) from Lemma \ref{resolventproperties})}   \\
    &= C_f^2\frac{\sqrt{u}\lambda}{\eta^{5/2}\sqrt{N}}\int_0^{\infty}\sqrt{\eta v}\e^{-\eta v}\dd{(\eta v)} = \bigO\left(  \frac{\lambda\sqrt{u}}{\eta^{5/2}\sqrt{N}}\right),
    \end{split}
\end{align*}
where in the last step we do a change of variable $\eta v = v'$ to show the integral is finite. So, if we now take an expectation in \eqref{pap1-expBesselcentered}, we get 
\begin{equation}\label{pap1-expBesselexpectation}
\E[\e^{\ota ur^N_{jj}} ] = 1 - \Khorunzkyintegral\e^{\ota vz}\E\left[\exp\left(\ota v \sum_{k\neq j}\tilde{r}^{N-1}_{kk}a_{kj}^2 \right) \right]\dd{v} + q_{N,\lambda}(u,z),
\end{equation}
where $q_{N,\lambda}(u,z) = \bigO\left(  \frac{\lambda\sqrt{u}}{\eta^{5/2}\sqrt{N}}\right)$. Note that the expectation could be pulled inside the integral in \eqref{pap1-expBesselcentered} using Fubini's Theorem since the integral is bounded above by a constant. To evaluate the expectation inside \eqref{pap1-expBesselexpectation}, we use a conditioning argument as follows. We have 
\begin{align*}
\E\left[\exp\left(\ota v \sum_{k\neq j}\tilde{r}^{N-1}_{kk}a_{kj}^2 \right)\right] &= \E\left[   \E\left[\exp\left(\left. \ota v \sum_{k\neq j}\tilde{r}^{N-1}_{kk}a_{kj}^2 \right) \right|\bA_{N-1}^{(j)}\right]\right] \\
&= \E\left[\prod_{k=1}^N\left( 1 - \frac{\lambda}{N}f(w_k,w_j) + \frac{\lambda}{N}f(w_k,w_j)\e^{\ota v \tilde{r}^{N-1}_{kk}/\lambda}\right)   \right]\\
&= \E\left[\prod_{k=1}^N\left( 1 +\frac{\lambda}{N}f(w_k,w_j)\left( \e^{\ota v \tilde{r}^{N-1}_{kk}/\lambda} - 1\right)    \right)   \right]\\
&= \E\left[  \prod_{k=1}^N \left(\exp\left( \frac{\lambda}{N}f(w_k,w_j)\left( \e^{\ota v \tilde{r}^{N-1}_{kk}/\lambda} - 1\right)  \right)  + q'_k(N,\lambda) \right)\right], \numberthis\label{pap1-condexplaststep}
\end{align*}
where $q'_k(N,\lambda)$ is an error given by 
\begin{align*}
    q'_k(N,\lambda) &= 1 + \frac{\lambda}{N}f(w_k,w_j)\left( \e^{\ota v \tilde{r}^{N-1}_{kk}/\lambda} - 1\right) -\exp\left( \frac{\lambda}{N}f(w_k,w_j)\left( \e^{\ota v \tilde{r}^{N-1}_{kk}/\lambda} - 1\right)  \right) .
\end{align*} 
Since $|\e^{\ota v \tilde{r}^{N-1}_{kk}/\lambda} - 1| \leq 2$, doing a Taylor expansion for the exponential term in $q_k'(N,\lambda)$ gives us
\begin{align}\label{pap1-qk'error}
    |q'_k(N,\lambda)| \leq \frac{4C_f^2\lambda^2}{N^2} = \bigO\left( \frac{\lambda^2}{N^2} \right).
\end{align}
We can write 
\begin{align}\label{pap1-productleadingorder}
   \E\left[\exp\left(\ota v \sum_{k\neq j}\tilde{r}^{N-1}_{kk}a_{kj}^2 \right)\right] = \E\left[ \prod_{k=1}^N \left(\exp\left( \frac{\lambda}{N}f(w_k,w_j)\left( \e^{\ota v \tilde{r}^{N-1}_{kk}/\lambda} - 1\right)  \right)\right)\right] + \E[E_2],
\end{align}
where $E_2$ is an expression involving all the other terms of the product in \eqref{pap1-condexplaststep}. To get the order of $E_2$, we take a supremum over $k$ in \eqref{pap1-condexplaststep} and compute the binomial expansion of the form $(a+b)^N$ modulo the leading term $a^N$. In particular, since $|\e^{\ota v \tilde{r}^{N-1}_{kk}/\lambda} - 1| \leq 2$, and again using \eqref{pap1-qk'error}, we have  
\begin{align*}
    |E_2| \leq \sum_{j=1}^N {N\choose j}\left( \e^{\frac{2\lambda C_f}{N}}\right)^{N-j}\left(\frac{4C_f^2\lambda^2}{N^2}\right)^j,
\end{align*}
which for some constant $C_a>0$ and $N$ large enough further simplifies to 
\begin{align*}
    |E_2| \leq C_a\sum_{j=1}^N (2C_f\lambda)^{2j}N^j \e^{-\frac{2j\lambda C_f}{N}}N^{-2j} = C_a\sum_{j=1}^N (2C_f\lambda)^{2j}N^{-j}\e^{-\frac{2j \lambda C_f}{N}} =  C_a\frac{4C_f\lambda^2N^{-1}\e^{-\frac{2\lambda C_f}{N}}}{1 - 4C_f\lambda^2N^{-1}\e^{-\frac{2\lambda C_f}{N}} }.
\end{align*}
where the last equality is due to the sum being a geometric series. Thus, 
\begin{align}\label{pap1-e2error}
    |E_2|  = \bigO\left(   \frac{\lambda^2}{N} \right),
\end{align}
which is a faster error than $q_{N,\lambda}(u,z)$ so we can later absorb it into the existing error of \eqref{pap1-expBesselexpectation}. Thus, using \eqref{pap1-e2error}, we can rewrite \eqref{pap1-productleadingorder} as
\begin{align}\label{pap1-finalexpectationinsideintegral}
\E\left[\exp\left(\ota v \sum_{k\neq j}\tilde{r}^{N-1}_{kk}a_{kj}^2 \right) \right]= \E\left[ \e^{-\lambda{d_j}}\exp\left( \lambda\tilde{g}_{N-1}\left(w_j,\frac{v}{\lambda},z\right)  \right)    \right] + \bigO\left(   \frac{\lambda^2}{N} \right)
\end{align}
where 
\begin{align}\label{pap1-eq:d_jdefn}
d_j = \frac{1}{N}\sum\limits_{k=1}^Nf(w_j,w_k) \hspace{0.3cm} \text{and} \hspace{0.3cm} \tilde{g}_{N-1}(w_j,b,z) = \sum\limits_{k=1}^N f(w_j,w_k)\e^{\ota b \tilde{r}^{N-1}_{kk}}.
\end{align} Note that $\tilde{g}_N$ is a bounded function and is bounded above by $C_f$. To get the error down from the exponent, we again use inequality \eqref{pap1-expineqsum}. 

To conclude the proof of the proposition, we need to return back to an expression involving terms of the form $r_{kk}^N$ of the original resolvent. To do so, we do an interpolation argument. Let $0\leq t\leq 1$ and define $\bA_N^t = (1-t)\bA_N + t\bA_{N-1}^{(j)}$ with the resolvent $\R_{\bA_N^t}(z)$, whose entries we denote by $r^N_{kl}(t)$, that also implicitly depends on $z$ but we drop that for convenience of notation. Also, define 
\[
\mathbf{g}_N^t(w_j,b,z) = \frac{1}{N}\sum_{i=1}^Nf(w_i,w_j)\e^{\ota b r^N_{kk}(t) }.
\]
We remark using property (i) from Lemma \ref{resolventproperties} that $\mathbf{g}_N^t$ is also bounded above by $C_f$ for all values of $t$, since the complex exponential $\e^{\ota br_{kk}^N(t)}$ is bounded by 1 for any $b\geq 0$ and $1\leq k\leq N$. In particular, we have that $|g_N(x,b,z)|\leq C_f$ for all $x,b\geq 0$.

Our target function is $g_N(w_j,b,z) = \frac{1}{N}\sum\limits_{i=1}^Nf(w_i,w_j)\e^{\ota b r^N_{kk}}$. By the fundamental theorem of calculus, 
\[
\begin{split}
    | g_N(w_j,b,z) - \tilde{g}_{N-1}(w_j,b,z) | &= \left|\mathbf{g}_N^0(w_j,b,z) - \mathbf{g}_N^1(w_j,b,z) \right| \\
    &= \left|  \int_0^1 \frac{\partial}{\partial t}\mathbf{g}_N^t \dd{t} \right| = \left|  \int_0^1 \frac{b}{N}\sum_{k=1}^N\e^{\ota b r^N_{kk}(t)}\frac{\partial}{\partial t}r^N_{kk}(t)  \right|.
\end{split}
\]
Now, $\R_{\bA_N^t}(z) = (\bA_N^t-zI)^{-1}$ and thus, $\frac{\dd}{\dd{t}}\R_{\bA_N^t}(z) = -\R_{\bA_N^t}(z)\frac{\dd{\bA_N^t}}{\dd{t}}\R_{\bA_N^t}(z)$. 

Note that $\frac{\dd{\bA_N^t}}{\dd{t}} = -J_N$, where $J_N$ is given by 
\begin{align*}
J_N(k,l) = \begin{cases}
0, & \text{if $k,l\neq j$} \\
a_{kl}, &\text{if $k=j$ or $l=j$}.
\end{cases}    
\end{align*}
Thus,
\begin{align*}
    &| g_N(w_j,b,z) - \tilde{g}_{N-1}(w_j,b,z)|\\
    &=\left|  \int_0^1 \frac{b}{N}\sum_{k=1}^N\e^{\ota b r^N_{kk}(t)}\sum_{m,n=1}^Nr^N_{km}(t)\frac{\partial a^t_{mn}}{\partial t}r^N_{nk}(t)  \right| \\
    &= \left|  \int_0^1 \frac{b}{N}\sum_{k=1}^N\e^{\ota b r^N_{kk}(t)}\sum_{m=1}^Nr^N_{km}(t)a_{mj}r^N_{jk}(t) \dd{t} \right|\\
    &\leq \int_0^1 \frac{b}{N}\sum_{k=1}^N\sum_{m=1}^N |r^N_{km}(t)a_{mj}r^N_{jk}(t)|\dd{t} \numberthis \label{pap1-eq:interpolationderivative}
    \end{align*}
since the complex exponential $\e^{\ota br_{kk}^N(t)}$ is trivially bounded by 1 as $r_{kk}^N(t)\in\C^+$. Then, using Cauchy-Schwarz and Lemma \ref{pap1-Wald} in \eqref{pap1-eq:interpolationderivative}, we have 
    \begin{align*}
     | g_N(w_j,b,z) - \tilde{g}_{N-1}(w_j,b,z)|\leq \int_0^1 \frac{b}{N}\sum_{k=1}^N|r_{jk}^N(t)|\left(\frac{\Im (r^N_{kk}(t))}{\eta}\right)^{1/2}\left(  \sum_{m=1}^N a_{mj}^2 \right)^{1/2} \dd{t}.
    \end{align*}
Bounding $\Im(r_{kk}^N(t))$ by $1/\eta$ (property (iv) of Lemma \ref{resolventproperties}) and taking expectation, we get
\begin{align}\label{pap1-eq:interpolationbeforeCauchySchwarz}
        \E[| g_N(w_j,b,z) - \tilde{g}_{N-1}(w_j,b,z) | ]\leq \int_0^1 \frac{b}{N\eta}\E\left[\sum_{k=1}^N|r_{jk}^N(t)|\left(\sum_{m=1}^N a_{mj}^2\right)^{1/2}\right] \dd{t}.
\end{align}
Now, again using Cauchy-Schwarz and Lemma \ref{pap1-Wald}, we have for some constant $C'$ that
 \begin{align}\label{pap1-offdiagCS}
        \sum_{k=1}^N|r_{jk}^N(t)| \leq \sqrt{N}\left(\sum_{k=1}^N|r_{jk}^N(t)|^2\right)^{1/2} \leq C'\frac{\sqrt{N}}{\sqrt{\eta}}.
 \end{align}
Thus, using \eqref{pap1-offdiagCS} and Jensen's inequality on the function $\sqrt{X}$ in \eqref{pap1-eq:interpolationbeforeCauchySchwarz}, we get
    \begin{align*}
    \E[| g_N(w_j,b,z) - \tilde{g}_{N-1}(w_j,b,z) | ]&\leq \int_0^1 \frac{b}{N\eta}\E\left[\frac{C'\sqrt{N}}{\sqrt{\eta}}\left(\sum_{m=1}^N a_{mj}^2\right)^{1/2}\right] \dd{t}\\
    & C'\int_0^1 \frac{b}{\sqrt{N}\eta^{3/2}}\left(\E\left[\sum_{m=1}^N a_{mj}^2\right]\right)^{1/2} \dd{t}.
\end{align*}
Since $f$ is bounded, we have for some new constant $C_f'$ that 
\begin{align*}
     \E[| g_N(w_j,b,z) - \tilde{g}_{N-1}(w_j,b,z) |] \leq \frac{C_f'b\sqrt{\lambda}}{\eta^{3/2}\sqrt{N}}.
\end{align*}
Using the fact that $\mathbf{g}_N^t$ is bounded by $C_f$ for all $t$, we get  
\[
\E[|\e^{\lambda \tilde{g}_{N-1}} - \e^{\lambda g_N}|] \leq \E[|\tilde{g}_{N-1} - g_N|]\e^{2C_f\lambda} = \bigO\left(  \frac{\sqrt{\lambda}}{\eta^{3/2}\sqrt{N}} \right).
\]
Since this is an error of the same order as $q_{N,\lambda}(u,z)$, we can absorb it into the existing error $q_{N,\lambda}$. Finally, using \eqref{pap1-finalexpectationinsideintegral} and the interpolation argument allows us to write \eqref{pap1-expBesselexpectation} as
\[
\E[\e^{\ota u r^N_{jj}}] = 1 - \e^{-\lambda d_j}\Khorunzkyintegral\e^{\ota vz}\E\left[\e^{\lambda g_N\left(w_j,\frac{v}{\lambda},z\right)}\right] \dd{v} + q_{N,\lambda}(u,z),
\]
which proves the proposition.
\end{proof}
Now, consider the expression \eqref{pap1-expBesselinterpolated} from the Proposition \ref{pap1-analyticstep1}. If we multiply throughout by $f(x,w_j)$ and then sum over $j$, and finally scale by $N$, we get 
\begin{align}\label{pap1-initialrecursion}
\begin{split}
&\E[ g_N(x,u,z)] \\ 
&= \frac{1}{N}\sum_{j=1}^Nf(x,w_j) - \frac{1}{N}\sum_{j=1}^Nf(x,w_j)\e^{-\lambda d_j}\Khorunzkyintegral\e^{\ota vz}\E\left[\e^{\lambda g_N\left(w_j,\frac{v}{\lambda},z\right)}\right] \dd{v} + q_{N,\lambda}(u,z).
\end{split}
\end{align}
Consider the space of Lipschitz functions $Lip(\mathbb{R})$ defined as 
\begin{align*}
    \mathrm{Lip}(\mathbb{R}) = \left\{ h\in C_b(\mathbb{R}) : \sup_x |h(x)| \leq 1, \sup_{x\neq y}\frac{|h(x)-h(y)|}{|x-y|} \leq C_L, 0< C_L< \infty \right\}.
\end{align*}
Now, under the bounded Lipshitz metric $d_{BL}(\cdot,\cdot)$ given by 
\[
d_{BL}(\mu,\nu) = \sup\limits_{h \in \mathrm{Lip}(\mathbb{R})}\left\{ \left|\int h\dd{\mu} - \int h\dd{\nu}\right| \right\},
\]
we have 
\[
\mu_{W_N} \implies \mu_w \hspace{0.3cm} \text{if and only if} \hspace{0.3cm} d_{BL}(\mu_{W_N},\mu_w) \to 0,
\]
where $W_N = w_{o_N}$ for a uniformly chosen vertex $o_N$. So, taking $f$ to be Lipschitz in one coordinate (and since we already have that $f$ is bounded), the first term in the RHS of \eqref{pap1-initialrecursion} becomes 
\begin{align}\label{pap1-eq:lipschitzerrordf}
\frac{1}{N}\sum_{j=1}^Nf(x,w_j) = \int f(x,y) \mu_{W_N}(\dd{y}) \leq  d_f(x) + E_N,  
\end{align}
where $E_N = d_{BL}(\mu_{W_N},\mu_w)$.

Recall from \eqref{pap1-eq:degree} that we have 
\[
d_f(w_j) := \int f(x,w_j) \mu_w(\dd{x}).
\]
Then, one simply gets 
\begin{equation}\label{pap1-eq:expLipschtiz}
|\e^{-\lambda d_j} - \e^{-\lambda d_f(w_j)}| \leq \lambda E_N \e^{2\lambda}.
\end{equation}
Thus, using \eqref{pap1-eq:lipschitzerrordf} and \eqref{pap1-eq:expLipschtiz} in \eqref{pap1-expBesselinterpolated} gives us
\begin{align}\label{pap1-eq:initalgN}
\begin{split}
&\E[ g_N(x,u,z)]\\
&= d_f(x) - \frac{1}{N}\sum_{j=1}^Nf(x,w_j)\e^{-\lambda d_f(w_j)}\left(\Khorunzkyintegral\e^{\ota vz}\E\left[\e^{\lambda g_N\left(w_j,\frac{v}{\lambda},z\right)}\right] \dd{v}\right) \hspace{0.1cm} + \tilde{q}_{N,\lambda}(u,z),
\end{split}
\end{align}
where \[
\tilde{q}_{N,\lambda}(u,z) = q_{N,\lambda}(u,z) + \bigO(E_N).
\]
Finally, for a fixed $x\in[0,\infty)$, define
\[
\mathbf{I}_g(y) = f(x,y)\e^{-\lambda d_f(y)}\left(\Khorunzkyintegral\e^{\ota vz}\E\left[\e^{\lambda g_N\left(y,\frac{v}{\lambda},z\right)}\right] \dd{v}\right).
\]
Then, we have the following lemma. 
\begin{lemma}\label{pap1-gyLipschitz}
    $\mathbf{I}_g(y)$ is Lipschitz. 
\end{lemma}
\begin{proof}
Consider $\mathbf{I}_g(y)$ as defined. Then, 
\begin{equation}\label{pap1-Leibnizruleverification}
\begin{split}
    &\left| \partial_y\mathbf{I}_g(y)  \right| \\ &\leq \left| \partial_yf(x,y)\e^{-\lambda d_f(y)}\left(\Khorunzkyintegral\e^{\ota vz}\E\left[\e^{\lambda g_N\left(y,\frac{v}{\lambda},z\right)}\right] \dd{v}\right) \right| \\
    &+ \left|  f(x,y)\e^{-\lambda d_f(y)}\partial_yd_f(y)\left(\Khorunzkyintegral\e^{\ota vz}\E\left[\e^{\lambda g_N\left(y,\frac{v}{\lambda},z\right)}\right] \dd{v}\right) \right| \\
    &+ \left|   f(x,y)\e^{-\lambda d_f(y)}\left(\Khorunzkyintegral\e^{\ota vz}\E\left[\e^{\lambda g_N\left(y,\frac{v}{\lambda},z\right)}\right]\partial_yg_N(y,v/\lambda,z)  \dd{v}\right)   \right|.
\end{split}
\end{equation}
Recall that a function is Lipschitz if and only if it has a bounded derivative. Thus, if $f$ is Lipschitz in $y$, the first term in \eqref{pap1-Leibnizruleverification} uniformly bounded in $y$. Moreover, this makes the second term in \eqref{pap1-Leibnizruleverification} bounded as well since 
\begin{align}\label{pap1-eq:Leibnizruleonf}
\left|\partial_yd_f(y) \right| \leq  \int_0^{\infty} | \partial_yf(x,y)|\mu_w(\dd{x}) 
\end{align}
is bounded. To justify interchanging the derivative and the integral in \eqref{pap1-eq:Leibnizruleonf}, we have to utilise Theorem \ref{pap1-Leibnizrulethm} for which we need to verify the following conditions. 
\begin{itemize}
    \item $f(x,y)$ is $\mu_w-$integrable for each $y$ and the map $y\mapsto f(x,y)$ is continuous for each $x$. 
    \item For each $x$, the derivative $\partial_yf(x,y)$ exists. 
    \item For each $y$, there is a $\mu_w-$integrable function $\Psi_y(x)$ and a neighbourhood $U_y$ containing $y$, such that for all $y'\in U_y$, $|\partial_{y'} f(x,y')|\leq \Psi_y(x)$.
\end{itemize}
The first and second are trivial to check, and by Lipschitz property, since $\partial_{y} f(x,y) \equiv const.$, we have $\Psi_y(x) \equiv const$, which is integrable on $[0,\infty)$ since $\mu_w$ is a probability measure. 

Finally, for notational convenience, let $h(y,v)$ be denote 
\[
h(y,v) = \frac{\J(2\sqrt{uv})}{\sqrt{v}}\e^{\ota vz}\E\left[\e^{\lambda g_N(y,v,z)}\right].
\]
Once again, we need to verify the three conditions as above so as to apply Theorem \ref{pap1-Leibnizrulethm}. Note that $h(y,v)$ is integrable with respect to $v$. Moreover, 
\[
\partial_yh(y,v) = h(y,v)\partial_y g_N(y,v,z)
\]
where one can compute 
\[
\partial_yg_N(y,v,z) = \frac{1}{N}\sum_{k=1}^N\partial_yf(w_k,y)\e^{\ota vr_{kk}},
\]
which again is bounded. Thus, $\partial_yh(y,v)$ exists, and is bounded above by $C_0 v^{-\frac{1}{2}}\e^{-\eta v}$, which is integrable with respect to $v$.  This verifies the three conditions and allows us to pull the derivative inside the third term in \eqref{pap1-Leibnizruleverification}, and also makes that term bounded. Thus, $\mathbf{I}_g(y)$ is Lipschitz.
\end{proof}
Since $\mathbf{I}_g(y)$ is Lipschitz, we can exploit the weak convergence of $\mu_w$ under the Lipschitz metric $d_{BL}$ in \eqref{pap1-eq:initalgN} to give us
\begin{equation}\label{pap1-initialrecurison}
\begin{split}
&\E[ g_N(x,u,z)]\\
&= d_f(x) - \int_0^{\infty}f(x,y)\e^{-\lambda d_f(y)}\left(\Khorunzkyintegral\e^{\ota vz}\E\left[\e^{\lambda g_N\left(y,\frac{v}{\lambda},z\right)}\right] \dd{v}\right)\mu_w(\dd{y}) \hspace{0.1cm} + \tilde{q}_{N,\lambda}(u,z).
\end{split}
\end{equation}
Recall the Banach space as defined in \eqref{pap1-Banachdefn}, and consider $\phi\in(\Ba,\|\cdot\|)$. In this space, consider the map 
\begin{equation}\label{pap1-contractionmap}
F_z(\phi)(x,u) = d_f(x) - \sqrt{u}\int_0^{\infty}f(x,y)\e^{-\lambda d_f(y)}\left(\Khorunzkyintegral\e^{\ota vz}\e^{\lambda \phi\left(y,\frac{v}{\lambda},z\right)} \dd{v}\right)\mu_w(\dd{y}).
\end{equation}
Note that $\phi$ also implicitly depends on $z$ but we drop that for notational purposes since we fix $z$ throughout.  

Take $\phi_1, \phi_2 \in (\mathcal{B},\|\cdot\|)$ such that $\|\phi_1\|, \|\phi_2\| \leq C_f$. Then, using the norm we defined in \eqref{pap1-Banachdefn} and inequality \ref{pap1-expineqsum}, from \eqref{pap1-contractionmap} we get
\[
\begin{split}
    &\|F_z(\phi_1) - F_z(\phi_2)\| \\
    &\leq \sup_{x,u\geq 0}\sqrt{\frac{1}{1+u}}\left|\int_0^{\infty}f(x,y)\e^{-\lambda d_f(y)}\left(\Khorunzkyintegral\e^{\ota vz}\left(\e^{\lambda \phi_1\left(y,\frac{v}{\lambda}\right)}-\e^{\lambda \phi_2\left(y,\frac{v}{\lambda}\right)}\right)  \dd{v}\right)\mu_w(\dd{y})\right|\\
    &\leq \sup_{u\geq 0}\sqrt{\frac{1}{1+u}}\int_0^{\infty}\int_0^{\infty}\frac{\lambda}{\sqrt{v}}\e^{-\eta v}\left|\phi_1\left(y,\frac{v}{\lambda}\right) - \phi_2\left(y,\frac{v}{\lambda}\right) \right|\e^{\lambda\left| \phi_1\left(y,\frac{v}{\lambda}\right)  \right| + \lambda\left| \phi_2\left(y,\frac{v}{\lambda}\right)  \right|}\dd{v}\hspace{0.2cm}\mu_w(\dd{y})\\
    &\leq \lambda\|\phi_1 - \phi_2\|\int_0^{\infty}\int_0^{\infty}\frac{\lambda}{\sqrt{v}}\e^{-\eta v}\sup_{y,v\geq0}\frac{\sqrt{1+v/\lambda}}{\sqrt{1+v/\lambda}}\e^{\lambda\left| \phi_1\left(y,\frac{v}{\lambda}\right)  \right| + \lambda\left| \phi_2\left(y,\frac{v}{\lambda}\right)  \right|}\dd{v}\hspace{0.2cm}\mu_w(\dd{y})\\
    &\leq \lambda\|\phi_1 - \phi_2\|\int_0^{\infty}\int_0^{\infty}\frac{\sqrt{1+v/\lambda}}{\sqrt{v}}\e^{-\eta v}\exp\left(\lambda\sqrt{1 + v/\lambda}(\|\phi_1\| + \|\phi_2\|)\right)\dd{v}\hspace{0.2cm}\mu_w(\dd{y})\\
    &\leq \|\phi_1 -\phi_2\|\int_0^{\infty}\left(\frac{\e^{-\eta v}}{\sqrt{v}}+\frac{\e^{-\eta v}}{\sqrt{\lambda}}\right)\e^{2C_f\sqrt{\lambda  v}}\dd{v}\leq \frac{C_{1}}{\eta^{5/2}}\|\phi_1 - \phi_2\|,
\end{split}
\]
where $C_{1}$ is the constant upper bound to the integral of the form \begin{align*}
    \int_0^{\infty} c_1\e^{-c_2x + c_3\sqrt{x}}\dd{x},
\end{align*}
which is finite. Taking $\eta >0$ sufficiently large, we get that $F_z$ is a contraction in an open ball $B\subset\mathcal{B}$ of radius $C_f<\infty$, and thus, by the Banach Fixed Point Theorem, there exists a unique $\phi^*$ such that $\phi^* = F_z(\phi^*)$ for $F_z:B\to B$.

\

We are now ready to prove a concentration result. Recall the function $G_N(u)$ defined in \eqref{pap1-eq:G_Ndefn} as 
\begin{align*}
    G_N(u) = \frac{1}{N}\sum_{i=1}^N\e^{\ota u r_{ii}^N}.
\end{align*}
If we now define a new function $\tilde{G}_N(x,u)$ that acts identically on the first coordinate as 
\[
\tilde{G}_N(x,u) := G_N(u),
\]
then one can see that $\sup_{x,u} \frac{1}{\sqrt{1+u}}\tilde{G}_N(x,u) < \infty$, and so $\tilde{G}_N(x,u) \in \Ba$, and consequently, a concentration result for $\tilde{G}_N$ would imply concentration for $G_N$.
\begin{proposition}[Concentration and convergence]\label{pap1-Banachconcentration}
For any $z\in\C^+$ and $x\in[0,\infty)$, and uniformly over $u$ in $[0,1]$, we have $\E[ g_N(x,u,z)] \xrightarrow{N\to\infty} \phi^*(x,u)$. Further, we have 
\begin{align*}
\E\left[ \left\| g_N- \E [g_N]  \right\|^2   \right] &= o(1), \hspace{0.3cm} \text{and}\\
\E\left[ \left\| \tilde{G}_N- \E [\tilde{G}_N ]  \right\|^2   \right] &= o(1).
\end{align*}
\end{proposition} 

\begin{proof}[Proof of Proposition \ref{pap1-Banachconcentration}]
Let $\delta_N(x,u,z)$ denote the error 
\[
\delta_N(x,u,z) := \e^{\lambda g_N(x,u,z)} - \e^{\lambda \E{[ g_N(x,u,z)]}}.
\]
Let $1\leq k\neq l \leq N$ and consider the covariance
\[
A_{k,l} :=\E[\e^{\ota ur_{kk}^N}\e^{\ota ur_{ll}^N}] - \E[\e^{\ota ur_{kk}^N}]\E[\e^{\ota ur_{kk}^N}].
\]
Using \eqref{pap1-expBesselcentered} for the first term and Proposition \ref{pap1-analyticstep1} for the second term, we get 
\begin{align}\label{pap1-eq:doubleintegral}
\begin{split}
    &A_{k,l} = \\
    &-\E [T_j] - \E[ T_k] + u\int\int \frac{\J(2\sqrt{uv_1})}{\sqrt{v_1}}\frac{\J(2\sqrt{uv_2})}{\sqrt{v_2}}\e^{\ota(v_1+v_2)z}\E\left[\e^{\ota v_1\sum\limits_{l\neq j}\tilde{r}_{ll}^{N-1}a_{jl}^2 + \ota v_2\sum\limits_{l\neq k}\tilde{r}_{ll}^{N-1}a_{kl}^2  }\right]\dd{v_1}\dd{v_2}\\
    &+ \E[\tilde{T}_j] + \E[\tilde{T}_k] - u\int\int \frac{\J(2\sqrt{uv_1})}{\sqrt{v_1}}\frac{\J(2\sqrt{uv_2})}{\sqrt{v_2}}\e^{\ota(v_1+v_2)z}\E\left[ \e^{\lambda g_N(w_j,\frac{v_1}{\lambda},z)+\lambda g_N(w_k,\frac{v_2}{\lambda},z)} \right]{}\dd{v_1}\dd{v_2},
\end{split} 
\end{align}
where $T_i$ and $\tilde{T}_i$ are the RHS of equations \eqref{pap1-expBesselcentered} and \eqref{pap1-expBesselinterpolated} respectively, and differ by the error $q_{N,\lambda}(u,z)$ in expectation. In the first double integral of \eqref{pap1-eq:doubleintegral}, one can do the interpolation argument term-wise, and obtain the error $C_Iq_{N,\lambda}^2(u,z) + q_{N,\lambda}^2(u,z)$ by making a difference with the second double integral in \eqref{pap1-eq:doubleintegral}, where $C_I$ is the constant upper bound to $\tilde{T}_k$ for any $k$. Thus, we have that 
\begin{align}\label{pap1-eq:Aklerror}
    |A_{k,l}| \leq C_I'q_{N,\lambda}(u,z) + q_{N,\lambda}^2(u,z).
\end{align}
Using inequality \ref{pap1-expineqsum} on $\delta_N(x,u,z)$ gives us
\begin{align*}
    \E[|\delta_N(x,u,z)|^2] = \E \left[ \left| \e^{\lambda g_N(x,u,z)}  - \e^{\lambda \E\left[g_N(x,u,z)\right]}  \right|^2  \right]
\leq C_1\E\left[ |g_N(x,u,z) - \E[ g_N(x,u,z)] |^2 \right].
\end{align*}
since $|g_N(x,v,z)|\leq C_f$ and $C_1 = \e^{2\lambda C_f}$. We can now bound this by using the definition of $g_N$ to get 
\begin{align}\label{pap1-eq:deltaNexpansion}
   \E[|\delta_N(x,u,z)|^2]  &\leq  \frac{C_1}{N^2}\left|\sum_{k,l=1}^N\E[f(x,w_k)\e^{\ota ur^N_{kk}}f(x,w_l)\e^{\ota ur^N_{ll}}] - \E[f(x,w_k)\e^{\ota ur^N_{kk}}]\E[f(x,w_l)\e^{\ota ur^N_{ll}}]\right|.
\end{align}
Since $f$ is deterministic, we can pull it out of the expectation and take it common, giving us 
\begin{align*}
    \E[|\delta_N(x,u,z)|^2]  &\leq  \frac{C_1}{N^2}\left|\sum_{k,l=1}^N f(x,w_k)f(x,w_l) A_{k,l} \right|.
\end{align*}
We can conclude using triangle inequality that 
\begin{align}\label{pap1-eq:deltaNerror}
    \E[|\delta_N(x,u,z)|^2] \leq C_1C_f^2\sup_{k,l}|A_{k,l}| = \bigO\left(   \frac{\lambda\sqrt{u}}{\eta^{5/2}\sqrt{N}} \right).
\end{align}
For $\eta>0$ sufficiently large, taking the norm, we get 
\begin{equation}\label{pap1-gNexpconc}
\E\left[ \left\| \e^{\lambda g_N} - \e^{\lambda\E [g_N]}   \right\|^2   \right] = o(1).
\end{equation}
However, $\delta_N $ is a bounded analytic function in $[0,\infty)^2\times\C^+$. Using the identity theorem from complex analysis , which states that if two holomorphic functions agree in an open set of the domain then they must agree everywhere on the domain, we have that since $\delta_N \to 0$ on an open set of the upper-half complex plane, it must approach 0 everywhere on the upper-half plane. Since the error in \eqref{pap1-eq:deltaNerror} can be absorbed in $\tilde{q}_{N,\lambda}(u,z)$, using \ref{pap1-expineqsum} gives us 
\begin{equation}\label{pap1-recursionafterconcentration}
\begin{split}
&\E[ g_N(x,u,z)]\\
&= d_f(x) - \int_0^{\infty}f(x,y)\e^{-\lambda d_f(y)}\left(\Khorunzkyintegral\e^{\ota vz}\e^{\lambda\E \left[g_N\left(y,\frac{v}{\lambda},z\right)\right]} \dd{v}\right)\mu_w(\dd{y}) \hspace{0.1cm} + \tilde{q}_{N,\lambda}(u,z),
\end{split}
\end{equation}
where the error vanishes in the norm as 
\[
\|\tilde{q}_{N,\lambda}\|  = \| q_{N,\lambda}(u,z) + \bigO(E_N)\| \leq \sup_{x,u\geq 0} \left| C\frac{\lambda\sqrt{u}}{\eta^{5/2}\sqrt{N}} \right| + E_N = o(1).
\]
Now, consider the function $\tilde{G}_N(x,u)$ and the error 
\[
\Delta_N(u) := \tilde{G}_N(x,u) - \E[\tilde{G}_N(x,u)].
\]
By definition of $\tilde{G}_N$, one can see that expanding $\Delta_N(u)$ will yield an expression similar to \eqref{pap1-eq:deltaNexpansion} modulo $f$, and so, using \eqref{pap1-eq:Aklerror} again, we get
that \[
\E[|\Delta_N|^2] \leq C_1C_f^2\sup_{k,l}|A_{k,l}| = \bigO\left(   \frac{\lambda\sqrt{u}}{\eta^{5/2}\sqrt{N}} \right).
\]
By taking the norm and again using identity theorem, we get that $\Delta_N$ vanishes in $[0,\infty)^2\times\C^+$ and thus 
\begin{equation}\label{pap1-GNconcentration}
\E\left[ \left\| \tilde{G}_N- \E [\tilde{G}_N]   \right\|^2   \right] =o(1).
\end{equation}
A quick inspection of \eqref{pap1-eq:deltaNexpansion} shows that in fact we also have the concentration for $g_N$, since the RHS is precisely the upper bound on 
\begin{align*}
    \E[|g_N(x,u,z) - \E[ g_N(x,u,z)]|^2], 
\end{align*} and so, 
\begin{equation}\label{pap1-gNconcentration}
\E\left[ \left\| g_N- \E[g_N]  \right\|^2   \right] = o(1).
\end{equation}
Finally, comparing \eqref{pap1-recursionafterconcentration} with the contraction mapping \eqref{pap1-contractionmap}, we have the following:\[
\begin{split}
\E[ g_N(x,u,z)] &= F_z(\E[ g_N(x,u,z)]) + \tilde{q}_{N,\lambda}(u,z),\\
\phi^*(x,u) &= F_z(\phi^*(x,u)). 
\end{split}
\]
So, with $\eta>0$ large enough and $F_z$ being a contraction on $B\subset \Ba$ of radius $C_f$, we have 
\begin{align*}
\|\E[g_N]- \phi^*\| \leq \|F_z(\E [g_N]) - F_z(\phi^*)\| + \|\tilde{q}_N\|,
\end{align*}
and consequently, 
\begin{align*}
    \frac{1}{2}\|\E[g_N]- \phi^*\| \leq \|\tilde{q}_N\|.
\end{align*}
Thus, since $\|\E g_N\| \leq C_f$,
\[
\|\E[g_N]- \phi^* \| \xrightarrow{N\to\infty} 0.
\]
As a quick remark, notice that  \begin{equation}\label{eq:phibdd}
    \|\phi^*\| \leq C_f, 
\end{equation} 
since $g_N$ is bounded. 

Now, since $\E[ g_N(x,u,z)]$ is an analytic function on $[0,\infty)^2\times\C^+$, we have $\lim_{N\to\infty}\E[ g_N(x,u,z)]$ is an analytic function. Again from the identity theorem of complex analysis, since $\lim_{N\to\infty}\E [g_N]$ and $\phi^*$ are analytic and agree on an open set of $[0,\infty)^2\times\C^+$, they agree everywhere in the complex domain $[0,\infty)^2\times\C^+$, and thus the convergence holds for any $z \in \C^+$. Note that \emph{for a fixed $z\in\C^+$}, although both the functionals $\E [g_N]$ and $\phi^*$ live in $(\Ba,\|\cdot\|_{\Ba})$, the domain of $\phi^*$ is $[0,\infty)^2\times\C^+$ since $\E g_N$ has the domain $[0,\infty)^2\times\C^+$. Now, for each $z\in\C^+$, fixing $u$ in the compact set $[0,1]$ gives us that for each $x\in[0,\infty)$ and uniformly over $u\in[0,1]$,
\begin{equation}\label{eq:g_N}
\E[ g_N(x,u,z)] \xrightarrow{N\to\infty} \phi^*(x,u) 
\end{equation}
\end{proof}
We can now prove Theorem \ref{pap1-Resolventtheorem}.
\begin{proof}[Proof of Theorem \ref{pap1-Resolventtheorem}]
\noindent Equation \eqref{pap1-GNconcentration} proves the concentration statement of Theorem \ref{pap1-Resolventtheorem}. 
Recall that we had shown that 
\[
\E\left[\e^{\ota u r^N_{jj}}\right] = 1 - \e^{-\lambda d_j}\Khorunzkyintegral\e^{\ota vz}\E\left[\e^{\lambda g_N\left(w_j,\frac{v}{\lambda},z\right)}\right] \dd{v} + q_{N,\lambda}(u,z),
\]
and so, 
\begin{align*}
\E [G_N(u,z)] &= \frac{1}{N}\sum_{j=1}^N\E[\e^{\ota u r^N_{jj}}] \\
&=1 - \frac{1}{N}\sum_{j=1}^N\e^{-\lambda d_j}\Khorunzkyintegral\e^{\ota vz}\E\left[\e^{\lambda g_N\left(w_j,\frac{v}{\lambda},z\right)}\right] \dd{v} + q_{N,\lambda}(u,z). \numberthis\label{pap1-eq:G_N}
\end{align*}
Next, we see that the function 
\begin{align*}
    \tilde{I}_g(y) = \e^{-\lambda d_f(y)}\Khorunzkyintegral \e^{\ota vz}\e^{\lambda \phi^*(y,v/\lambda)} \dd{v}
\end{align*}
is Lipschitz by using an argument similar to Lemma \ref{pap1-gyLipschitz}. Thus, we get
\begin{align*}
\E [G_N(u,z)]= 1 - \int_0^{\infty}\e^{-\lambda d_f(y)}\left(\Khorunzkyintegral\e^{\ota vz}\E\left[\e^{\lambda g_N\left(y,\frac{v}{\lambda},z\right)}\right] \dd{v}\right)\mu_w(\dd{y}) \hspace{0.1cm} + \tilde{q}_{N,\lambda}(u,z).
\end{align*}
Since from Proposition \ref{pap1-Banachconcentration} we have concentration for $g_N$, using inequality \eqref{pap1-expineqsum} we have that 
\[
\E [G_N(u,z)] =1 - \int_0^{\infty}\e^{-\lambda d_f(y)}\Khorunzkyintegral\e^{\ota vz}\e^{\lambda\E \left[g_N\left(y,\frac{v}{\lambda},z\right)\right]} \dd{v}\hspace{0.2cm}\mu_w(\dd{y}) + \tilde{q}_{N,\lambda}(u,z).
\]
Finally, taking the limit $N\to\infty$ gives us
\begin{equation}\label{eq:limGN}
\lim_{N\to\infty}\E [G_N(u,z)] = 1 -\int_0^{\infty}\e^{-\lambda d_f(y)}\Khorunzkyintegral \e^{\ota vz}\e^{\lambda \phi^*(y,v/\lambda)} \dd{v}\hspace{0.2cm}\mu_w(\dd{y}),
\end{equation}
completing the proof of Theorem \ref{pap1-Resolventtheorem}.
\end{proof}
\

\subsection{\bf Deriving the expression for the Stieltjes Transform} Since we took $u$ to be in $[0,1]$, we can take a derivative with respect to $u$ and evaluate it at $u=0$. Recall from equation \eqref{pap1-eq:G_N} that we have
\begin{align*}
\begin{split}
    \E [G_N(u,z)] &= \frac{1}{N}\E\sum_{j=1}^N\e^{\ota u r^N_{jj}} \\
    &= 1 - \frac{1}{N}\sum_{j=1}^N\e^{-\lambda d_j}\Khorunzkyintegral\e^{\ota vz}\E\left[\e^{\lambda g_N\left(w_j,\frac{v}{\lambda},z\right)}\right] \dd{v} + q_{N,\lambda}(u,z).
\end{split}
\end{align*}
Note that by definition, $G_N(u,z)$ is a bounded function, and thus by DCT, limit operations can be interchanged with expectation. We would like to take a derivative with respect to $u$ and evaluate at $u=0$ to extract out $\tr(\R_{\bA_N}(z))$ from the LHS of \eqref{pap1-eq:G_N}. On the other hand, we would first like to take $N\to\infty$ for the RHS to remove the error term. To interchange these operations, we have the following result. 
\begin{proposition}\label{pap1-MooreOsgoodstep}
    Both the limits $\lim_{N\to\infty} \left.\frac{\partial}{\partial u} \E [G_N(u,z)]\right|_{u=0}$ and $\left.\frac{\partial}{\partial u} \lim_{N\to\infty} \E [G_N(u,z)]\right|_{u=0}$ exist and are equal.
\end{proposition}
\begin{proof}
We fix a $z\in\C^+$. Now, $\lim_{N\to\infty} \E [G_N(u,z)]$ exists due to the RHS of \eqref{pap1-eq:G_N}, which we denote by $G(u,z)$. If we define $H_N(u,z)$ and $H(u,z)$ as 
\begin{align*}
H_N(u,z) &= \frac{\E [G_N(u,z)] - \E [G_N(0,z)]}{u},\\
 H(u,z) & = \frac{G(u,z)-G(0,z)}{u}.
\end{align*}
Then, 
\begin{align*}
\lim_{u\to 0}H_N(u,z) &=\left.\frac{\partial}{\partial u}\E [G_N(u,z)]\right|_{u=0},\\
\lim_{u\to 0}H(u,z) &=\left.\frac{\partial}{\partial u}G(u,z)\right|_{u=0}.
\end{align*}
We would like to claim 
\[
\lim_{N\to\infty}\left.\frac{\partial}{\partial u}\E [G_N(u,z)]\right|_{u=0} = \left.\frac{\partial}{\partial u}G(u,z)\right|_{u=0}.
\]
Thus, we want to interchange the order of limits. Note that 
\[
\lim_{N\to\infty}H_N(u,z) = H(u,z)
\]
uniformly in $u \in (0,1]$, and 
\[
\lim_{u\to 0} H_N(u,z) = \left.\frac{\partial}{\partial u}\E [G_N(u,z)]\right|_{u=0} = \E[\tr(\R_{\bA_N}(z))]
\]
for each $N$, where the limit can be taken inside the expectation using dominated convergence. Thus, using \citet[Theorem 7.11]{babyrudin}, we have that the limits $\lim_{u\to 0}H(u,z)$ and $\lim_{N\to\infty}\E[\tr(\R_{\bA_N}(z))]$ 
exist and are equal. 
\end{proof}

We are now ready to prove Corollary \ref{pap1-Stieltjescorollary}.
\begin{proof}[Proof of Corollary \ref{pap1-Stieltjescorollary}]
We now do precisely as we stated before Proposition \ref{pap1-MooreOsgoodstep}. We evaluate the derivative at $u=0$ and then take $N\to\infty$ on the LHS of \eqref{pap1-eq:G_N}, and we do the reverse for the RHS of \eqref{pap1-eq:G_N}. Note that since $\lim_{N\to\infty}\mu_{N, \lambda} = \mu_{\lambda}$ in probability, $\St_{\bA_N}(z)\to S_{\mu_{\lambda}}(z)$ and also $\bar\St_{\bA_N}(z)\to S_{\mu_{\lambda}}(z)$ as $N\to\infty$ for all $z\in \C^+$. 
Thus, we then obtain using Proposition \ref{pap1-MooreOsgoodstep}
\begin{align}\label{pap1-eq:derivativetobeevaluated}
\begin{split}
\ota\St_{\mu_{\lambda}}(z)=\ota \lim_{N\to\infty}\bar \St_{\bA_N}(z) &\overset{\eqref{eq:defstieltjes}}= \ota\lim_{N\to\infty}\E\tr(\R_{\bA_N}(z))\overset{\eqref{eq:derivativeST}} = \lim_{N\to\infty} \left.\frac{\partial}{\partial u} \E [G_N(u,z)]\right|_{u=0}\\
&= \left.\frac{\partial}{\partial u} \lim_{N\to\infty} \E [G_N(u,z)]\right|_{u=0}\\
&\overset{\eqref{eq:limGN}}=-\frac{\partial}{\partial u}\left. \int_0^{\infty}\e^{-\lambda d_f(y)}\Khorunzkyintegral\e^{\ota vz}\e^{\lambda \phi^*_z\left(y,\frac{v}{\lambda}\right)} \dd{v} \hspace{0.2cm}\mu_w(\dd{y})\right|_{u=0}\\
&=-\left. \int_0^{\infty}\e^{-\lambda d_f(y)}\frac{\partial}{\partial u}\Khorunzkyintegral\e^{\ota vz}\e^{\lambda \phi^*_z\left(y,\frac{v}{\lambda}\right)} \dd{v} \hspace{0.2cm}\mu_w(\dd{y})\right|_{u=0}.
\end{split}
\end{align}
We now wish to evaluate the derivative on the RHS of \eqref{pap1-eq:derivativetobeevaluated}. Let $K(u)$ denote 
\begin{equation}\label{pap1-integraltobediff}
K(u):= \Khorunzkyintegral\e^{\ota vz}\e^{\lambda \phi^*_z\left(y,\frac{v}{\lambda}\right)} \dd{v}.
\end{equation}
Observe that 
\begin{align}\label{pap1-eq:Besselfubini}
\sum_{k\geq 0}\int_0^{\infty} \frac{v^k}{k!(k+1)!}\e^{-\eta v}\dd{v} = \sum_{k\geq 0}\frac{\Gamma(k+1)}{k!(k+1)!\eta^k}\leq \e^{1/\eta}
\end{align}
for $\eta>0$ by a change of variables. If we expand the Bessel function as defined in \eqref{eq:bessel} in equation \eqref{pap1-integraltobediff} and take the absolute value, we observe using \eqref{pap1-eq:Besselfubini} and using  $|\phi^*(x,u)|\leq C_f$ (from \eqref{eq:phibdd}), that we can use Fubini's Theorem to interchange the integral with the summand. Thus, we have
\begin{align*}
    K(u) &= \Khorunzkyintegral\e^{\ota vz}\e^{\lambda \phi^*_z\left(y,\frac{v}{\lambda}\right)} \dd{v} \\
    &= \sqrt{u}\int_0^{\infty} \frac{1}{\sqrt{v}}\sum_{k=0}^{\infty}\frac{(-1)^k(\sqrt{uv})^{2k+1}}{k!(k+1)!}\e^{\ota vz}\e^{\lambda \phi^*_z\left(y,\frac{v}{\lambda}\right)} \dd{v}\\
    &= \sum_{k=0}^{\infty}\frac{(-1)^ku^{k+1}}{k!(k+1)!}\int_0^{\infty} v^k\e^{\ota vz}\e^{\lambda \phi^*_z\left(y,\frac{v}{\lambda}\right)} \dd{v}.
\end{align*}
Denote by $I_k(y)$ the integral $$I_k(y):=\int v^k\e^{\ota vz}\e^{\lambda \phi^*_z\left(y,\frac{v}{\lambda}\right)} \dd{v}. $$ Therefore, 
\begin{equation}\label{eq:ku}
\frac{K(u)}{u} =\sum_{k=0}^{\infty}\frac{(-1)^ku^k}{k!(k+1)!}I_k(y)= I_0(y) + \sum_{k\geq 1}\frac{(-1)^ku^k}{k!(k+1)!}I_k(y) =:I_0(y) +\sum_{k=1}^{\infty}a_k(u),
\end{equation}
where $a_k(u)$ denotes 
\[
a_k(u) := \frac{(-1)^ku^kI_k(y)}{k!(k+1)!}. 
\]
Note that for any $k$, we have that $I_k(y)$ is finite since 
\begin{align*}
    |I_k(y)| \leq \int_0^{\infty}v^k\e^{-\eta v}\e^{C_f\lambda} \dd{v} =\frac{\e^{C_f\lambda}}{\eta^{k+1}}\Gamma(k+1).
\end{align*}
Since $K(0)=0$ and by \eqref{eq:ku} it follows that 
\begin{equation}\label{eq:KIA}
   \left.\frac{\partial}{\partial u}K(u)\right|_{u=0} = \lim_{u\to 0}\frac{K(u)}{u}= I_0(y)+\lim_{u\to0}\sum_{k\geq 1}a_k(u), 
\end{equation}

Therefore we would like to evaluate $\lim_{u\to0}\sum_{k\geq 1}a_k(u)$. Note that $|a_k(u)| \leq \frac{\e^{ C_f\lambda}\Gamma(k+1)}{\eta^{k+1}k!(k+1)!}$, as $u$ is bounded by 1.
Note that the series 
\[
\sum_{k \geq 1} \frac{\Gamma(k+1)\e^{ C_f\lambda}}{k!(k+1)!\eta^{k+1}} = \frac{\e^{C_f\lambda}}{\eta^2}\sum_{k\geq 0} \frac{1}{\eta^k (k+2)!} \leq \frac{\e^{C_f\lambda}\e^{\frac{1}{\eta}}}{\eta^2}
\]
converges, and consequently by the dominated convergence theorem, we have
\[
\lim_{u\to0}\sum_{k\geq 1}a_k(u) = \sum_{k\geq 1}\lim_{u\to0}a_k(u) = 0.
\]
Thus by \eqref{eq:KIA} we have
\[
\lim_{u\to 0}\frac{K(u)}{u} =I_0(y).
\] 
Therefore we get
\[
\ota\St_{\mu_{\lambda}}(z) = -\int_0^{\infty}\e^{-\lambda d_f(y)} I_0(y) \mu_w(\dd{y}) -\int_0^{\infty}\e^{-\lambda d_f(y)} \int_0^{\infty}\e^{\ota vz}e^{\lambda \phi_z^*(y,\frac{v}{\lambda})}\dd{v}\hspace{0.2cm}\mu_w(\dd{y}).
\]
To conclude the argument, we use Lemma \ref{stieltjesconvergencefact} with Theorem \ref{pap1-momenttheorem} to state that $\St_{\bA_N}(z)$ converges in probability to $\St_{\mu_{\lambda}}(z)$ for each $z\in\C^+$.
\end{proof}
We conclude with the proof of Corollary \ref{pap1-Stieltjeslargelambda}
\begin{proof}[Proof of Corollary \ref{pap1-Stieltjeslargelambda}]
    From Corollary \ref{pap1-Stieltjescorollary}, we have 
    \begin{align*}
    \St_{\mu_{\lambda}}(z) = \ota\int_0^{\infty}\int_0^{\infty}\e^{\ota vz}\e^{-\lambda d_f(y)+\lambda\phi^*(y,v/\lambda)}\dd{v}\hspace{0.2cm}\mu_w(\dd{y}).
    \end{align*}
    Recall that 
    \begin{align}\label{pap1-eq:phi*}
    \phi^*(x,u) = d_f(x) - \int_0^{\infty}f(x,y)\e^{-\lambda d_f(y)}\left(\Khorunzkyintegral\e^{\ota vz}\e^{\lambda \phi^*\left(y,\frac{v}{\lambda}\right)} \dd{v}\right)\mu_w(\dd{y})
   \end{align}
    is the unique analytical solution of the fixed point equation as in \eqref{pap1-contractionmap}.
    Expanding the Bessel function $\J(x)$ in \eqref{pap1-eq:phi*} using \eqref{eq:bessel} gives 
    \begin{align}\label{pap1-eq:besselexpandedphi*}
    \phi^*(x,u) = d_f(x) - \int_0^{\infty}f(x,y)\e^{-\lambda d_f(y)}\left(\int_0^{\infty}\sum_{k\geq 0}\frac{(-1)^ku^{k+1}v^k}{k!(k+1)!} \e^{\ota vz}\e^{\lambda \phi^*\left(y,\frac{v}{\lambda}\right)} \dd{v}\right)\mu_w(\dd{y}).
    \end{align}
    We would like to interchange the summand and integral with respect to $v$ in \eqref{pap1-eq:besselexpandedphi*}. Using the $z=
    \zeta+\ota \eta$ for some $
    \zeta\in \mathbb R$ and $\eta>0$, we have that 
    \[
    \sum_{k\geq 0}\int_0^{\infty} \left| \frac{(-1)^ku^{k+1}v^k}{k!(k+1)!}\e^{\ota vz}\e^{-\lambda d_f(y)+\lambda\phi^*(y,v/\lambda)} \right|\dd{v} \leq  \e^{C_f\lambda - \lambda d_f(y)}\sum_{k\geq 0}\frac{u^{k+1}\Gamma(k+1)}{k!(k+1)!\eta^{k+1}} \leq \frac{u}{\eta}\e^{C_f \lambda - \lambda d_f(y)}\e^{u/\eta}.
    \]
    Thus, by Fubini's Theorem, we can interchange the summand with the integral with respect to $v$, giving us 
    \begin{align}\label{pap1-eq:phi*afterFubini}
        \phi^*(x,u) = d_f(x) - \int_0^{\infty}f(x,y)\e^{-\lambda d_f(y)}\left(\sum_{k\geq 0}\frac{(-1)^ku^{k+1}}{k!(k+1)!} \int_0^{\infty}v^k\e^{\ota vz}\e^{\lambda \phi^*\left(y,\frac{v}{\lambda}\right)} \dd{v}\right)\mu_w(\dd{y}).
    \end{align}

    Now, denote by $\mathcal{H}^{\lambda}(z,y)$ the function 
    \begin{align}\label{pap1-eq:Hz(y)}
    \mathcal{H}^{\lambda}(z,y) := \ota\int_0^{\infty}\e^{\ota vz}\e^{-\lambda d_f(y)+\lambda\phi^*(y,v/\lambda)}\dd{v}. 
    \end{align}
    Then, by Corollary \ref{pap1-Stieltjescorollary}, we can see that $\St_{\mu_{\lambda}}(z) = \int_0^{\infty}\mathcal{H}^{\lambda}(z,y)\mu_w(\dd{y})$. From \eqref{pap1-eq:phi*afterFubini} we get that 
    \begin{align*}
        \phi^*(x,u) = d_f(x) &- u\int_0^{\infty} f(x,y)\int_0^{\infty} \e^{\ota vz}\e^{-\lambda d_f(y)+\lambda\phi^*(y,v/\lambda)}\dd{v}\hspace{0.2cm} \mu_w(\dd{y}) \\
    &- \int_0^{\infty} f(x,y) \sum_{k\geq 1}\int_0^{\infty}\frac{(-1)^k u^{k+1}v^k}{k!(k+1)!}\e^{\ota vz}\e^{-\lambda d_f(y)+\lambda\phi^*(y,v/\lambda)}\dd{v}\hspace{0.2cm} \mu_w(\dd{y}),
    \end{align*}
    and so, we can write 
    \begin{align}\label{pap1-eq:phi*expansion}
        \phi^*(x,u) &= d_f(x) + \ota u\int_0^{\infty} f(x,y)\mathcal{H}^{\lambda}(z,y)\mu_w(\dd{y}) + T(x,u,\lambda,z)
    \end{align}
    where 
    \begin{align}\label{pap1-eq:T(x,u,z)}
    T(x,u,\lambda,z):= - \int_0^{\infty}f(x,y) \sum_{k\geq 1}\int_0^{\infty}\frac{(-1)^k u^{k+1}v^k}{k!(k+1)!}\e^{\ota vz}\e^{-\lambda d_f(y)+\lambda\phi^*(y, v/\lambda)}\dd{v}\hspace{0.2cm}\mu_w(\dd{y}).
    \end{align}
    Substituting $u=v/\lambda$ for $v\in\mathbb{R}_+$ in \eqref{pap1-eq:phi*expansion} and multiplying throughout by $\lambda$, we have 
    \begin{align*}
    -\lambda d_f(x)+\lambda\phi^*(x,v/\lambda) = \ota v\int_0^{\infty} f(x,y)\mathcal{H}^{\lambda}(z,y)\mu_w(\dd{y}) + \lambda T(x,v/\lambda,\lambda,z).
    \end{align*}
    We begin by claiming the following:
    \begin{claim}\label{pap1-claimbound}
        For any $x,u\geq 0$, we have 
        \begin{align}
        |\e^{-\lambda d_f(x) + \lambda\phi^*(x,u)}| \leq 1.
    \end{align}
    \end{claim}
    Then, one can see that 
    \begin{align*}
    |T(x,v/\lambda,\lambda,z)| \leq \int_0^{\infty}f(x,y) \frac{v}{\lambda\eta}\left(\sum_{k\geq1}\frac{v^k\Gamma(k+1)}{\eta^k\lambda^k k!(k+1)!} \right) \mu_w(\dd{y}) \leq \frac{v^2}{\lambda^2\eta^2}\e^{\frac{v}{\eta\lambda}}d_f(x),
    \end{align*}
    % {\color{red} \bf This isnt correct completely. T was missing the exponential factor. Here you will need to use $|\phi^*(y, v/\lambda)|\le (1+ \sqrt{v/\lambda}) $}
    and so for each $v\in (0,\infty)$
    \begin{align*}
    \lim_{\lambda\to\infty} \lambda|T(x,v/\lambda,\lambda,z)|\to 0.
    \end{align*}
    Thus, from \eqref{pap1-eq:phi*expansion}, for  any $v$ we have
    \begin{align}\label{pap1-eq:lamtoinftybeforeintegral}
    \lim_{\lambda\to\infty}(-\lambda d_f(x)+\lambda\phi^*(x,v/\lambda)) = \ota v\lim_{\lambda\to\infty}\int_0^{\infty} f(x,y)\mathcal{H}^{\lambda}(z,y)\mu_w(\dd{y}).
    \end{align}
    What remains now is to justify Claim \ref{pap1-claimbound}, and taking the limit $\lambda\to\infty$ inside the integral in \eqref{pap1-eq:lamtoinftybeforeintegral}.

 First we consider the homogeneous case when $f \equiv 1$. Recall from Remark \ref{remark:st:homogeneous}, that due to the lack of dependency of one coordinate, we denote $\widetilde{\phi^*}(u)= \phi^*(x,v/\lambda)$ Then, 
    \[
    \widetilde{\phi^*}(u) = 1 - \Khorunzkyintegral \e^{\ota vz}\e^{-\lambda + \lambda\widetilde{\phi^*}(v/\lambda)}\dd{v}, 
    \]
    and from \eqref{pap1-eq:lamtoinftybeforeintegral} we have $\lim_{\lambda\to\infty} (-\lambda+\lambda\widetilde{\phi^*}(v/\lambda)) = \ota v\St_{\mu_f}(z)$. Moreover, from Corollary \ref{pap1-Stieltjescorollary}, we have 
    \[
    \St_{\mu_{\lambda}}(z) = \ota\int_0^{\infty}\e^{\ota vz}\e^{-\lambda + \lambda\widetilde{\phi^*}(v/\lambda)}\dd{v}.
    \]
    Since $f\equiv 1$, from \eqref{eq:phibdd} we have that $C_f = 1$ and $|\widetilde{\phi^*}| \leq 1$. Then, $|\e^{-\lambda + \lambda\widetilde{\phi^*}}| \leq 1$, justifying Claim \ref{pap1-claimbound}. Thus, the expression inside the integral is uniformly bounded by $\e^{-\eta v}$. Using dominated convergence, we can pull the limit $\lambda\to\infty$ inside the integral to obtain
    \begin{align*}
    \St_{\mu_f}(z) = \ota\int_0^{\infty}\e^{\ota vz}\e^{\ota v\St_{\mu_f}(z)}\dd{v} = -\frac{1}{z+\St_{\mu_f}(z)},
    \end{align*}
    which is precisely the Stieltjes transform of the semicircular law. 
    
    In the case of general $f$, recall from \eqref{eq:g_N} that for any $x$ and $u$, 
    \begin{align*}
    \phi^*(x,u) = \lim_{N\to\infty}\frac{1}{N}\E \left[\sum_{i=1}^Nf(x,w_i)\e^{\ota u r^N_{ii}}\right].
    \end{align*}
    Now, for any $N$, by trivially bounding the complex exponential $\e^{\ota ur_{ii}^N}$ by 1 for any $i$, we have that 
    \begin{align*}
    \left|\frac{1}{N}\E \sum_{i=1}^N f(x,w_i)\e^{\ota u r^N_{ii}}  \right| \leq \frac{1}{N}\sum_{i=1}^N|f(x,w_i)| = \frac{1}{N}\sum_{i=1}^Nf(x,w_i).
     \end{align*}
    Thus, by triangle inequality, we have that 
    \begin{align*}
    |\phi^*(x,u)| \leq \left|  \phi^*(x,u) - \E[ g_N(x,u,z)] \right| + \frac{1}{N}\sum_{i=1}^Nf(x,w_i).
    \end{align*}
    Thus, we have that 
    \begin{align}\label{pap1-eq:phi*g_Nerror}
    -\frac{\lambda}{N}\sum_{i=1}^Nf(x,w_i) + \lambda |\phi^*(x,u)| &\leq \lambda\sqrt{1+u} \frac{1}{\sqrt{1+u}} \left|  \phi^*(x,u) - \E[ g_N(x,u,z)] \right|\nonumber\\
    &\leq \lambda\sqrt{1+u} \hspace{0.1cm}\| \phi^* - \E g_N\|_{\Ba}.
    \end{align}
    Taking $N\to\infty$ on both sides in \eqref{pap1-eq:phi*g_Nerror} yields that 
    \begin{align*}
    -\lambda d_f(x) + \lambda |\phi^*(x,u)| \leq 0.
    \end{align*}
    Using this, we conclude that
    \begin{align}\label{pap1-eq:exponentialforDCT}
    \left| \e^{-\lambda d_f(x) + \lambda \phi^*(x,u)} \right| \leq \e^{-\lambda d_f(x)}\e^{\lambda |\phi^*(x,u)|} \leq 1
    \end{align}
    for any $x$ and $u$, proving Claim \ref{pap1-claimbound}. Now, to evaluate $\lim_{\lambda\to\infty}\St_{\mu_{\lambda}}(z)$, we take the limit inside the integral in the RHS of \eqref{eq:stmain} using DCT, which we can use from \eqref{pap1-eq:exponentialforDCT}. This gives us 
    \begin{align*}
\St_{\mu_f}(z)  = \lim_{\lambda\to\infty}\St_{\mu_\lambda}(z) = \ota\int_0^{\infty}\int_0^{\infty}\e^{\ota vz}\lim_{\lambda\to\infty}\left(\e^{-\lambda d_f(y)+\lambda\phi^*(y,v/\lambda)}\right)\dd{v}\hspace{0.2cm}\mu_w{(\dd{y})}.
    \end{align*}
    and so, using \eqref{pap1-eq:lamtoinftybeforeintegral}, we get
    \begin{align}\label{pap1-eq:limitingstieltjespenultimate}
    \St_{\mu_f}(z) = \ota \int_0^{\infty}\int_0^{\infty}\e^{\ota vz} \lim_{\lambda\to\infty}\e^{\ota v \int_0^{\infty} f(x,y)\mathcal{H}^{\lambda}(z,x)\mu_w(\dd{x})}\dd{v}\hspace{0.2cm}\mu_w{(\dd{y})}.
    \end{align}
    Recall from \eqref{pap1-eq:Hz(y)} that 
    \begin{align*}
    \mathcal{H}^{\lambda}(z,y) = \ota\int_0^{\infty}\e^{\ota vz}\e^{-\lambda d_f(y)+\lambda\phi^*(y,v/\lambda)}\dd{v}. 
    \end{align*}
    Again using \eqref{pap1-eq:exponentialforDCT}, we have that the integral is bounded in absolute value, and so, using DCT allows us to define 
    \begin{align*}
    \mathcal{H}(z,y) := \lim_{\lambda\to\infty}\mathcal{H}^{\lambda}(z,y)
    \end{align*}
    where $\int_0^{\infty}\mathcal{H}(z,y)\mu_w(\dd{y}) = \St_{\mu_f}(z) $. Moreover, since $|\mathcal{H}^{\lambda}(z,y)|$ is bounded by a constant, and $\mu_w$ is a probability measure, we use DCT once again to take the limit $\lambda\to\infty$ inside $\int_0^{\infty} f(x,y)\mathcal{H}^{\lambda}(z,x)\mu_w(\dd{x})$. Thus, we obtain 
    \[
    \St_{\mu_{\infty}}(z) = \ota \int_0^{\infty}\int_0^{\infty}\e^{\ota vz} \e^{\ota v \int_0^{\infty} f(x,y)\mathcal{H}(z,x)\mu_w(\dd{x})}\dd{v}\hspace{0.2cm}\mu_w(\dd{y}) = -\int_0^{\infty} \frac{\mu_w(\dd{y})}{z+\int_0^{\infty}f(x,y)\mathcal{H}(z,x)\mu_w(\dd{x})}.
    \]
    The proof follows by observing that $\mathcal H(z,x)$ satisfies the analytic equation defined in \eqref{eq:hzx}.
    \end{proof}
    % \

\begin{wrapfigure}{l}{0.07\textwidth} 
  \vspace{-\intextsep}   \includegraphics[width=0.07\textwidth]{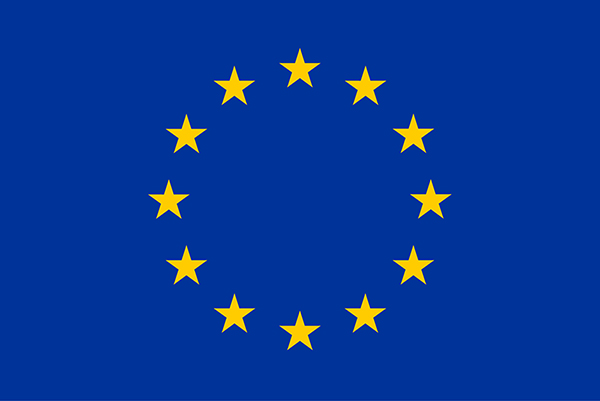}
\end{wrapfigure}
\noindent The work of L.A., R.S.H. and N.M. is supported in part by the Netherlands Organisation for Scientific Research (NWO) through the Gravitation NETWORKS grant 024.002.003. The work of N.M. is further supported by the European Union’s Horizon 2020 research and innovation programme under the Marie Skłodowska-Curie grant agreement no. 945045.

\bibliographystyle{abbrvnat}
\bibliography{reference.bib}

\section{\bf Appendix} \label{pap1-appendix}
\noindent \begin{proposition}[Banach Space]
Let $X=[0,\infty)^2$ and consider the space $\Ba$ defined by 
\[
\Ba = \left\{ \phi:X\to \C \hspace{0.2cm} \text{analytic } \middle| \hspace{0.2cm} \sup_{x,y\geq 0}\frac{|\phi(x,y)|}{\sqrt{1+y}} <\infty  \right\}
\]
and consider the norm 
\[
\|\phi\|_{\Ba} = \sup_{x,y\geq 0} \frac{|\phi(x,y)|}{\sqrt{1+y}}.
\]
Then, $(\Ba,\|\cdot\|_{\Ba})$ is a Banach space.\end{proposition}\label{pap1-Banachspace} 

\noindent \begin{proof}[Proof of Proposition \ref{pap1-Banachspace}]

For ease of notation, throughout this argument, $\|\cdot\| := \|\cdot\|_{\Ba}$. 
Clearly $\|\cdot\|$ is a norm, and thus, $(\Ba,\|\cdot\|_{\Ba})$ is a normed vector space. 

Let $\{\phi_n\}_n$ be a Cauchy sequence in $(\Ba,\|\cdot\|_{\Ba})$. Thus, for all $\epsilon>0$, there is an $N_{\varepsilon} \in \mathbb{N}$ such that for all $m,n > N_{\varepsilon}$, 
\[
\|\phi_m - \phi_n\| < \varepsilon.
\]
Let $\mu$ be the Lebesgue measure on $X$. Define 
\begin{align*}
E_{mn} = \{(x,y) \in X : |\phi_n(x,y) - \phi_m(x,y)| > \|\phi_n - \phi_m\|\sqrt{1+y}\}.
\end{align*} 
Then, $\mu(E_{mn}) = 0$. Let $E = \bigcup\limits_{m,n}E_{mn}$ and $F=E^c$. Then, $\mu(E)=0$, and 
\[
F = \{(x,y)\in X :  |\phi_n(x,y) - \phi_m(x,y)| \leq \|\phi_n - \phi_m\|\sqrt{1+y} \}.
\]
So, for all $\varepsilon>0$, we have an $N_{\varepsilon}$ such that for all $(x,y)\in F$ and $m,n>N_{\varepsilon}$,
\[
|\phi_n(x,y) - \phi_m(x,y)| < \varepsilon\sqrt{1+y}.
\]
Let $\psi_m(x,y) := \frac{\phi_m(x,y)}{\sqrt{1+y}}$. Then, we have for all $(x,y)\in F$ and $m,n>N_{\varepsilon}$
\[
|\psi_n(x,y) - \psi_m(x,y)| < \varepsilon.
\]
In other words, for all $(x,y)\in F$, denoting $a_n=\psi_n(x,y)$ gives us that  $\{a_n\}_n$ is a Cauchy sequence in the metric space $(\C,|\cdot|)$. Since $\C$ is a complete metric space, for all $(x,y)\in F$, there exists a limit $a := \lim_n a_n $, that is, for all $(x,y)\in F$, there exists a $\psi$ such that 
\[
\psi(x,y) := \lim_{n\to\infty}\psi_n(x,y).
\]
 For $(x,y) \in E$ with $\mu(E)=0$, $\psi(x,y)=0$. This is a well-defined limit. Note that since $\phi_n$ lives in $(\Ba,\|\cdot\|_{\Ba})$, $\psi_n$ lives in $(L^{\infty}(X),\|\cdot\|_{\infty})$, and we thus conclude that 
\[
\|\psi_n - \psi_m\|_{\infty} < \varepsilon.
\]
Passing the limit through $m$, we have
\[
\|\psi_n - \psi\|_{\infty} < \varepsilon.
\]
For all $(x,y)\in X$, define 
\[
\phi(x,y) = \psi(x,y)\sqrt{1+y}.
\]
One can see that $\|\phi_n - \phi\| = \|\psi_n - \psi\|_{\infty}$. Use triangle inequality to conclude $\phi \in (\Ba,\|\cdot\|_{\Ba})$
\end{proof}
\noindent For the next theorem, we refer the reader to \citet[Theorem 16.8]{Billingsley}. 
\begin{theorem}[Interchanging derivative and integral]\label{pap1-Leibnizrulethm}
    Consider the measure space $(\Omega,\mathcal{F},\mu)$ and an open set $A\subset \mathbb{R}$. Let $f:A\times\Omega\to \C$ be such that for each $x\in A$, $\omega\mapsto f(x,\omega)$ is $\mu-$integrable, and moreover for $\mu-$a.e. $\omega$, $x\mapsto f(x,\omega)$ is continuous. Consider the function $g:A\to\C$ defined by 
    \begin{align*}
        g(x) = \int_{\Omega}f(x,\omega)\mu(\dd{\omega}).
    \end{align*}
    Suppose that for each $\omega$ the partial derivative $\partial_xf(x,\omega)$ of $f$ with respect to $x$ exists. Then, if for every $x$, there is a nonnegative $\mu-$integrable function $h_x:\Omega\to\C$ and a neighbourhood $U_x$ containing $x$ such that for all $x'\in U_x$, $|\partial_{x'}f(x',\omega)| \leq h_x(\omega)$, then, $g(x)$ is continuously differentiable and 
    \begin{align*}
        \partial_xg(x) = \int_{\Omega}\partial_xf(x,\omega)\mu(\dd{\omega}).
    \end{align*}
\end{theorem}

% \

%     \noindent For the following result, we refer the reader to \citet[Theorem 7.11]{babyrudin}.
% \begin{theorem}[Uniform convergence and continuity]\label{pap1-interchanginglimits}
% Suppose $f_n \to f$ uniformly on a set $E$ in a metric space. Let $a$ be a limit point of $E$, and suppose 
% \[
% \lim_{x\to a}f_n(x) = A_n
% \]
% for each $n$. Then, $\{A_n\}_n$ converges, and 
% \[
% \lim_{x\to a}f(x) = \lim_{n\to\infty}A_n.
% \]
% In particular, 
% \[
% \lim_{n\to\infty}\lim_{x\to a}f_n(x)=\lim_{x\to a}\lim_{n\to\infty}f_n(x).
% \]
% \end{theorem}

\end{document}